\documentclass[fullpage,11pt]{amsart}
\usepackage{hyperref,pb-diagram,amssymb,epic,eepic,verbatim,graphicx,graphics,epsfig,psfrag}
\hypersetup{colorlinks}

\usepackage{color}

\definecolor{darkred}{rgb}{0.5,0,0}
\definecolor{darkgreen}{rgb}{0,0.5,0}
\definecolor{darkblue}{rgb}{0,0,0.5}

\hypersetup{ colorlinks,
linkcolor=darkblue,
filecolor=darkgreen,
urlcolor=darkred,
citecolor=darkblue }

\newcommand\redout{\bgroup\markoverwith
{\textcolor{red}{\rule[.5ex]{2pt}{0.4pt}}}\ULon}
\usepackage{xcolor}

\theoremstyle{plain}
\newtheorem{theorem}{Theorem}[section]
\newtheorem{lemma}[theorem]{Lemma}
\newtheorem{corollary}[theorem]{Corollary}
\newtheorem{proposition}[theorem]{Proposition}
\newtheorem{definition}[theorem]{Definition}

\theoremstyle{remark}
\newtheorem{remark}[theorem]{Remark}

\newcommand\cW{\mathcal{W}}
\newcommand\M{\mathcal{M}}

\renewcommand\M{\mathcal{M}}

\newcommand{\XX}{\mathcal{X}}
\newcommand{\DD}{\mathcal{D}}

\newcommand{\R}{\mathbb{R}}

\newcommand{\C}{\mathbb{C}}
\newcommand{\cC}{\mathcal{C}}

\newcommand{\Z}{\mathbb{Z}}
\newcommand{\Q}{\mathbb{Q}}

\renewcommand{\P}{\mathbb{P}}
\newcommand{\bA}{\mathbb{A}}

\newcommand\lie[1]{\mathfrak{#1}}

\newcommand{\g}{\lie{g}}

\newcommand{\q}{\lie{q}}
\renewcommand{\l}{\lie{l}}

\newcommand{\z}{\lie{z}}
\renewcommand{\t}{\lie{t}}

\renewcommand{\u}{\lie{u}}

\newcommand{\on}{\operatorname}

\newcommand{\st}{\on{st}}

\newcommand{\quot}{\on{quot}}

\newcommand{\dual}{\vee}

\newcommand{\Edge}{\on{Edge}}

\newcommand{\Ver}{\on{Vert}}

\newcommand{\Proj}{\on{Proj}}

\newcommand{\Aut}{ \on{Aut} }

\newcommand{\Ad}{ \on{Ad} }

\newcommand{\Hom}{ \on{Hom}}

\newcommand{\Spec}{\on{Spec}}

\newcommand{\ssm}{-}

\newcommand\dirac{/\kern-1.2ex\partial} 
\newcommand\qu{/\kern-.7ex/} 
\newcommand\lqu{\backslash \kern-.7ex \backslash} 

\newcommand\dr{r_+ \kern-.7ex - \kern-.7ex r_-}
 
\newcommand{\lev}{{\on{lev}}} 

\usepackage{amsmath, amsfonts,amssymb, latexsym, mathrsfs}
\newtheorem{example}[theorem]{Example}

\newcommand{\labell}\label

\newcommand{\ra}{\rightarrow}

\renewcommand{\d}{{\on{d}}}
\newcommand{\ol}{\overline}

\newcommand{\lan}{\langle}
\newcommand{\ran}{\rangle}

\newcommand{\ti}{\tilde}

\newcommand\pt{\on{pt}}
\newcommand\Def{\on{Def}}
\newcommand\age{\on{age}}
\newcommand\cE{\mathcal{E}}

\renewcommand{\ss}{\on{ss}}

\newcommand{\us}{\on{us}}

\newcommand\mE{\mathcal{E}}
\newcommand\MM{\mathfrak{M}}

\newcommand\Gr{\on{Gr}}

\newcommand\Sym{\on{Sym}}

\newcommand\rank{\on{rank}}
\newcommand\ev{\on{ev}}

\newcommand\ul{\underline}
\newcommand\mO{\mathcal{O}}

\newcommand\bra[1]{ < \kern-.7ex {#1} \kern-.7ex >} 
\newcommand\bdefn{\begin{definition}}
\newcommand\edefn{\end{definition}}
\newcommand\bea{\begin{eqnarray*}}
\newcommand\eea{\end{eqnarray*}}
\newcommand\bcv{\left[ \begin{array}{r} }
\newcommand\ecv{\end{array} \right] }

\newcommand\bma{\left[ \begin{array} }
\newcommand\ema{\end{array} \right]}
\newcommand\ben{\begin{enumerate}}
\newcommand\een{\end{enumerate}}
\newcommand\beq{\begin{equation}}
\newcommand\eeq{\end{equation}}
\newcommand\bex{\begin{example}}
\newcommand\bsj{\left\{ \begin{array}{rrr} }
\newcommand\esj{\end{array} \right\}}

\newcommand\cI{\mathcal{I}}

\newcommand\eex{\end{example}}

\newcommand\sx{*\kern-.5ex_X}

\newcommand{\fr}{{\on{fr}}}

\newcommand{\cT}{{\mathcal{T}}}

\newcommand\lefttwoarrow{%
        \mathrel{\vcenter{\mathsurround0pt
                \ialign{##\crcr
                        \noalign{\nointerlineskip}$\leftarrow$\crcr
                        \noalign{\nointerlineskip}$\leftarrow$\crcr
                }%
        }}%
}

\def\mathunderaccent#1{\let\theaccent#1\mathpalette\putaccentunder}
\def\putaccentunder#1#2{\oalign{$#1#2$\crcr\hidewidth \vbox
to.2ex{\hbox{$#1\theaccent{}$}\vss}\hidewidth}}

\author{Eduardo Gonz\'alez} \thanks{Partially supported by NSF grants
  DMS-1207194 and DMS-1510518}

\address{
Department of Mathematics,
University of Massachusetts Boston,
100 William T. Morrissey Boulevard,
Boston, MA 02125, U.S.A.}
\email{eduardo@math.umb.edu}

\author{Pablo Solis}
\address{Department of Mathematics,
California Institute of Technology, 1200 East California
Boulevard, Pasadena, CA 91125, U.S.A.}

\email{pablos@caltech.edu}

\author{Chris T. Woodward}
\address{Mathematics-Hill Center,
Rutgers University, 110 Frelinghuysen Road, Piscataway, NJ 08854-8019,
U.S.A.}  \email{ctw@math.rutgers.edu}

\begin{document}

\title[Properness for scaled gauged maps]{Properness for scaled gauged
  maps}

\maketitle

\begin{abstract}  
  We prove properness of moduli stacks of gauged maps satisfying a
  stability condition introduced by Mundet \cite{mund:corr}, Schmitt
  \cite{schmitt:univ} and Ziltener \cite{zilt:qk}. The proof combines
  a git construction of Schmitt \cite{schmitt:univ}, properness for
  twisted stable maps by Abramovich-Vistoli
  \cite{abramovich:compactifying}, a variation of a boundedness
  argument due to Ciocan-Fontanine-Kim-Maulik \cite{cf:st}, and a
  removal of singularities for bundles on surfaces in
  Colliot-Th\'el\`ene-Sansuc \cite{ciollot}.
\end{abstract}

\tableofcontents

\section{Introduction}

The moduli stack of maps from a curve to the stack quotient of a smooth
projective variety by the action of a complex reductive group has a
natural stability condition introduced by Mundet in \cite{mund:corr}
and investigated further in Schmitt \cite{schmitt:univ,schmitt:git};
the condition generalizes stability for bundles over a curve
introduced by Mumford, Narasimhan-Seshadri and Ramanathan
\cite{ra:th}.  In an earlier paper \cite{cross} the first and third
authors used the moduli of Mundet-stable maps to give a formula that
relates the genus zero gauged Gromov-Witten invariants and
Gromov-Witten invariants of the git quotient of a smooth projective
variety with reductive group action, termed a quantum analog of
Witten's localization theorem.  The proof of the formula depended on
the properness of the stack.  This properness was proved via
symplectic geometry and results of Ziltener \cite{zilt:qk} and Ott
\cite{ott:remov}. In this paper we give a purely algebraic proof of
properness via the valuative criterion for stacks \cite[Chapter
  7]{la:ch}.

The stability condition for maps to quotient stacks combines several
stability conditions already present in the literature, and leads to a
notion of gauged Gromov-Witten invariant.  Let $X$ be a smooth
projective $G$-variety such that the semi-stable locus is equal to the
stable locus, and $X/G$ the quotient stack.  By definition a map from a
curve $C$ to $X/G$ is a pair that consists of a bundle $P \to C$ and a
section $u$ of the associated bundle $P \times_G X \to C$.  We denote
by $\pi: X/G \to \on{pt}/G =: BG$ the projection to the classifying
space.  In case $X$ is a point, a stability condition for
$\Hom(C,X/G)$, bundles on $C$, was introduced by Ramanathan
\cite{ra:th}.  A stability condition that combines bundle and target
stability was introduced by Mundet \cite{mund:corr}.  There is a
compactified moduli stack $\ol{\M}^G_n(C,X,d)$ whose open locus
consists of Mundet semistable maps of class $d \in H_2^G(X,\Z)$ with
markings:
$$ C \to S, \quad v: C \to X/G, \quad (z_1,\ldots, z_n): S \to C^n
\ \text{distinct} .$$
The compactification uses the notion of Kontsevich stability for maps
\cite{qk1}, \cite{qk2}, \cite{qk3}.  The stack admits evaluation maps
to the quotient stack
$$ \ev: \ol{\M}^G_n(C,X,d) \to (X/G)^n,\quad
(\hat{C},P,u,\ul{z})\mapsto (z_i^* P, u \circ z_i) .$$
In addition, assuming stable=semistable there is a virtual fundamental
class constructed via the machinery of Behrend-Fantechi \cite{bf:in}.

Let $\widehat{QH}_G(X)$ denote the formal completion of $QH_G(X)$ at
$0$.  The {\em gauged Gromov-Witten trace} is the map
\begin{equation} \label{gtrace} \tau_X^G : \widehat{QH}_G(X) \to
  \Lambda_X^G, \quad \alpha \mapsto \sum_{n,d} \frac{q^d}{n!}
  \int_{\ol{\M}_n^G(C,X,d)} \ev^* (\alpha,\ldots,\alpha)
  .\end{equation}
The derivatives of the potential will be called gauged Gromov-Witten
invariants.  For toric varieties, the potential $\tau_X^G$ already
appears in Givental \cite{gi:eq} and Lian-Liu-Yau \cite{lly:mp1} under
the name of {\em quasimap potential}.\footnote{We are simplifying
  things a bit for the sake of exposition; actually the quasimap
  potentials in those papers involve an additional determinant line
  bundle in the integrals.}  In those papers (following earlier work
of Morrison-Plesser \cite{mp:si}) the gauged potential is explicitly
computed in the toric case, and questions about Gromov-Witten
invariants of toric varieties or complete intersections therein
reduced to a computation of quasimap invariants.  We wish re-prove and
extend the results of those papers in a uniform and geometric way that
extends to quantum K-theory and non-abelian quotients and does not use
any assumption such as the existence of a torus action with isolated
fixed points.  The splitting axiom for the gauged invariants is
somewhat different than the usual splitting axiom in Gromov-Witten
theory: the potential $\tau_X^G$ is a non-linear version of a
{\em trace} on the Frobenius manifold $QH_G(X)$.  Note that there are
several other notions of gauged Gromov-Witten invariants, for example,
Ciocan-Fontanine-Kim-Maulik \cite{cf:st}, Frenkel-Teleman-Tolland
\cite{toll:gw1}, as well as a growing body of work on gauged
Gromov-Witten theory with potential \cite{tianxu}, \cite{glsm}.

The gauged Gromov-Witten invariants so defined are closely related to,
but different from in general, the Gromov-Witten invariants of the
stack-theoretic geometric invariant theory quotient.  The stack of
marked maps to the git quotient
$$ v: C \to X \qu G, \quad (z_1,\ldots, z_n) \in C^n \ \text{distinct} $$
is compactified by the {\em graph space} \label{graphs}
$$ \ol{\M}_n(C, X \qu G, d) := \ol{\M}_{g,n}(C \times X \qu G,
(1,d))$$
the moduli stack of stable maps to $C \times X \qu G$ of class
$(1,d)$; in case $X \qu G$ is an orbifold the domain is allowed to
have orbifold structures at the nodes and markings as in
\cite{cr:orb}, \cite{agv:gw}.  The stack admits evaluation maps
$$ \ev: \ol{\M}_n(C,X \qu G,d) \to (\ol{\cI}_{X \qu G})^n $$
where $\ol{\cI}_{X\qu G}$ is the {\em rigidified inertia stack} of
$X\qu G$.  
The {\em graph trace} is the map
$$ \tau_{X \qu G}: \widehat{QH}_{\C^\times}(X \qu G) \to \Lambda_X^G, \quad \alpha \mapsto
\sum_{n,d} \frac{q^d}{n!}  \int_{\ol{\M}_n(C,X \qu G,d)} \ev^*
(\alpha,\ldots,\alpha) $$
and where the equivariant parameter for the $\C^\times$-action is
interpreted as a $\psi$-class at the corresponding marking.  The
relationship between the graph Gromov-Witten invariants of $X \qu G$
and Gromov-Witten invariants arising from stable maps to $X \qu G$ in
the toric case is studied in \cite{gi:eq}, \cite{lly:mp1}, and other
papers.

The goal of this paper is to construct, using only algebraic geometry, a
proper algebraic cobordism between the moduli stack of Mundet
semistable maps and the moduli stack of stable maps to the git
quotient with corrections coming from ``affine gauged maps''.  Affine
gauged maps are maps
$$ v: \P^1 \to X /G, \quad u(\infty) \in X^{\ss}/G, \quad
z_1,\ldots,z_n \in \P^1 - \{ \infty \}  \ \text{distinct}$$
where $\infty = [0,1] \in \P^1$ is the point ``at infinity'', modulo
{\em affine} automorphisms, that is, automorphisms of $\P^1$ which
preserve the standard affine structure on $\P^1 - \{ 0 \}$.  Denote by
$\ol{\M}^G_{n,1}(\bA,X)$ the compactified moduli stack of such affine
gauged maps to $X$; we use the notation $\bA$ to emphasize that the
equivalence only uses affine automorphisms of the domains.  A table
with the different kinds of stable maps to quotients stacks is
presented in Section \ref{table}.
 \label{scaledaffine} Evaluation at the markings defines a morphism
$$ \ev \times \ev_\infty: \ol{\M}^G_{n,1}(\bA,X,d) \to (X/G)^n \times
\ol{\cI}_{X \qu G} .$$
In the case $d = 0$, the moduli stack $\ol{\M}^G_{0,1}(\bA,X,d)$ is
isomorphic to $\ol{\cI}_{X \qu G}$ via evaluation at infinity. 
The {\em quantum Kirwan map} is the map
$$ \kappa_X^G: \widehat{QH}_G(X) \to QH_{\C^\times}(X \qu G) $$
defined as follows.  Let $\ev_{\infty,d}: \ol{\M}_n^G(\bA,X,d) \to
\ol{\cI}_{X \qu G}$ be evaluation at infinity restricted to affine
gauged maps, and
$$ \ev_{\infty,d,*}: H( \ol{\M}_n^G(\bA,X,d)) \otimes_\Q \Lambda_X^G \to
H_G(\ol{\cI}_{X \qu G}) \otimes_\Q \Lambda_X^G $$
push-forward using the virtual fundamental class.  The quantum Kirwan
map is
$$ \kappa_X^G: \widehat{QH}_G(X) \to QH_{\C^\times}(X \qu G), \quad \alpha \mapsto
\sum_{n,d} \frac{q^d}{n!} \ev_{\infty,d,*} \ev^* (\alpha,\ldots,
\alpha) .$$
As a formal map each term in the Taylor series of $\kappa_X^G$ is
well-defined on $QH_G(X)$, but in general the sum of terms may have
convergence issues.  The $q =0 $ specialization of $\kappa_X^G$ is the
Kirwan map to the cohomology of a git quotient studied in
\cite{ki:coh}.

The cobordism relating stable maps to the quotient with Mundet
semistable maps is itself a moduli stack of gauged maps with scaling
defined by allowing the linearization to tend towards infinity.  In
order to determine which stability condition to use, the source curves
must be equipped with additional data of a {\em scaling}: a section
$$ \delta:
\hat{C} \to \P \left(\omega_{\hat{C}/(C \times S)} \oplus
  \mO_{\hat{C}} \right) $$
of the projectivized relative dualizing sheaf.  If the section is
finite, one uses the Mundet semistability condition, while if infinite
one uses the stability condition on the target.  The possibility of
constructing a cobordism in this way was suggested by a symplectic
argument of Gaio-Salamon \cite{ga:gw}.  A {\em scaled gauged map} is a
map to the quotient stack whose domain is a curve equipped with a
section of the projectivized dualizing sheaf and a collection of
distinct markings: A datum
$$ \hat{C} \to S, \quad v:\hat{C} \to C \times X/G, \quad 
\delta: \hat{C} \to \P \left(\omega_{\hat{C}/(C \times S)} \oplus
  \mO_{\hat{C}} \right), \quad z_1,\ldots, z_n \in \hat{C} $$
where $\hat{C} \to S$ is a nodal curve of genus $g=\on{genus}{C}$, $v
= (P,u)$ is a morphism to the quotient stack $X/G$ that consists of a
principal $G$-bundle $P \to \hat{C}$ and a map $u: \hat{C} \to P
\times_G X$ of whose class projects to $[C] \in H_2(C)$, and $\delta$
is a section of the projectivization of the relative dualizing sheaf
$\omega_{\hat{C}/(C \times S)}$ satisfying certain properties.  In the
case that $X \qu G$ is an orbifold, the domain $\hat{C}$ is allowed to
have orbifold singularities at the nodes and markings and the morphism
is required to be representable.  The moduli stack of stable scaled
gauged maps $\ol{\M}^G_{n,1}(C,X,d)$ \label{scaledproj} with $n$
markings and class $d \in H_2^G(X,\Q)$ is equipped with a forgetful
map
$$ \rho: \ol{\M}^G_{n,1}(C,X,d) \to \ol{\M}_{0,1} \cong \P^1, \quad
[\hat{C}, u, \delta, \ul{z}] \mapsto \delta .$$
The fibers of $\rho$ over zero $0,\infty \in \P^1$ consist of either
Mundet semistable gauged maps, in the case $\delta = 0$, or stable
maps to the git quotient together with affine gauged maps, in the case
$\delta = \infty$: In notation,
\begin{multline} \label{fibers} \rho^{-1}(0) = \ol{\M}^G_n(C,X,d),
  \quad \rho^{-1}(\infty) = \bigcup_{d_0 + \ldots + d_r = d}
  \bigcup_{I_1 \cup \ldots \cup I_r =\{1,\ldots,n\} } \\ \left(
    \ol{\M}^{\fr}_{g,r}(C \times X \qu G,(1,d_0)) \times_{(\ol{\cI}_{X
        \qu G})^r} \prod_{j=1}^r \ol{\M}^G_{|I_j|,1} (\bA,X,d_j)
  \right) / (\C^\times)^r \end{multline}
where the superscipt $\fr$ indicates the inclusion of framings at the
tangent spaces to the markings, $(\C^\times)^r$ acts diagonally on the
framings and on the scalings, and we identify $H_2(X \qu G)$ as a
subspace of $H_2^G(X)$ via the inclusion $X \qu G \subset X / G$.  The
properness of these moduli stacks was argued via symplectic geometry
in \cite{qk2}.  The advantage of the symplectic proof is that the
compactness is somewhat more natural; it follows by a combination of
Gromov and Uhlenbeck compactness theorems as in Ott \cite{ott:remov}
and also applies in the presence of Lagrangian boundary conditions as
in Xu \cite{xu:compact}, for an arbitrary symplectic manifold.
However, in the setting of virtual fundamental classes constructed
algebraically, one prefers to stay in the framework of algebraic
geometry.  This is especially true in a subsequent paper of the first
and third authors in which we extend the results to quantum K-theory,
and in particular give presentations of the quantum K-theory ring of
toric stacks; the definition of quantum K-theory is at the moment
heavily algebraic, and there is no known definition purely in terms of
symplectic geometry.  Also, it is good to have several proofs.  In
this paper we give an algebraic proof of the following:

\begin{theorem}  \label{main} For any real $E > 0$, the union of components
$\ol{\M}_{n,1}^G(C,X,d)$, $\ol{\M}_n^G(C,X,d)$, and
  $\ol{\M}_{n,1}^G(\bA,X,d)$ with $(d, c_1^G(\ti{X})) < E$ is proper.
\end{theorem}

\noindent The proof is a combination of boundedness arguments and
valuative criteria.  By integration over the moduli stack of stable
scaled gauged maps one obtains the following identity: Let
$\tau_X^{G,k}$ denote the gauged potential of \eqref{gtrace} defined
using the polarization $\ti{X}^k$ for $k$ a positive integer.
\begin{equation} \label{qwit} \lim_{k \to \infty} \tau_X^{G,k} = \tau_{X
    \qu G} \circ \kappa_X^G .\end{equation}
This is called in \cite{ga:gw} and \cite{qk1} the {\em adiabatic limit
  theorem}.

\section{Scaled curves}

Scaled curves are curves with a section of the projectivized dualizing
sheaf incorporated, intended to give complex analogs of spaces
introduced by Stasheff \cite{st:hs} such as the multiplihedron,
cyclohedron etc.  Recall from Deligne-Mumford \cite{dm:irr} and
Behrend-Manin \cite[Definition 2.1]{bm:gw} the definition of stable
and prestable curves.  A {\em prestable curve} over the scheme $S$ is
a flat proper morphism $\pi: C \to S$ of schemes such that the
geometric fibers of $\pi$ are reduced, connected, one-dimensional and
have at most ordinary double points (nodes) as singularities.  A {\em
  marked prestable curve} over $S$ is a prestable curve $\pi: C \to S$
equipped with a tuple $\ul{z} = (z_1,\ldots,z_n): S \to C^n$ of
distinct non-singular sections.  A {\em morphism} $p : C \to D$ of
prestable curves over $S$ is an $S$-morphism of schemes, such that for
every geometric point $s$ of $S$ we have (a) if $\eta$ is the generic
point of an irreducible component of $D_s$, then the fiber of $p_s$
over $\eta$ is a finite $\eta$-scheme of degree at most one, (b) if
$C'$ is the normalization of an irreducible component of $C_s$, then
$p_s(C')$ is a single point only if $C'$ is rational.  A prestable
curve is {\em stable} if it has finitely many automorphisms.  Denote
by $\ol{\M}_{g,n}$ the proper Deligne-Mumford stack of stable curves
of genus $g$ with $n$ markings \cite{dm:irr}.  The stack
$\ol{\MM}_{g,n}$ of prestable curves of genus $g$ with $n$ markings is
an Artin stack locally of finite type \cite[Proposition 2]{be:gw}.

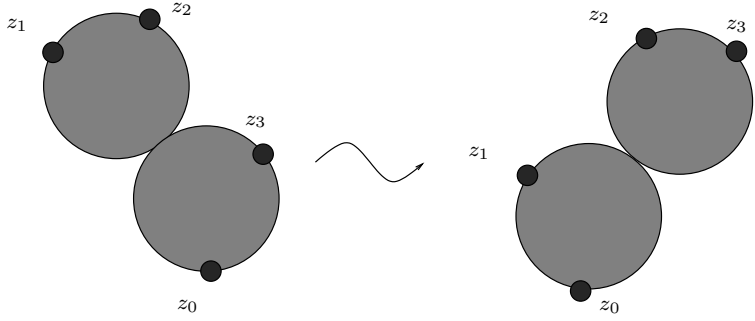
\begin{figure}[ht]
\begin{picture}(0,0)%
\includegraphics{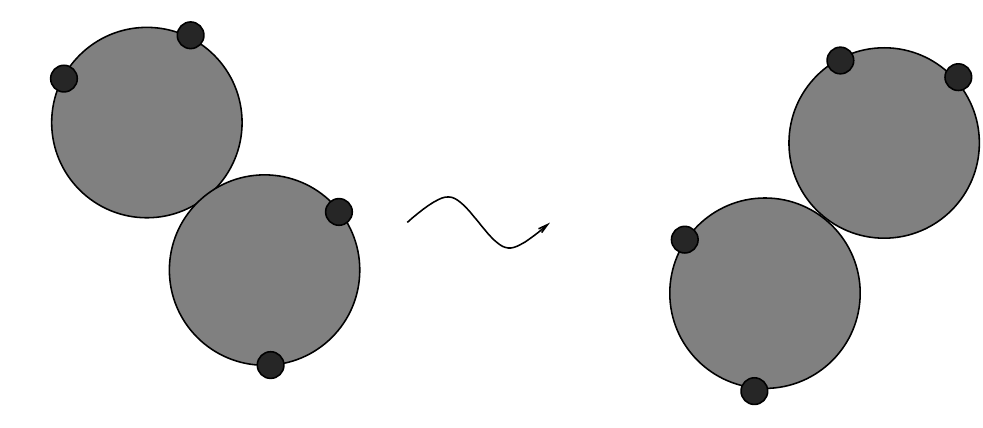}%
\end{picture}%
\setlength{\unitlength}{3947sp}%
\begingroup\makeatletter\ifx\SetFigFont\undefined%
\gdef\SetFigFont#1#2#3#4#5{%
  \reset@font\fontsize{#1}{#2pt}%
  \fontfamily{#3}\fontseries{#4}\fontshape{#5}%
  \selectfont}%
\fi\endgroup%
\begin{picture}(4709,2035)(2021,-1137)
\put(3107,-1080){\makebox(0,0)[lb]{\smash{{\SetFigFont{8}{9.6}{\rmdefault}{\mddefault}{\updefault}{\color[rgb]{0,0,0}$z_0$}%
}}}}
\put(5759,-1086){\makebox(0,0)[lb]{\smash{{\SetFigFont{8}{9.6}{\rmdefault}{\mddefault}{\updefault}{\color[rgb]{0,0,0}$z_0$}%
}}}}
\put(3535, 82){\makebox(0,0)[lb]{\smash{{\SetFigFont{8}{9.6}{\rmdefault}{\mddefault}{\updefault}{\color[rgb]{0,0,0}$z_3$}%
}}}}
\put(6555,682){\makebox(0,0)[lb]{\smash{{\SetFigFont{8}{9.6}{\rmdefault}{\mddefault}{\updefault}{\color[rgb]{0,0,0}$z_3$}%
}}}}
\put(5691,741){\makebox(0,0)[lb]{\smash{{\SetFigFont{8}{9.6}{\rmdefault}{\mddefault}{\updefault}{\color[rgb]{0,0,0}$z_2$}%
}}}}
\put(2036,686){\makebox(0,0)[lb]{\smash{{\SetFigFont{8}{9.6}{\rmdefault}{\mddefault}{\updefault}{\color[rgb]{0,0,0}$z_1$}%
}}}}
\put(3072,787){\makebox(0,0)[lb]{\smash{{\SetFigFont{8}{9.6}{\rmdefault}{\mddefault}{\updefault}{\color[rgb]{0,0,0}$z_2$}%
}}}}
\put(4937,-113){\makebox(0,0)[lb]{\smash{{\SetFigFont{8}{9.6}{\rmdefault}{\mddefault}{\updefault}{\color[rgb]{0,0,0}$z_1$}%
}}}}
\end{picture}%
\caption{Associativity divisor relation} 
\label{assoc} 
\end{figure} 

The following constructions give complex analogs of the spaces
constructed in Stasheff \cite{st:hs}.  For any family of possibly
nodal curves $C \to S$ we denote by $\omega_C$ the relative dualizing
sheaf defined for example in Arbarello-Cornalba-Griffiths
\cite[p. 97]{ar:alg2}.  Similarly for any morphism $\hat{C} \to C$ we
denote by $ \omega_{\hat{C}/C}$ the relative dualizing sheaf and
$\P(\omega_{\hat{C}/C} \oplus \mO_{\hat{C}}) \to \hat{C} $ the
projectivization.  A {\em scaling} is a section
$$ \delta: \hat{C} \to \P(\omega_{\hat{C}/C} \oplus \mO_{\hat{C}}),
\quad \P(\omega_{\hat{C}/C} \oplus \mO_{\hat{C}}) =
(\omega_{\hat{C}/C} \oplus \mO_{\hat{C}})^\times / \C^\times .$$
If $\hat{C} \to C$ is an isomorphism then $\omega_{\hat{C}/C}$ is
trivial:
$$ (\hat{C} \cong C) \implies (\P(\omega_{\hat{C}/C} \oplus
\mO_{\hat{C}}) \cong C \times \P^1) .$$
In this case a scaling $\delta$ is a section $C \to \P^1$, and
$\delta$ is required to be constant.  Thus the space of scalings on an
unmarked, irreducible curve is $\P^1$.

Scalings on nodal curves with markings are required to satisfy the
following properties.  First, $\delta$ should satisfy the {\em
  affinization} property that on any component $\hat{C}_i$ of
$\hat{C}$ on which $\delta$ is finite and non-zero, $\delta$ has no
zeroes and a single double pole.  In particular, this implies that in
the case $\hat{C} \cong C$, then $\delta$ is a constant section as in
the last paragraph, while on any component $\hat{C}_i$ of $\hat{C}$
with finite non-zero scaling which maps to a point in $C$, $\delta$
defines an affine structure on the complement of the pole.  To define
the second property, note that any morphism $\hat{C} \to C$ of class
$[C]$ defines a {\em rooted tree} whose vertices are components
$\hat{C}_i$ of $\hat{C}$, whose edges are nodes $w_j \in \hat{C}$, and
whose root vertex is the vertex corresponding to the component
$\hat{C}_0$ that maps isomorphically to $C$.  Let $\cT$ denote the set
of indices of {\em terminal} components $\hat{C}_i$ that meet only one
other component of $\hat{C}$:
$$ \cT = \{ i \ | \ \# \{ j \neq i | \hat{C}_j \cap \hat{C}_i \neq
\emptyset \} = 1 \} $$
as in Figure \ref{leaves}. The {\em bubble components} are the
components of $\hat{C}$ mapping to a point in $C$.
%
\begin{figure}[ht]
\begin{picture}(0,0)%
\includegraphics{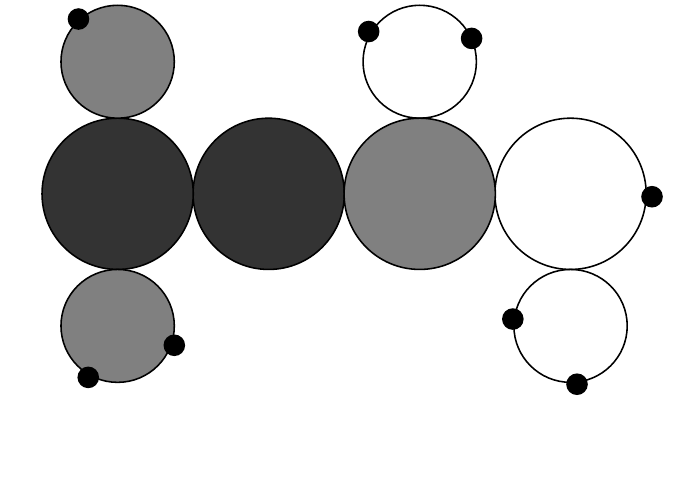}%
\end{picture}%
\setlength{\unitlength}{3947sp}%
\begingroup\makeatletter\ifx\SetFigFont\undefined%
\gdef\SetFigFont#1#2#3#4#5{%
  \reset@font\fontsize{#1}{#2pt}%
  \fontfamily{#3}\fontseries{#4}\fontshape{#5}%
  \selectfont}%
\fi\endgroup%
\begin{picture}(3263,2393)(1695,-4257)
\put(1923,-3849){\makebox(0,0)[lb]{\smash{{\SetFigFont{8}{9.6}{\rmdefault}{\mddefault}{\updefault}{\color[rgb]{0,0,0}$z_2$}%
}}}}
\put(3996,-2049){\makebox(0,0)[lb]{\smash{{\SetFigFont{8}{9.6}{\rmdefault}{\mddefault}{\updefault}{\color[rgb]{0,0,0}$z_5$}%
}}}}
\put(4943,-2849){\makebox(0,0)[lb]{\smash{{\SetFigFont{8}{9.6}{\rmdefault}{\mddefault}{\updefault}{\color[rgb]{0,0,0}$z_6$}%
}}}}
\put(2903,-2727){\makebox(0,0)[lb]{\smash{{\SetFigFont{8}{9.6}{\rmdefault}{\mddefault}{\updefault}{\color[rgb]{1,1,1}root}%
}}}}
\put(2690,-2902){\makebox(0,0)[lb]{\smash{{\SetFigFont{8}{9.6}{\rmdefault}{\mddefault}{\updefault}{\color[rgb]{1,1,1}component}%
}}}}
\put(1784,-1975){\makebox(0,0)[lb]{\smash{{\SetFigFont{8}{9.6}{\rmdefault}{\mddefault}{\updefault}{\color[rgb]{0,0,0}$z_3$}%
}}}}
\put(3210,-1975){\makebox(0,0)[lb]{\smash{{\SetFigFont{8}{9.6}{\rmdefault}{\mddefault}{\updefault}{\color[rgb]{0,0,0}$z_4$}%
}}}}
\put(3870,-3368){\makebox(0,0)[lb]{\smash{{\SetFigFont{8}{9.6}{\rmdefault}{\mddefault}{\updefault}{\color[rgb]{0,0,0}$z_8$}%
}}}}
\put(2602,-3622){\makebox(0,0)[lb]{\smash{{\SetFigFont{8}{9.6}{\rmdefault}{\mddefault}{\updefault}{\color[rgb]{0,0,0}$z_1$}%
}}}}
\put(4389,-3876){\makebox(0,0)[lb]{\smash{{\SetFigFont{8}{9.6}{\rmdefault}{\mddefault}{\updefault}{\color[rgb]{0,0,0}$z_7$}%
}}}}
\end{picture}%
\caption{A scaled marked curve}
\label{leaves}
\end{figure}
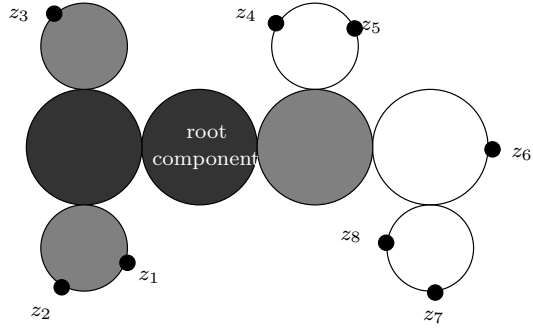 
For each terminal component $\hat{C}_i, i \in \cT$ there is a
canonical non-self-crossing path of components $\hat{C}_{i,0} =
\hat{C}_0,\ldots, \hat{C}_{i,k(i)} = \hat{C}_{i}$.  Define a partial
order on components by $\hat{C}_{i,j} \preceq \hat{C}_{i,k}$ for $j
\leq k$.  The {\em monotonicity property} requires that $\delta$ is
finite and non-zero on at most one of these (gray shaded) components,
say $\hat{C}_{i, f(i)}$, and
\begin{equation} \label{monotone} \delta |  \hat{C}_{i,j} = \begin{cases} \infty & j < f(i) \\
                                             0 & j > f(i) \end{cases}
  . \end{equation}
We call $\hat{C}_{i,f(i)}$ a {\em transition component}.  That is, the
scaling $\delta$ is infinite on the components before the transition
components and zero on the components after the transition components,
in the ordering $\preceq$.  See Figure \ref{leaves}.  In addition the
{\em marking condition } requires that the scaling is finite at the
markings:
$$\delta(z_i) < \infty,  \quad \forall i =1,\ldots, n .$$

\begin{definition} A {\em prestable scaled curve} with target a smooth
  projective curve $C$ is a morphism from a prestable map $\hat{C}$ to
  $C$ of class $[C]$ equipped with section $\delta$ and $n$ markings
  $\ul{z} = (z_1,\ldots, z_n)$ satisfying the affinization,
  monotonicity and marking properties.  Isomorphisms of prestable scaled
  curves are diagrams
$$ \begin{diagram} \node{ \hat{C}_1} \arrow{s} \arrow{e,t}{\varphi}
  \node{ \hat{C}_2 } \arrow{s} \\ \node{S_1} \arrow{e}
  \node{S_2} \end{diagram}, \quad (D\varphi^*) \varphi^*( \delta_2) =
\delta_1, \quad \varphi(z_{i,1}) = z_{i,2}, \ \ \forall i = 1,\ldots,
n $$
where the top arrow is an isomorphism of prestable curves and
$$ D\varphi^* : \varphi^* \P(\omega_{\hat{C}_2/C} \oplus \mO_{\hat{C}_2}) \to
\P(\omega_{\hat{C}_1/C} \oplus \mO_{\hat{C}_1})$$ 
is the associated morphism of projectivized relative dualizing
sheaves.  A scaled curve is {\em stable} if on each bubble component
$\hat{C}_i \subset \hat{C}$ (that is, component mapping to a point in
$C$) there are at least three special points (markings or nodes),
$$ (\delta| \hat{C}_i \in \{0, \infty\} ) \ \implies \ \# (( \{
z_i \} \cup \{ w_j \} ) \cap \hat{C}_i) \ge 3 $$
or the scaling is finite and non-zero and there are least two special points
$$ (\delta| \hat{C}_i \notin \{0, \infty\} ) \ \implies \ \# (( \{
z_i \} \cup \{ w_j \} ) \cap \hat{C}_i) \ge 2 .$$
\end{definition} 

Introduce the following notation for moduli spaces.  Let
$\ol{\MM}_{n,1}(C)$ denote the category of prestable $n$-marked scaled
curves and $\ol{\M}_{n,1}(C)$ the subcategory of stable $n$-marked
scaled curves.

The {\em combinatorial type} of a prestable marked scaled curve is
defined as follows.  Given such $(\hat{C},u: \hat{C} \to C,
\ul{z},\delta)$ Let $\Gamma$ be the graph whose vertex set
$\Ver(\Gamma)$ is the set of irreducible components of $C$, finite
edges $\Edge_{< \infty}(\Gamma)$ correspond to nodes, semi-infinite
edges $\Edge_\infty(\Gamma)$ correspond to markings, and equipped with
the labelling of semi-infinite edges by $\{ 1,\ldots , n\}$ a
distinguished {\em root vertex} $v_0 \in \Ver(\Gamma)$ corresponding
to the root component and a set of {\em transition vertices}
$\Ver^t(\Gamma) \subset \Ver(\Gamma)$ corresponding to the transition
components.  Graphically we represent a combinatorial type as a graph
with transition vertices shaded by grey, and the vertices lying on
three levels depending on whether they occur before or after the
transition vertices.  See Figure \ref{spider}.  Note that the
combinatorial type is functorial; in particular any automorphism of
prestable marked scaled curves induces an automorphism of the
corresponding type, that is, an automorphism of the graph preserving
the additional data.

We note that the graphical representation of the combinatorial type of
a curve can be viewed as the graph of a Morse/height function on the
curve. In general this gives a spider like figure with the root
component being the body of the spider. From this perspective the
paths used in the monotonicity property of scalings are the legs of
the spider.

\begin{figure}[ht]
\includegraphics[height=2in]{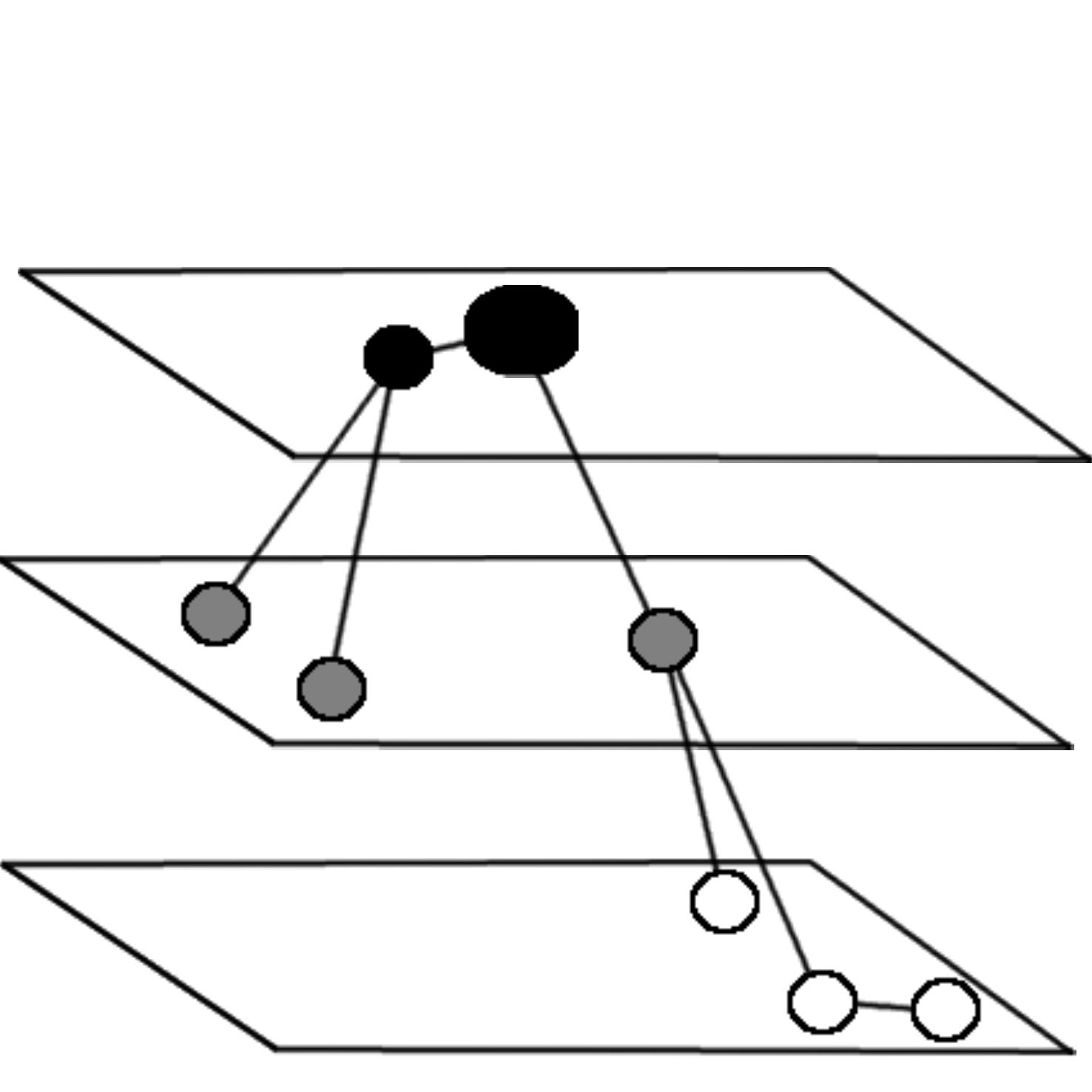}
\caption{Combinatorial type of a scaled marked curve} 
\label{spider}
\end{figure}

\begin{example}
\begin{enumerate}
\item For $n = 0$, no bubbling is possible and $\ol{\M}_{0,1}(C)$ is
  the projective line, $ \ol{\M}_{0,1}(C) \cong \P^1 .$
\item For $n = 1$, $\ol{\M}_{1,1}(C)$ consists of configurations
  $\M_{1,1}(C) \cong C \times \C$ with irreducible domain and finite
  scaling; a configurations $\ol{\M}_{1,1} \ssm \M_{1,1}$ with one
  component $\hat{C}_0 \cong C$ with infinite scaling $\delta |
  \hat{C}_0$, and another component $\hat{C}_1$ mapping trivially to
  $C$, equipped with a one-form $\delta | \hat{C}_1$ with a double
  pole at the node and a marking $z_1 \in \hat{C}_1$.  Thus $
  \ol{\M}_{1,1}(C) \cong C \times \P^1 .$
\item For $n= 2$, $\ol{\M}_{2,1}(C)$ consists of configurations
  $\M_{2,1}(C)$ with two distinct points $z_1, z_2 \in C$ and a
  scaling $\delta \in \P^1$; configurations $\M_{2,1,\Gamma_1}$ where
  the two points $z_1,z_2$ have come together and bubbled off onto a
  curve $z_1,z_2 \in \hat{C}_1$ with zero scaling $\delta |
  \hat{C}_1$, so that $\M_{2,1,\Gamma_1} \cong C \times \P^1$;
  configurations $\M_{2,1,\Gamma_2}$ with a root component $\hat{C}_0$
  with infinite scaling $\delta | \hat{C}_0$, and two components
  $\hat{C}_1,\hat{C}_2$ with non-trivial scalings $\delta | \hat{C}_1,
  \delta | \hat{C}_2$ containing markings $z_1 \in \hat{C}_1, z_2 \in
  \hat{C}_2$; a stratum $\M_{2,1,\Gamma_2}$ of configurations with a
  component $\hat{C}_1$ containing two markings $z_1,z_2 \in
  \hat{C}_1$ and $\delta | \hat{C}_1$ non-zero; a stratum
  $\M_{2,1,\Gamma_3}$ containing with three components, one
  $\hat{C}_0$ mapping isomorphically to $C$; one $\hat{C}_1$ with two
  nodes and a one form $\delta | \hat{C}_1$ with a double pole at the
  node attaching to $\hat{C}_0$; and a component $\hat{C}_2$ with two
  markings $z_1,z_2 \in \hat{C}_2$, a node, and vanishing scaling
  $\delta | \hat{C}_2$; and a stratum a stratum $\M_{2,1,\Gamma_4}$
  containing the root component $\hat{C}_0$, a component $\hat{C}_1$
  with infinite scaling with three nodes, and two components
  $\hat{C}_2, \hat{C}_3$ with finite, non-zero scaling, each
  containing a node and a marking.  The two evaluation maps at the
  markings, together with the forgetful map to $\ol{\M}_{0,1}(C)$,
  define an isomorphism $ \ol{\M}_{2,1}(C) \to C \times C \times \P^1
  .$
\end{enumerate} 
\end{example}

\begin{remark} The extension of the one-form in a family of scaled
  curves may be explicitly described as follows.  On each component of
  the limit, the one-form is determined by the limiting behavior of
  the product of deformation parameters for the nodes connecting that
  component to the root component of the limit:
  Let
$$\hat{C} \to S,\delta : \hat{C} \to \P(\omega_{\hat{C}/C \times S} \oplus
  \mO_{\hat{C}}),\ul{z}: S \to \hat{C}^n$$
  be a family of scaled curves over a 
  punctured curve $S = \ol{S} - \{ \infty \} $ 
and 
$\hat{C}_\infty$ a curve over $\infty$ extending the family $\hat{C}$.
Let $\Def(\hat{C}_\infty)/\Def_\Gamma(\hat{C}_\infty)$ denote the
deformation space of the curve $\hat{C}_\infty$ normal to the stratum
of curves of the same combinatorial type $\Gamma$ as $\hat{C}_\infty$.
This normal deformation space is canonically identified with the sum
of products of cotangent lines at the nodes
$$ \Def(\hat{C}_\infty)/\Def_\Gamma(\hat{C}_\infty) =  \sum_{w} 
T^\dual_w \hat{C}_{i_-(w)} \otimes T^\dual_w \hat{C}_{i_+(w)} $$
where $\hat{C}_{i_\pm(w)}$ are components of $\hat{C}_\infty $
adjacent to $w$, see \cite[p. 176]{ar:alg2}.  Over the deformation
space $\Def(\hat{C}_\infty)$ lives a semiversal family, universal if
the curve is stable.  Given family of curves $\hat{C} \to S$ as above
the curve $\hat{C}$ is obtained by pull-back of the semiversal family
by a map
$$ S  \to \sum_{w} T^\dual_w \hat{C}_{i_-(w)} \otimes
T^\dual_w \hat{C}_{i_+(w)}, \quad z \mapsto (\delta_w(z)) $$
describing the curves as local deformations (non-uniquely, since the
curves themselves may be only prestable.)  Let
$$ \hat{C}_0 = \hat{C}_{i,0}, \ldots, \hat{C}_{i,l(i)} := \hat{C}_i $$
denote the path of components from the root component, and
$$ w_{i,0},\ldots, w_{i,l(i)-1} \in \hat{C}_\infty $$
the corresponding sequence of nodes.  The nodes $w_{i,j}, w_{i,j+1}$
lie
%
in the same component $C_{i,j+1}$ and we
have a canonical isomorphism
$$ T_{w_{i,j}}^\dual C_{i,j+1} \cong T_{w_{i,j+1}}
C_{i,j+1} $$
corresponding to the relation of local coordinates $z_+ = 1/z_-$ near
$w_{i,j}$.  Deformation parameters for this chain lie in the space
%
\begin{multline} 
  \Hom(T_{w_{i,0}}^{\dual} \hat{C}_{i,0}, T_{w_{i,1}}^{\dual} \hat{C}_{i,1})
  \oplus \Hom(T_{w_{i,1}}^{\dual} \hat{C}_{i,1}, T_{w_{i,2}}^{\dual}
  \hat{C}_{i,2}) \ldots \\ \oplus \Hom(T_{w_{i,l(i) - 2}}^{\dual}
  \hat{C}_{i,l(i)-2}, T_{w_{i,l(i)-1}}^{\dual} \hat{C}_{i,l(i)-1})
    .\end{multline}
In particular, the product of deformation parameters
\begin{equation} 
\label{product}
\gamma_{w_{i,0}}(z) \cdots \ldots \cdot \gamma_{w_{i,l(i) -1}}(z)
\in \Hom(T_{w_{i,0}}^{\dual} \hat{C}_{i,0}, T_{w_{i,l(i)-1}}^{\dual} \hat{C}_{i,l(i)-1}) 
\end{equation}
is well-defined.  The product represents the {\em scale} at which the
bubble component $\hat{C}_i$ forms in comparison with $\hat{C}_0 =
\hat{C}_{i,0}$, that is, the ratio between the derivatives of local
coordinates on $\hat{C}_i$ and $\hat{C}_0$.  If $z$ is a point in
$\hat{C}_i$ then we also have a canonical isomorphism
$ T_z^\dual \hat{C}_i \to T_{w_{i,0}} \hat{C}_0 .$
The product \eqref{product} gives an isomorphism
$ T_z^\dual \hat{C}_i \to T_{w_0}^{\dual} \hat{C}_{0} .$
%
\begin{equation} \label{oneform} \delta | \hat{C}_i = \lim_{z \to 0}
  \delta(z) (\gamma_{w_{i,0}}(z) \cdots \ldots \cdot \gamma_{w_{i,l(i)
      -1}}(z)) \end{equation}
the ratio of the scale of the bubble component with the parameter
$\delta(z)^{-1}$.  This ends the Remark.
\end{remark}

One may view a scaled curve with infinite scaling on the root
component as a nodal curve formed from the root component and a
collection of bubble trees as follows.  

\begin{definition} \label{affine} An {\em affine prestable
  scaled curve} consists of a tuple $(C,\delta,\ul{z})$ where $C$ is a
  connected projective nodal curve, $\delta: C \to \P( \omega_C \oplus
  \mO_C)$ a section of the projectivized dualizing sheaf, and $\ul{z}
  = (z_0,\ldots,z_n)$ non-singular, distinct points, such that
\begin{enumerate} 
\item $\delta$ is monotone in the following sense: For each terminal
  component $\hat{C}_{i}, i \in \cT$ there is a canonical
  non-self-crossing path of components
$$\hat{C}_{l(i),0} = \hat{C}_0,\ldots,
  \hat{C}_{i,k(i)} = \hat{C}_i .$$  
The monotonicity condition is for any such non-self-crossing path of
components starting with a root component, that $\delta$ is finite and
non-zero on at most one of these {\em transition components}, say
$\hat{C}_{i, f(i)}$, and the scaling is infinite for all components
before the transition component and zero for components after the
transition component:
$$ \delta | \hat{C}_{i,j} = \begin{cases} \infty & j < f(i) \\ 0 & j
  > f(i) \end{cases} . $$
\item $\delta$ is infinite at $z_0$, and finite at $z_1,\ldots, z_n$.
\end{enumerate} 
A prestable affine scaled curve is {\em stable} if it has finitely
many automorphisms, or equivalently, if each component ${C}_i \subset
{C}$ has at least three special points (markings or nodes),
$$ (\delta| {C}_i \in \{0, \infty\} ) \ \implies \ \# (( \{
z_i \} \cup \{ w_j \} ) \cap {C}_i) \ge 3 $$
or the scaling is finite and non-zero and there are least two special points
$$ (\delta| {C}_i \notin \{0, \infty\} ) \ \implies \ \# (( \{
z_i \} \cup \{ w_j \} ) \cap {C}_i) \ge 2 .$$
\end{definition} 

We will see below in Theorem \ref{scaledproper} that scaled marked
curves have no automorphisms.  Examples of stable affine scaled curves
are shown in Figure \ref{affine}.  Denote the moduli stack of
prestable affine scaled curves resp. stable affine $n$-marked scaled
curves by $\ol{\MM}_{n,1}(\bA)$ resp. $\ol{\M}_{n,1}(\bA)$.

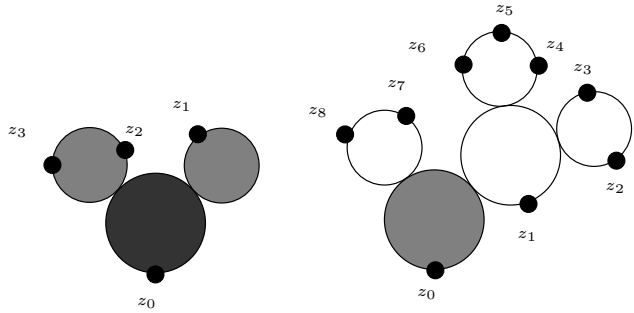
\begin{figure}[ht]
\begin{picture}(0,0)%
\includegraphics{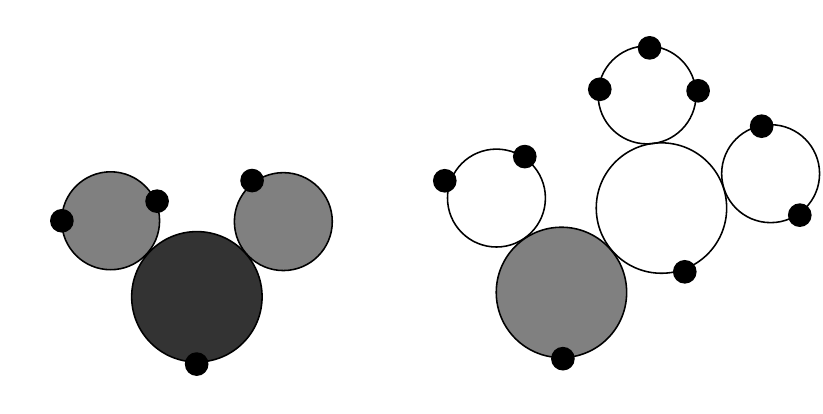}%
\end{picture}%
\setlength{\unitlength}{3947sp}%
\begingroup\makeatletter\ifx\SetFigFont\undefined%
\gdef\SetFigFont#1#2#3#4#5{%
  \reset@font\fontsize{#1}{#2pt}%
  \fontfamily{#3}\fontseries{#4}\fontshape{#5}%
  \selectfont}%
\fi\endgroup%
\begin{picture}(3942,1975)(1401,-3013)
\put(2230,-2970){\makebox(0,0)[lb]{\smash{{\SetFigFont{6}{7.2}{\rmdefault}{\mddefault}{\updefault}{\color[rgb]{0,0,0}$z_0$}%
}}}}
\put(3988,-2913){\makebox(0,0)[lb]{\smash{{\SetFigFont{6}{7.2}{\rmdefault}{\mddefault}{\updefault}{\color[rgb]{0,0,0}$z_0$}%
}}}}
\put(4619,-2548){\makebox(0,0)[lb]{\smash{{\SetFigFont{6}{7.2}{\rmdefault}{\mddefault}{\updefault}{\color[rgb]{0,0,0}$z_1$}%
}}}}
\put(5172,-2245){\makebox(0,0)[lb]{\smash{{\SetFigFont{6}{7.2}{\rmdefault}{\mddefault}{\updefault}{\color[rgb]{0,0,0}$z_2$}%
}}}}
\put(4969,-1500){\makebox(0,0)[lb]{\smash{{\SetFigFont{6}{7.2}{\rmdefault}{\mddefault}{\updefault}{\color[rgb]{0,0,0}$z_3$}%
}}}}
\put(4796,-1354){\makebox(0,0)[lb]{\smash{{\SetFigFont{6}{7.2}{\rmdefault}{\mddefault}{\updefault}{\color[rgb]{0,0,0}$z_4$}%
}}}}
\put(4478,-1129){\makebox(0,0)[lb]{\smash{{\SetFigFont{6}{7.2}{\rmdefault}{\mddefault}{\updefault}{\color[rgb]{0,0,0}$z_5$}%
}}}}
\put(3930,-1359){\makebox(0,0)[lb]{\smash{{\SetFigFont{6}{7.2}{\rmdefault}{\mddefault}{\updefault}{\color[rgb]{0,0,0}$z_6$}%
}}}}
\put(3800,-1599){\makebox(0,0)[lb]{\smash{{\SetFigFont{6}{7.2}{\rmdefault}{\mddefault}{\updefault}{\color[rgb]{0,0,0}$z_7$}%
}}}}
\put(3304,-1771){\makebox(0,0)[lb]{\smash{{\SetFigFont{6}{7.2}{\rmdefault}{\mddefault}{\updefault}{\color[rgb]{0,0,0}$z_8$}%
}}}}
\put(2449,-1734){\makebox(0,0)[lb]{\smash{{\SetFigFont{6}{7.2}{\rmdefault}{\mddefault}{\updefault}{\color[rgb]{0,0,0}$z_1$}%
}}}}
\put(2152,-1886){\makebox(0,0)[lb]{\smash{{\SetFigFont{6}{7.2}{\rmdefault}{\mddefault}{\updefault}{\color[rgb]{0,0,0}$z_2$}%
}}}}
\put(1416,-1896){\makebox(0,0)[lb]{\smash{{\SetFigFont{6}{7.2}{\rmdefault}{\mddefault}{\updefault}{\color[rgb]{0,0,0}$z_3$}%
}}}}
\end{picture}%

\caption{Examples of stable affine scaled curves}
\label{affine}
\end{figure}

\begin{theorem} \label{scaledproper} For each $n \ge 0$ and smooth
  projective curve $C$ the moduli stack $\ol{\M}_{n,1}(C)$
  resp. $\ol{\M}_{n,1}(\bA)$ of stable scaled affine curves is a
  proper scheme locally isomorphic to a product of a number of copies
  of $C$ with a toric variety.  The stack $\ol{\MM}_{n,1}(C)$
  resp. $\ol{\MM}_{n,1}(\bA)$ of prestable scaled curves is an Artin
  stack of locally finite type.
\end{theorem}

\begin{proof}  
 Standard arguments on imply that $\ol{\M}_{n,1}(C)$ and
 $\ol{\MM}_{n,1}(C)$ are stacks, that is, categories fibered in
 groupoids satisfying effective descent for objects and for which
 morphisms form a sheaf.  An object $(\hat{C},\ul{z},\delta)$ of
 $\ol{\M}_{n,1}(C)$ over a scheme $S$ is a family of curves with
 sections.  Families of curves with markings and sections satisfy the
 gluing axioms for objects; similarly morphisms are determined
 uniquely by their pull-back under a covering.  Standard results on
 hom-schemes imply that the diagonal for $\ol{\MM}_{n,1}(C)$, hence
 also $\ol{\M}_{n,1}(C)$, is representable, see for example
 \cite[1.11]{dm:irr} for similar arguments, hence the stacks
 $\ol{\MM}_{n,1}(C)$ and $\ol{\M}_{n,1}(C)$ are algebraic.

 In preparation for showing that $\ol{\M}_{n,1}(C)$ is a variety we
 claim that for any object $(\hat{C},\ul{z},\delta)$ of the moduli
 stack $\ol{\M}_{n,1}(C)$ the automorphism group is trivial.  Let
 $\Gamma$ be the combinatorial type.  The association of $\Gamma$ to
 $(\hat{C},\ul{z},\delta)$ is functorial and any automorphism of
 $(\hat{C},\ul{z},\delta)$ induces an automorphism of $\Gamma$.  The
 graph $\Gamma$ is a tree with labelled semi-infinite edges, each
 vertex is determined uniquely by the partition of semi-infinite edges
 given by removing the vertex; hence the automorphism acts trivially
 on the vertices of $\Gamma$.  Each component has at least three
 special points, or two special points and a non-trivial scaling and
 so has trivial automorphism group fixing the special points.  Thus
 the automorphism is trivial on each component of $\hat{C}$.  The
 claim follows.

 The moduli space of stable scaled curves has a canonical covering by
 varieties corresponding to the versal deformations of prestable
 curves constructed by gluing.  Suppose that $(u: \hat{C} \to C,
 \ul{z}, \delta)$ is an object of $\ol{\M}_{n,1}(C)$ of combinatorial
 type $\Gamma$.  Let $\rho: \ol{\M}_{n,1}(C) \to \ol{\M}_{0,1}(C)
 \cong \P^1$ denote the forgetful morphism.  The locus $\rho^{-1}(\C)
 \subset \ol{\M}_{n,1}(C)$ of curves with finite scaling is isomorphic
 to $\ol{\M}_n(C) \times \C$, where the last factor denotes the
 scaling.  In the case that the root component has infinite scaling,
 let $\Gamma_1,\ldots,\Gamma_k$ denote the (possibly empty)
 combinatorial types of the bubble trees attached at the special
 points.  The stratum ${\M}_{n,1,\Gamma}(C)$ is the product of $C^k$
 with moduli stacks of scaled affine curves
 ${\M}_{n_i,1,\Gamma_i}(\bA)$ for $i =1,\ldots, k$, each isomorphic to
 an affine space given by the number of markings and scalings minus
 the dimension of the automorphism group $(n_i + 1) + 1 -
 \dim(\Aut(\P^1)) = n_i - 1$ \cite{mau:mult}.  Let
$$\gamma_e \in T^\dual_{w(e)} \hat{C}_{i_-(e)} \otimes T^\dual_{w(e)}
\hat{C}_{i_+(e)}, \quad e \in \Edge_{< \infty}(\Gamma)$$
be the deformation parameters for the nodes.  A collection of
deformation parameters $ \ul{\gamma} = ( \gamma_e)_{e \in
  \Edge(\Gamma)}$ is {\em balanced} if the signed product
\begin{equation} \label{balanced} \prod_{e \in P} \gamma_e^{\pm 1} \end{equation} 
of parameters corresponding to any non-self-crossing path $P$ between
transition components is equal to $1$, where the sign is positive for
edges pointing towards the root vertex and equal to $-1$ if the edge
is oriented away from it.  Let $Z_\Gamma$ denote the set of
deformation parameters satisfying the condition \eqref{balanced}.
Then there is a morphism
$$ \M_{n,1,\Gamma}(C) \times Z_\Gamma \to \ol{\M}_{n,1}(C) $$
described as follows.  Choose local \'etale coordinates $z_e^\pm$ on
the adjacent components to each node $w_e, \in \Edge_{<
  \infty}(\Gamma)$ and glue together the components using the
identifications $z_e^+ \mapsto \gamma_e/ z_e^-$, see for example
\cite[p. 176]{ar:alg2}, \cite[2.2]{ol:logtwist}.  Set the scaling on
the root component
$$ \delta = \prod_{e \in P} \gamma_e $$
where $P$ is a path of nodes from the root component to the transition
component, independent of the choice of component by \eqref{balanced}.
This gives a family $(\hat{C},u,\delta,\ul{z})$ of stable scaled
curves over $\M_{n,1,\Gamma}(C) \times Z_\Gamma$ and hence a morphism
to $\ol{\M}_{n,1}(C)$.  The family $(\hat{C},\ul{z},u,\delta)$ defines
a universal deformation of any curve of type $\Gamma$.  Indeed,
$(\hat{C},\ul{z})$ is a versal deformation of any of its prestable
fibers by \cite{ar:alg2}, and it follows that the family
$(\hat{C},\ul{z},u)$ is a versal deformation of any of its fibers
since there is a unique extension of the stable map on the central
fiber, up to automorphism.  The equation \eqref{product} implies that
any family of stable scaled curves satisfies the balanced relation
\eqref{balanced} between the deformation parameters for any family of
marked curves with scalings.  This provides a cover of
$\ol{\M}_{n,1}(C)$ by varieties.  It follows that $\ol{\M}_{n,1}(C)$
is a variety.

The stack of prestable scaled curves $\ol{\MM}_{n,1}(C)$ is an Artin
stack of locally finite type.  Charts for the stack
$\ol{\MM}_{n,1}(C)$, as in the case of prestable curves in
\cite{be:gw}, are given by using forgetful morphisms
$\ol{\M}_{n+k,1}(C) \to \ol{\MM}_{n,1}(C)$.  Since these morphisms
admit sections locally, they provide a smooth covering of
$\ol{\MM}_{n,1}(C)$ by varieties.

We check the valuative criterion for properness for
$\ol{\M}_{n,1}(C)$.  Given a family of stable scaled marked curves
over a punctured curve $S$ with finite scaling $\delta$
  $$(\hat{C}, u: \hat{C} \to C, \ul{z},\delta) \to S = \ol{S} - \{
\infty \} $$
we wish to construct there exists an extension over $\ol{S}$.  We
consider only the case $\hat{C} \cong C \times S$; the general case is
similar.  After forgetting the scaling $\delta$ and stabilizing we
obtain a family of stable maps to $C$ of degree $[C]$,
$$ (\hat{C}^{\st}, u: \hat{C}^{\st} \to C, \ul{z}^{\st}) \to 
\ol{S} - \{ \infty \} .$$
By properness of the stack $\ol{\M}_n(C)$ of stable maps to $C$, this
family extends over the central fiber $\infty$ to give a family over
$\ol{S}$.  The section $\delta$ of $\omega_{\hat{C}^{\st}/C}$ defines
an extension over $\ol{S}$ except possibly at the nodes.  Here there
are possible irremovable singularities corresponding to the following
situation: suppose that $\hat{C}_0,\ldots \hat{C}_i$ is a chain of
components in the curve at the central fiber, with $\hat{C}_0 \cong C$
the root component.  Suppose that $\hat{C}_{i}, \hat{C}_{i+1}$
are adjacent component with $\delta$ infinite on $\hat{C}_i$ and zero
on $\hat{C}_{i+1}$.  Taking the closure of the graph of $\delta$ gives
a family $\hat{C}$ of curves over $C$ given by replacing some of the
nodes of $\hat{C}^{\st}$ with fibers of $\P(\omega_{\hat{C}^{\st}/C}
\oplus \mO_{\hat{C}^{\st}})$ over the node.  The relative cotangent
bundle of $\hat{C}$ is related to that of $\hat{C}^{\st}$ by a twist
at $D_0,D_\infty$: If $\pi: \hat{C} \to \hat{C}^{\st}$ denotes the
projection onto $\hat{C}$ then on the components of $\hat{C}$
collapsed by $\pi$ we have
$$ \omega_{\hat{C}/C} = \pi^* \omega_{\hat{C}^{\st}/C} (-D_0 
-D_\infty) $$
where $D_0,D_\infty$ are the inverse images of the sections at zero
and infinity in $\P( \omega_{\hat{C}^{\st}/C} \oplus
\mO_{\hat{C}^{\st}})$.  Abusing notation $\omega_{\hat{C}_i^{\st}/C}
(-D_0 ) = \omega_{\hat{C}_i^{\st}/C}$ resp.
$\omega_{\hat{C}_i^{\st}/C} (-D_\infty ) = \omega_{\hat{C}_i^{\st}/C}$
on components $\hat{C}_i^{\st}$ contained in $D_0$ resp. $D_\infty$.
The extension of $\delta$ to a rational section of $\pi^*
\omega_{\hat{C}^{\st}/C}$ has, by definition a zero at
$\delta^{-1}(D_0)$ and a pole at $\delta^{-1}(D_\infty)$.  Hence the
extension of $\delta$ to a section of $\pi^*
\omega_{\hat{C}^{\st}/C}(-D_0 - D_\infty)$ has no zeroes at $D_0$ and
a double pole at $D_\infty$.  This implies that $\delta$ extends
uniquely as a section of $\P( \omega_{\hat{C}/C} \oplus
\mO_{\hat{C}})$ to all of $\ol{S}$.

By the construction \eqref{oneform}, the extension of $\delta$
satisfies the monotonicity condition \eqref{monotone}.  Indeed suppose
that a component $\hat{C}_i$ is further away from a component
$\hat{C}_j$ in the path of components from the root component
$\hat{C}_0$.  Since all deformation parameters $\gamma_{w_{i,k}}(z)$ are
approaching zero, from \eqref{oneform}, at most one of the limits
$\delta | \hat{C}_i, \delta | \hat{C}_j$ can be finite, and
$$ \begin{cases} \delta | \hat{C}_i   \ \text{finite}   \ \implies \delta | \hat{C}_j \ \ \text{zero}  \\
                 \delta | \hat{C}_j   \ \text{finite}   \ \implies \delta | \hat{C}_i \ \ \text{infinite} .
\end{cases}. $$ 
The condition \eqref{monotone} follows.
\end{proof}

\section{Mumford stability}

In this section we review the relationship between the stack-theoretic
quotient and Mumford's geometric invariant theory quotient
\cite{mu:ge}.  First we introduce various Lie-theoretic notation. Let
$G$ be a connected complex reductive group with Lie algebra $\g$.
When $G$ is abelian (so a complex torus) we denote by
$$ \g_\Z = \{ D \phi(1) \in \g \ | \ \phi \in \Hom(\C^\times,G) \},
\quad \g_\Q = \g_\Z \otimes_\Z \Q$$
the {\em coweight lattice} of derivatives of one-parameter subgroups
resp. rational one-parameter subgroups.  Dually denote by
$$ \g_\Z^\dual = \{ D \chi \in \g^\dual \ | \ \chi \in
\Hom(G,\C^\times) \}, \quad \g_\Q^\dual = \g_\Z^\dual \otimes_\Z \Q$$
the {\em weight lattice} of derivatives of characters of $G$ and the set
of {\em rational weights}, respectively.  If $G$ is non-abelian then
we still denote by $\g_\Q$ the set of derivatives of rational
one-parameter subgroups.

The targets of our maps are quotient stacks defined as follows.  Let
$X$ be a smooth projective $G$-variety.  Let $X/G$ denote the quotient
stack, that is, the category fibered in groupoids whose fiber over a
scheme $S$ has objects pairs $v = (P,u)$ consisting of a principal
$G$-bundle $P \to S$ and a section $u: S \to P \times_G X$; and whose
morphisms are given by diagrams
$$ \begin{diagram} \node{P_1} \arrow{s} \arrow{e,t}{\phi} \node{P_2}
  \arrow{s} \\ \node{S_1}  \arrow{e,b}{\psi} \node{S_2} \end{diagram}, \quad \phi(X) \circ u_1 = u_2 \circ \psi $$
where $\phi(X) : P_1(X) \to P_2(X)$ denotes the map of associated
fiber bundles \cite{dm:irr},
\href{http://stacks.math.columbia.edu/tag/04UV}{Tag 04UV}
\cite{dejong:stacks}.

Mumford's geometric invariant theory quotient \cite{mu:ge} is
traditionally defined as the projective variety associated to the
graded ring of invariant sections of a linearization of the action in
the previous paragraph.  Let $\ti{X} \to X$ be a linearization, that
is, ample $G$-line bundle.  Then
$$ X \qu G := \Proj \left( \oplus_{k \ge 0} H^0(\ti{X}^k)^G \right)
.$$
Mumford \cite{mu:ge} realizes this projective variety as the quotient
of a {\em semistable locus} by an equivalence relation.  The
semistable locus consists of points $x \in X$ such that some tensor
power $\ti{X}^k, k > 0 $ of $\ti{X}$ has an invariant section
non-vanishing at $x$, while the unstable locus is the complement of
the semistable locus:
$$ X^{\ss} = \{ x \in X \ | \ \exists k > 0, \sigma \in
H^0(\ti{X}^k)^G, \quad \sigma(x) \neq 0 \}, 
\quad  X^{\us} := X - X^{\ss} .$$
A point $x \in X$ is {\em polystable} if its orbit is closed in the
semistable locus $\ol{Gx \cap X^{\ss}} = Gx \cap X^{\ss}$.  A point $x
\in X$ is {\em stable} if it is polystable and the stabilizer $G_x$ of
$x$ is finite.  In Mumford's definition the git quotient is the
quotient of the semistable locus by the {\em orbit equivalence
  relation}
$$ (x_1 \sim x_2) \iff \ol{Gx_1}\cap \ol{Gx_2} \cap X^{\ss} \neq
\emptyset. $$
Each semistable point is then orbit-equivalent to a unique polystable
point.  However, here we define the git quotient as the
stack-theoretic quotient
$$ X \qu G := X^{\ss}/G.$$ 
We shall always assume that $X^{\ss}/G$ is a Deligne-Mumford stack
(that is, the stabilizers $G_x$ are finite) in which case the coarse
moduli space of $X^{\ss}/G$ is the git quotient in Mumford's sense.
The Luna slice theorem \cite{luna:slice} implies that $X^{\ss}/G$ is
\'etale-locally the quotient of a smooth variety by a finite group,
and so has finite diagonal. By the Keel-Mori theorem \cite{km:quot},
explicitly stated in \cite[Theorem 1.1]{conrad:kl}, the morphism from
$X^{\ss}/G$ to its coarse moduli space is proper.  Since the coarse
moduli space of $X^{\ss}/G$ is projective by Mumford's construction,
it is proper, hence $X^{\ss}/G$ is proper as well.

Later we will need the following observation about the unstable locus.
As the quotient $X \qu G$ is non-empty, there exists an ample divisor
$D$ containing the unstable locus: take $D$ to be the vanishing locus
of any non-zero invariant section of $\ti{X}^k$ for some $k > 0$:
\begin{equation} \label{invsection} D = \sigma^{-1}(0), \quad \sigma \in H^0(\ti{X}^k)^G - \{ 0 \} .\end{equation} 

The Hilbert-Mumford numerical criterion \cite[Chapter 2]{mu:ge}
provides a computational tool to determine the semistable locus: A
point $x \in X$ is $G$-semistable if and only if it is
$\C^\times$-semistable for all one-parameter subgroups $\C^\times \to
G$.  Given a rational element $\lambda\in \g_\Z$ denote the
corresponding one-parameter subgroup $\C^\times \to G, \ z \mapsto
z^\lambda$.  Denote by
$$ x_\lambda := \lim_{z \to 0}  z^\lambda x $$
the limit under the one-parameter subgroup.  Let $\mu(x,\lambda) \in
\Z$ be the weight of the linearization $\ti{X}$ at $x_\lambda$ defined
by
$$ z \ti{x} = z^{\mu(x,\lambda)} \ti{x}, \quad \forall z \in
\C^\times, \ti{x} \in \ti{X}_{x_\lambda} .$$
By restricting to the case of a projective line one sees that the
point $x \in X$ is semistable if and only if $ \mu(x,\lambda) \leq 0$
for all $\lambda \in \g_\Z.$ Polystability is equivalent to
semistability and the additional condition $ \mu(x,\lambda) = 0 \iff
\mu(x,-\lambda) =0 .$ Stability is the condition that $ \mu(x,\lambda)
< 0$ for all $\lambda \in \g_\Z - \{ 0 \} .$

The Hilbert-Mumford numerical criterion \cite[Chapter 2]{mu:ge} can be applied explicitly to actions on projective spaces as follows.
Suppose that $G$ is a torus and $X = \P(V)$ the projectivization of a
vector space $V$.  Let
$\ti{X} = \mO_X(1) \otimes \C_\theta $ 
be the $G$-equivariant line bundle given by tensoring the hyperplane
bundle $\mO_X(1)$ and the one-dimensional representation $\C_\theta$
corresponding to some weight $\theta \in \g_\Z^\dual$. Recall if $p
\in X$ is represented by a line $l \subset V$ then the fiber of
$\mO_X(1) \otimes \C_\theta$ at $p$ is $l^\dual\otimes \C_\theta$. In
particular if $z^\lambda$ fixes $p$ then $z^\lambda$ scales $l$ by
some $z^{\mu(\lambda)}$, so that $z^\lambda \ti{x} =
z^{-\mu(\lambda)+\theta(\lambda)}\ti{x}$. Let $k = \dim(V)$ and
decompose $V$ into weight spaces $V_1,\ldots, V_k$ with weights
$\mu_1,\dots,\mu_k\in \g_\Z^\dual .$ Identify
$$H^2_G(X) \cong H^2_{\C^\times \times G}(V) \cong \Z \oplus
\g_\Z^\dual $$
Under this splitting the first Chern class $c_1^G(\ti{X})$ becomes
identified up to positive scalar multiple with the pair
 \begin{equation}\label{eqiv c_1}
 c_1^G(\ti{X}) \mapsto (1,\theta) \in \Z \oplus \g_\Z^\dual.
 \end{equation}
The following is essentially \cite[Proposition 2.3]{mu:ge}.

\begin{lemma} The semistable locus for the action of a torus $G$  on the projective space
$X =P(V)$ with weights $\mu_1,\ldots, \mu_k$ and linearization shifted
  by $\theta$ is $ X^{\ss} = \P(V)^{\ss} = \{ [x_1,\ldots,x_k] \in
  \P(V) \ | \ \on{hull} ( \{ \mu_i | x_i \neq 0 \}) \ni \theta \} .$ A
  point $x $ is polystable iff $\theta$ lies in the interior of the
  hull above, and stable if in addition the hull is of maximal
  dimension.
\end{lemma} 

\begin{proof}  The Hilbert-Mumford weights are computed as follows.  For
any non-zero $\lambda \in \g_\Z$, let
$$\nu(x,\lambda) := \min_i \left\{ - \mu_i(\lambda), x_i \neq 0 \right\} .$$
Then 
\begin{eqnarray*}
 z^\lambda [ x_1,\ldots, x_k ] &=& [ z^{\mu_1(\lambda)} x_1,\ldots, z^{\mu_k(\lambda)}
  x_k] \\
&=&  [ z^{\mu_1(\lambda) + \nu(x,\lambda)} x_1,\ldots,
  z^{\mu_k(\lambda) + \nu(x,\lambda)} x_k ] \end{eqnarray*}
and 
$$ (-\mu_i(\lambda) \neq \nu(x,\lambda)) \ \implies \ \left(\lim_{z \to
  0} z^{\mu_i(\lambda) + \nu(x,\lambda)} = 0 \right) .$$
Let
$$ x_\lambda := \lim_{z \to 0} z^{\lambda} x = \lim_{z \to 0}
[z^{\mu_i(\lambda)} x_i ]_{i =1}^k \in X $$
Then
$$ x_\lambda = [x_{\lambda,1},\ldots, x_{\lambda,k}],\quad
x_{\lambda,i} = \begin{cases} x_i & -\mu_i(\lambda) = \nu(x,\lambda)
  \\ 0 & \text{otherwise} \end{cases} .$$
The Hilbert-Mumford weight is therefore
\begin{equation} \label{hm} 
\mu(x,\lambda) = \nu(x,\lambda) + (\theta,\lambda) .\end{equation} 
By the Hilbert-Mumford criterion, the point $x$ is semistable if and
only if
$$ \nu(x,\lambda) := \min \{ - \mu_i(\lambda) \ | \ x_i \neq 0\} \leq
(- \theta,\lambda), \quad \forall \lambda \in \g_\Z - \{ 0 \} .$$
That is,
$$  (x \in X^{\ss}) \iff ( \theta \in \on{hull} 
\{ \mu_i \ |  \  x_i \neq 0\} ) .$$
This proves the claim about the semistable locus.  To prove the claim
about polystability, note that $\mu(x,\lambda) = 0 = \mu(x,-\lambda)$
implies that the minimum $\nu(x,\lambda)$ is also the maximum.  Thus
the only affine linear functions $\xi: \g^\dual \to \R$ which vanish
at $\theta$ are those $\xi$ that are constant on the hull of $\mu_i$
with $x_i$ nonzero.  This implies that the span of $\mu_i$ with $x_i$
non-zero contains $\theta$ in its relative interior.  The stabilizer
$G_x$ of $x$ has Lie algebra $\g_x$ the annihilator of the span of the
hull of the $\mu_i$ with $x_i \neq 0$.  So the stabilizer $G_x$ is
finite if and only if the span of $\mu_i$ with $x_i \neq 0$ is of
maximal dimension $\dim(G)$. This implies the claim on stability.
\end{proof}

We introduce the following notation.  As above $G$ is a connected
complex reductive group with maximal torus $T$ and $\g,\t$ are the Lie
algebras of $G,T$ respectively.  Fix an invariant inner product
$( \ , \ ):\g \times \g \to \C$ on $\g$ inducing an identification
$\g \to \g^\dual$. By taking a multiple of the basic inner product on
each factor we may assume that the inner product induces an
identification $\t_\Q \to \t_\Q^\dual$.  Denote by
$$\Vert \cdot \Vert: \q_\Q \to \R_{\ge 0}, \quad \Vert \xi \Vert : =
(\xi , \xi)^{1/2}$$
the norm with respect to the induced metric.

Next recall the theory of Levi decompositions of parabolic subgroups
from Borel \cite[Section 11]{bo:lag}.  A parabolic subgroup $Q$ of $G$
is one for which $G/Q$ is complete, or equivalently, containing a
maximal solvable subgroup $B \subset G$.  Any parabolic $Q$ admits a
Levi decomposition $ Q = L(Q) U(Q) $ where $L(Q)$ denote a maximal
reductive subgroup of $Q$ and $U(Q)$ is a maximal unipotent subgroup.
Let $\l(Q),\u(Q)$ denote the Lie algebras of $L(Q), U(Q)$.  Let $\g =
\t \oplus \bigoplus_{ \alpha \in R(G)} \g_\alpha$ denote the root
space decomposition of $\g$, where $R(G)$ is the set of roots.  The
Lie algebras $\l(Q),\u(Q)$ decompose into root spaces as
$$ \q = \t \oplus \bigoplus_{\alpha \in R(Q)} \g_\alpha, \quad \l(Q) =
\t \oplus \bigoplus_{\alpha \in R(Q) \cap -R(Q)} \g_\alpha, \quad
\u(Q) = \q/\l(Q) $$
where $R(Q) \subset R(G)$ is the set of roots for $\l(Q)$.  Let
$\z(Q)$ denote the center of $\l(Q)$ and
$$\z_+(Q) = \{ \xi \in \z(Q) \ | \ \alpha(\xi) \ge 0, \ \forall \alpha \in R(Q) \} $$ 
the {\em positive chamber} on which the roots of $Q$ are non-negative.
The Levi decomposition induces a homomorphism
\begin{equation} \label{piq} \pi_Q: Q \to Q/U(Q) \cong L(Q) .\end{equation}
This homomorphism has the following alternative description as a
limit.  Let $\lambda \in \z_+(Q) \cap \g_\Q$ be a positive coweight
and
$$ \phi_\lambda: \C^\times \to L(Q), \quad z \mapsto z^\lambda $$
the corresponding central one-parameter subgroup.  Then
$$ \pi_Q(g) = \lim_{z \to 0} \Ad(z^\lambda) g
.$$
In the case of the general linear group in which the parabolic
consists of block-upper-triangular matrices, this limit projects out
the off-block-diagonal terms.

The unstable locus admits a stratification by maximally destabilizing
subgroups, as in Hesselink \cite{hess:strat}, Kirwan \cite{ki:coh},
and Ness \cite{ne:st}.  The stratification reads
\begin{equation} \label{kn} X = \bigcup_{\lambda \in \cC(X)} X_\lambda, \quad X_\lambda = G
\times_{Q_\lambda} Y_\lambda, \quad Y_\lambda \mapsto Z_\lambda
\ \text{affine fibers} \end{equation}
where $Y_\lambda,Z_\lambda,Q_\lambda,\cC(X)$ are defined as
follows. For each fixed point component $\ol{Z}_\lambda$ of
$z^{\lambda}$ there exist a weight $\mu(\lambda)$ so $z^{\lambda}$
acts on $\ti{X} | Z_\lambda$ with weight $\mu(\lambda)$:
$$ z^\lambda \ti{x} = z^{\mu(\lambda)} \ti{x}, \quad \forall \ti{x}
\in \ti{X} | Z_\lambda .$$
The group $G_\lambda/\C^\times_\lambda$ acts on $\ol{Z}_\lambda$ and
we denote by $Z_\lambda \subset \ol{Z}_\lambda$ the semistable locus.
Define
\begin{equation} \label{CX} \cC(X) = \{ \lambda \in \t_+ \ | \ \exists Z_\lambda,
\ \mu(\lambda) = (\lambda, \lambda) \} \end{equation}
using the metric, where $\t_+$ is the closed positive Weyl chamber.
The variety $Y_\lambda$ is the set of points that flow to $Z_\lambda$
under $z^{\lambda}, z \to 0$:
$$ Y_\lambda = \left\{  x \in X  \ | \ \lim_{z \to 0} z^\lambda x \in Z_\lambda \right\} $$
The group $Q_\lambda$ is the parabolic of group elements that have a
limit under $\Ad(z^{\lambda})$ as $z \to 0$:
$$ Q_\lambda = \left\{ g \in G \ | \ \exists \lim_{z \to 0}
\Ad(z^\lambda) g \in G \right\} .$$
Then $Y_\lambda$ is a $Q_\lambda$-variety; and $X_\lambda$ is the
flow-out of $Y_\lambda$ under $G$.  By taking quotients we obtain a
stratification of the quotient stack by locally-closed substacks
$$ X / G = \bigcup_{\lambda \in \cC(X)} X_\lambda / G .$$
This stratification was used in Teleman \cite{te:qu} to give a formula
for the sheaf cohomology of bundles on the quotient stack.

\section{Kontsevich stability}

In this section we recall the definition of Kontsevich's moduli stacks
of stable maps \cite{ko:lo} as generalized to orbifold quotients by
Chen-Ruan \cite{cr:orb} and in the algebraic setting by
Abramovich-Graber-Vistoli \cite{agv:gw}.  Let $X$ be a smooth
projective variety.  Recall that a {\em prestable map} with target $X$
consists of a prestable curve $C \to S$, a morphism $u: C \to X$, and
a collection $z_1,\ldots,z_n : S \to C$ of distinct non-singular
points called {\em markings}.  An automorphism of a prestable map
$(C,u,\ul{z})$ is an automorphism
$$\varphi:C \to C, \quad \varphi \circ u = u, \quad \varphi(z_i) = z_i, \quad i = 1,\ldots, n .$$
A prestable map $(C,u,\ul{z})$ is {\em stable} if the number $\#
\Aut(C,u,\ul{z})$ of automorphisms is finite.  For $d \in H_2(X,\Z)$
we denote by $\ol{\M}_{g,n}(X,d)$ \label{mgn} the moduli stack of
stable maps $(C,u,\ul{z})$ of genus $g = \on{genus}(C)$ and class $d =
v_*[C]$ with $n$ markings.

The notion of stable map generalizes to orbifolds \cite{cr:orb},
\cite{agv:gw} as follows.  These definitions are needed for the
construction of the moduli stack of affine gauged maps in the case
that the git quotient is an orbifold, but not if the quotient is free.
First we recall the notion of twisted curve:

\begin{definition} \label{twistedcurve} {\rm (Twisted curves)} 
Let $S$ be a scheme.  An {\em $n$-marked twisted curve} over $S$ is a
collection of data $(f: \cC \to S, \{ {\mathcal z}_i \subset \cC
\}_{i=1}^n)$ such that
\begin{enumerate} 
\item {\rm (Coarse moduli space)} $\cC$ is a proper stack over $S$
  whose geometric fibers are connected of dimension $1$, and such that
  the coarse moduli space of $\cC$ is a nodal curve $C$ over $S$.
\item {\rm (Markings)} The ${\mathcal z}_i \subset \cC$ are closed
  substacks that are gerbes over $S$, and whose images in $C$ are
  contained in the smooth locus of the morphism $C \to S$.
\item {\rm (Automorphisms only at markings and nodes)} If $\cC^{ns}
  \subset \cC$ denotes the {\em non-special locus} given as the
  complement of the ${\mathcal z}_i$ and the singular locus of $\cC
  \to S$, then $\cC^{ns} \to C$ is an open immersion.
\item {\rm (Local form at smooth points)} If $p \to C$ is a geometric
  point mapping to a smooth point of $C$, then there exists an integer
  $r$, equal to $1$ unless $p$ is in the image of some ${\mathcal
    z}_i$, an \'etale neighborhood $\Spec(R) \to C$ of $p$ and an
  \'etale morphism $\Spec(R) \to \Spec_S(\mO_S[x])$ such that the
  pull-back $\cC \times_C \Spec(R)$ is isomorphic to $ \Spec(R[z]/z^r
  = x )/\mu_r .$
\item {\rm (Local form at nodal points)} If $p \to C$ is a geometric
  point mapping to a node of $C$, then there exists an integer $r$, an
  \'etale neighborhood $\Spec(R) \to C$ of $p$ and an \'etale morphism
  $\Spec(R) \to \Spec_S(\mO_S[x,y]/(xy - t))$ for some $t \in \mO_S$
  such that the pull-back $\cC \times_C \Spec(R)$ is isomorphic to $
  \Spec(R[z,w]/zw - t', z^r - x, w^r - y )/\mu_r $ for some $t' \in
  \mO_S$.
\end{enumerate}
\end{definition} 

Next we recall the notion of twisted stable maps.  Let $\XX$ be a
proper Deligne-Mumford stack with projective coarse moduli space $X$.
Algebraic definitions of twisted curve and twisted stable map to a
$\XX$ are given in Abramovich-Graber-Vistoli \cite{agv:gw},
Abramovich-Olsson-Vistoli \cite{aov:twisted}, and Olsson
\cite{ol:logtwist}.

\begin{definition} 
A {\em twisted stable map} from an $n$-marked twisted curve $(\pi :
\cC \to S, ( {\mathcal z}_i \subset \cC )_{i=1}^n )$ over $S$ to $\XX$
is a representable morphism of $S$-stacks $ u: \cC \to \XX $ such that
the induced morphism on coarse moduli spaces $ u_c: C \to X $ is a
stable map in the sense of Kontsevich from the $n$-pointed curve $(C,
\ul{z} = (z_1,\ldots, z_n ))$ to $X$, where $z_i$ is the image of
${\mathcal z}_i$.  The {\em homology class} of a twisted stable curve
is the homology class $u_* [ \cC_s] \in H_2(X,\Q)$ of any fiber
$\cC_s$.
\end{definition} 

\noindent Twisted stable maps naturally form a $2$-category.  Every $2$-morphism
is unique and invertible if it exists, and so this $2$-category is
naturally equivalent to a $1$-category which forms a stack over
schemes \cite{agv:gw}.

\begin{theorem} (\cite[4.2]{agv:gw}) The stack $\ol{\M}_{g,n}(\XX)$ of
  twisted stable maps from $n$-pointed genus $g$ curves into $\XX$ is
  a Deligne-Mumford stack.  If $\XX$ is proper, then for any $c > 0$
  the union of substacks $\ol{\M}_{g,n}(\XX,d)$ with homology class $d
  \in H_2(\XX,\Q)$ satisfying $(d, c_1(\ti{X}))< c$ is proper.
\end{theorem} 

The Gromov-Witten invariants takes values in the cohomology of the
{\em inertia stack}
$$ \cI_\XX := \XX \times_{\XX \times \XX} \XX $$
where both maps are the diagonal.  The objects of $\cI_\XX$ may be
identified with pairs $(x,g)$ where $x \in \XX$ and $g \in
\Aut_\XX(x)$.  For example, if $\XX = X/G$ is a global quotient by a
finite group then
$$ \cI_\XX = \bigcup_{[g] \in G/\Ad(G)} X^g/Z_g $$
where $G/\Ad(G)$ denotes the set of conjugacy classes in $X$ and $Z_g$
is the centralizer of $g$.  Let $\mu_r = \Z/r\Z$ denote the group of
$r$-th roots of unity.  The inertia stack may also be written as a hom
stack \cite[Section 3]{agv:gw}
$$ \cI_\XX = \cup_{r > 0} \cI_{\XX,r}, \quad \cI_{\XX,r} :=
\Hom^{\on{rep}}(B\mu_r, \XX) . $$
The classifying stack $B\mu_r$ is a Deligne-Mumford stack and if $\XX$
is a Deligne-Mumford stack then
$$ \ol{\cI}_\XX := \cup_{r > 0} \ol{\cI}_{\XX,r}, \quad
\ol{\cI}_{\XX,r} := \cI_{\XX/r}/ B\mu_r . $$
is the {\em rigidified inertia stack} of representable morphisms from
$B \mu_r$ to $\XX$, see \cite[Section 3]{agv:gw}.  There is a
canonical quotient cover $\pi: \cI_\XX \to \ol{\cI}_\XX$ which is
$r$-fold over $\ol{\cI}_{\XX,r}$.  Pullback acts on cohomology by an
isomorphism
$$ \pi^* H^*(\ol{\cI}_\XX,\Q) \to H^*(\cI_\XX,\Q) .$$
For the purposes of defining orbifold Gromov-Witten invariants,
$\ol{\cI}_\XX$ can be replaced by $\cI_\XX$ at the cost of additional
factors of $r$ on the $r$-twisted sectors. If $\XX = X/G$ is a global
quotient of a scheme $X$ by a finite group $G$ then
$$ \ol{\cI}_{X/G} = \coprod_{(g)} X^{g}/ (Z_g/ \lan g \ran) $$
where $\lan g \ran \subset Z_g$ is the cyclic subgroup generated by
$g$.
For example, suppose that $X$ is a polarized linearized projective
$G$-variety such that $X \qu G$ is locally free.  Then
$$ \cI_{X \qu G} = \coprod_{(g)} X^{\ss,g}/ Z_g $$
where $X^{\ss,g}$ is the fixed point set of $g \in G$ on $X^{\ss}$,
$Z_g$ is its centralizer, and the union is over all conjugacy classes,
$$ \ol{\cI}_{X \qu G} = 
\coprod_{(g)} X^{\ss,g}/ (Z_g/ \lan g \ran)  $$
where $\lan g \ran$ is the (finite) group generated by $g$.  The
moduli stack of twisted stable maps admits evaluation maps to the
rigidified inertia stack
$$ \ev: \ol{\M}_{g,n}(\XX) \to \ol{\cI}_\XX^n, 
\quad  \ol{\ev}: \ol{\M}_{g,n}(\XX) \to \ol{\cI}_\XX^n, 
 $$
where the second is obtained by composing with the involution
$\ol{\cI}_{\XX} \to \ol{\cI}_{\XX} $ induced by the map $\mu_r \to
\mu_r, \zeta \mapsto \zeta^{-1}$.  

Constructions of Behrend-Fantechi \cite{bf:in} provide the stack of stable maps
with virtual fundamental classes.  The virtual fundamental classes
$$[\ol{\M}_{g,n,\Gamma}(\XX,d)] \in H(\ol{\M}_{g,n}(\XX),\Q) $$
(where the right-hand-side denotes the singular homology of the coarse
  moduli space) satisfy the splitting axioms for morphisms of modular
  graphs similar to those in the case that $X$ is a variety.  Orbifold
  Gromov-Witten invariants are defined by virtual integration of
  pull-back classes using the evaluation maps above.  The orbifold
  Gromov-Witten invariants satisfy properties similar to those for
  usual Gromov-Witten invariants, after replacing rescaling the inner
  product on the cohomology of the inertia stack by the order of the
  stabilizer.  The definition of orbifold Gromov-Witten invariants
  leads to the definition of orbifold quantum cohomology as follows.

\begin{definition} {\rm (Orbifold quantum cohomology)}
  To each component $\XX_k$ of ${\cI}_\XX$ is assigned a rational
  number $\age(\XX_k)$ as follows.  Let $(x,g)$ be an object in
  $\XX_k$.  The element $g$ acts on $T_x \XX$ with eigenvalues
  $\alpha_1, \ldots,\alpha_n$ with $ n = \dim(\XX)$.  Let $r$ be the
  order of $g$ and define $s_j \in \{ 0,\ldots, r - 1 \}$ by $\alpha_j
  = \exp( 2\pi i s_j / r)$.  The {\em age} is defined by $ \age(\XX_k)
  = (1/r) \sum_{j=1}^n s_j .$ Let $ \Lambda_\XX \subset
  \Hom(H_2(\XX,\Q),\Q) $ denote the Novikov field of linear
  combinations of formal symbols $q^d, d \in H_2(\XX,\Q)$ where for
  each $c> 0$, only finitely many $q^d$ with $(d,c_1(\ti{X})) < c$
  have non-zero coefficient.  Denote the quantum cohomology
$$ QH(\XX) = \bigoplus QH^\bullet(\XX), \quad QH^\bullet(\XX) =
\bigoplus_{\XX_k \subset \cI_\XX} H^{\bullet + 2 \age(\XX_k)}(\XX_k)
\otimes \Lambda_\XX .$$
\end{definition}

\noindent The genus zero Gromov-Witten invariants define on $QH(\XX)$
the structure of a Frobenius manifold \cite{cr:orb}, \cite{agv:gw}.

\section{Mundet stability} 

In this section we explain the Ramanathan condition for semistability
of principal bundles \cite{ra:th} and its generalization to maps to
quotients stacks by Mundet \cite{mund:corr}, and the quot-scheme and
stable-map compactification of the moduli stacks.  

\subsection{Ramanathan stability} 

Morphisms from a curve to a quotient of a point by a reductive group
are by definition principal bundles over the curve.  Bundles have a
natural semistability condition introduced half a century ago by
Mumford, Narasimhan-Seshadri, Ramanathan and others in terms of {\em
  parabolic reductions} \cite{ra:th}.  First we explain stability for
vector bundles.  A vector bundle $E \to C$ of degree zero over a
smooth projective curve $C$ is semistable if there are no sub-bundles
of positive degree:
$$ (E \ \text{semistable} ) \quad \text{iff} \quad ( \deg(F) \leq 0, \quad \forall
F \subset E \ \text{sub-bundles}) .$$
A generalization of the notion of semistability to principal bundles
is given by Ramanathan \cite{ra:th} in terms of {\em parabolic
  reductions}.  A parabolic reduction of $P$ consists of a pair
$$Q \subset G, \quad \sigma : C \to P/Q$$
of a parabolic subgroup of $G$, that is and a section $\sigma: C \to
P/Q$.  Denote by $\sigma^* P \subset P$ the pull-back of the
$Q$-bundle $P \to P/Q$, that is, the reduction of structure group of
$P$ to $Q$ corresponding to $\sigma$.  Associated to the homomorphism
$\pi_Q$ of \eqref{piq} is an {\em associated graded} bundle $\Gr(P) :=
\sigma^*P \times_Q L(Q) \to C$ with structure group $L(Q)$.  In the
case that $P$ is the frame bundle of a vector bundle $E \to C$ of rank
$r$, that is,
$$ P = \cup_z P_z, \quad P_z = \{ (e_1,\ldots, e_r) \in E_z^r \ |
\ e_1 \wedge \ldots \wedge e_r \neq 0 \} $$
a parabolic reduction of $P$ is equivalent to a flag of
sub-vector-bundles of $E$
$$ \{0 \} \subset E_{i_1} \subset E_{i_2} \subset \ldots \subset
E_{i_l} \subset E. $$
Explicitly the parabolic reduction $\sigma^* P$ given by frames
adapted to the flag:
$$ \sigma(z) = \{ (e_1,\ldots,e_r) \in E_z^r \ | \ e_j \in E_{i_k,z},
\ \forall j \leq i_k, k = 1,\ldots, l \} .$$
Conversely, given a parabolic reduction the associated vector bundle
has a canonical filtration.

An analog of the degree of a sub-bundle for parabolic reductions is
the degree of a line bundle defined as follows.  Given $\lambda \in
\g_\Z - \{ 0 \}$ we obtain from the identification $\g \to \g^\dual$ a
rational weight $\lambda^\vee$.  Denote the corresponding characters
$ \chi_\lambda: L(Q) \to \C^\times$ and $ \chi_\lambda\circ \pi_Q: Q
\to \C^\times .$
Consider the associated line bundle over $C$ defined by 
$P(\C_{\lambda^\vee}) := \sigma^* P \times_Q \C_{\lambda^\vee} .$
The {\em Ramanathan weight} \cite{ra:th} is the degree of the line
bundle $P(\C_{\lambda^\vee}) $, that is,
%
$$ \mu_{BG}(\sigma,\lambda) := ([C], (c_1(P(\C_{\lambda^\vee})) \in \Z
.$$
The bundle $P \to C$ is {\em Ramanathan semistable} if
$$ \mu_{BG}(\sigma,\lambda) \leq 0 , \quad \forall (\sigma,\lambda) .$$
As in the case of vector bundle, it suffices to check semistability
for all reductions to {\em maximal parabolic} subgroups.  In fact, any
dominant weight may be used in the definition of
$\mu_{BG}(\sigma,\lambda)$, which shows that Ramanathan semistability
is independent of the choice of invariant inner product on the Lie
algebra and one obtains the definition given in Ramanathan
\cite{ra:th}.

\subsection{Mundet semistability} 

The Mundet semistability condition generalizes Ramanathan's condition
to morphisms from a curve to the quotient stack \cite{mund:corr},
\cite{schmitt:univ}.  Let
$$(p: P \to C, u: C \to P(X)) \in \on{Obj}(\Hom(C,X/G)) $$
be a gauged map.  Let $(\sigma,\lambda)$ consist of a parabolic
reduction $\sigma: C \to P/Q$ and a positive coweight $\lambda \in
\z_+(Q)$.  Consider the family of bundles $ P^\lambda \to S :=
\C^\times $ obtained by conjugating by $z^\lambda$.  That is, if $P$
is given as a cocycle in nonabelian cohomology with respect to a
covering $\{ U_i \to X \}$
$$ [P] = [\psi_{ij} :
(U_i \cap U_j) \to G] \ \in \ H^1(C,G) $$
then the twisted bundle is given by 
$$ [P^\lambda] = [ z^{\lambda} \psi_{ij}
z^{-\lambda}: (U_i \cap U_j) \to G ] \ \in \  H^1(C \times S,G) .$$
Define a family of sections
$$ u^\lambda: S \times C \to P^\lambda(X) $$
by multiplying $u$ by $z^\lambda, z \in \C^\times$.  This family has
an extension over $s = \infty$ called the {\em associated graded}
bundle and stable section
\begin{equation} \label{assocgrad} \Gr(P) \to C, \quad \Gr(u): \hat{C} \to \Gr(P)(X) \end{equation} 
whose bundle $\Gr(P)$ agrees with the definition of associated graded
above.  Note that the associated graded section $\Gr(u)$ exists by
compactness of the moduli space of stable maps to $\Gr(P)(X)$.  The
composition of $\Gr(u)$ with projection $\Gr(P)(X)\to C$ is a map of
degree one; hence there is a unique component $\hat{C}_0$ of $\hat{C}$
that maps isomorphically onto $C$.  The construction above is
$\C^\times$-equivariant and in particular over the central fiber $z =
0$ the group element $z^\lambda$ acts by an automorphism of $\Gr(P)$
fixing $\Gr(u)$ up to automorphism of the domain.

For each pair of a parabolic reduction and one-parameter subgroup as
above, the Mundet weight is a sum of {\em Ramanathan} and {\em
  Hilbert-Mumford} weights.  To define the Mundet weight, consider the
action of the automorphism $z^\lambda$ on the associated graded
$\Gr(P)$.  The automorphism of $X$ by $z^\lambda$ is $L(Q)$-invariant
and so defines an automorphism of the associated line bundle $\Gr(u)^*
P(\ti{X}) \to \Gr(C)$.  The weight of the action of $z^\lambda$ on the
fiber of $\Gr(u)^* P(\ti{X})$ over the root component $\hat{C}_0$
is the {\em Hilbert-Mumford weight}
$$ \mu_X (\sigma,\lambda) \in \Z, \quad z^\lambda \ti{x} =
z^{\mu_X(\sigma,\lambda)} \ti{x}, \quad \forall \ti{x} \in
(\Gr(u)|_{\hat{C}_0})^* \Gr(P) \times_G \ti{X} .$$

\begin{definition}  {\rm (Mundet stability)}  
Let $(P,u)$ be a gauged map from a smooth projective curve $C$ to the
quotient stack $X/G$.  The {\em Mundet weight} of a parabolic
reduction $\sigma$ and dominant coweight $\lambda$ is
$$ \mu(\sigma,\lambda) = \mu_{BG}(\sigma,\lambda) +
\mu_X(\sigma,\lambda) \in \Z .$$
The gauged map $(P,u)$ is Mundet {\em semistable} resp. {\em stable}
if and only if
$$ \mu(\sigma,\lambda) \leq 0, \ \text {resp.} \ < 0, \quad \forall
    (\sigma,\lambda) .$$
A pair $(\sigma,\lambda)$ such that $ \mu(\sigma,\lambda ) \ge 0 $ is
a {\em destabilizing pair}.  A pair $(P,u)$ is {\em polystable} iff
\begin{equation} \label{polystable}  \mu(\sigma,\lambda) = 0 \iff \mu(\sigma,-\lambda) = 0, \quad \forall
(\sigma,\lambda) .\end{equation}
That is, a pair $(P,u)$ is polystable if for any destabilizing pair
the opposite pair is also destabilizing. 
\end{definition}

More conceptually the semistability condition above is the
Hilbert-Mumford stability condition adapted to one-parameter subgroups
of the complexified gauge group, as explained in \cite{mund:corr}.
Semistability is independent of the choice of invariant inner product
as follows for example from the presentation of the semistable locus
in Schmitt \cite[Section 2.3]{schmitt:git}.

We introduce notation for various moduli stacks.  Let $\M^G(C,X)$
denote the moduli space of Mundet semistable pairs; in general,
$\M^G(C,X)$ is an Artin stack as follows from the git construction
given in Schmitt \cite{schmitt:univ,schmitt:git} or the more general
construction of hom stacks in Lieblich \cite[2.3.4]{lieblich:rem}.
For any $d \in H_2^G(X,\Z)$, denote by $\M^G(C,X,d)$ the moduli stack
of pairs $v = (P,u)$ with
$$v_* [C] := (\phi \times_G \on{id}_X)_* u_* [C] = d \in
H_2^G(X,\Z) $$
where $\phi:P \to EG$ is the classifying map.  

\subsection{Compactification} 

Schmitt \cite{schmitt:univ} constructs a Grothendieck-style
compactification \label{quots} of the moduli space of
Mundet-semistable obtained as follows.  Suppose $X$ is projectively
embedded in a projectivization of a representation $V$, that is $ X
\subset \P(V)$.  Any section $u: C \to P(X)$ gives rise to a line
sub-bundle
$ L := u^* (\mO_{\P(V)} (-1) \to \P(V))$
of the associated vector bundle $P \times_G V$.  From the inclusion
$\iota:L \to P(V)$ we obtain by dualizing a surjective map
$$ j: P(V^\dual) := P \times_G V^\dual \to L^\dual .$$
A {\em bundle with generalized map} in the sense of Schmitt
\cite{schmitt:git} is a pair $(P,j)$ as above where $j$ is allowed to
have base points in the sense that
$$\zeta \in C \ \text{basepoint} \ \iff ( (\rank(j_\zeta):
P(V)_\zeta^\dual \to L_\zeta^\dual) = 0) .$$
Schmitt \cite{schmitt:git} shows that the Mundet semistability
condition extends naturally to the moduli stack of bundles with
generalized map.  Furthermore, the compactified moduli space
$\ol{\M}^{\quot,G}(C,X,d)$ is projective, in particular proper.

Schmitt's construction of the moduli space of bundles with generalized
maps uses geometric invariant theory.  After twisting by a
sufficiently positive bundle we may assume that $P(V^\dual)$ is
generated by global sections.  A collection of sections $s_1,\ldots,
s_l$ generating $P(V^\dual)$ is called an {\em $l$-level structure}.
Equivalently, an $l$-level structure is a surjective morphism $ q:
\mO_C^{\oplus l} \to P(V^\dual) .$ Denote by
$\M^{G,\lev}(C,\P(V))$ \label{level} the stack of gauged maps to
$\P(V)$ with level structure.  The group $GL(l)$ acts on the stack of
$l$-level structures, with quotient
\begin{equation} \label{schmittgit} \M^{G,\lev}(C,\P(V)) / GL(l) = \M^G(C,\P(V)) .\end{equation}
Denote by $\M^{G,\lev}(C,X) \subset \M^{G,\lev}(C,\P(V))$ the substack
whose sections $u: C \to \P(V)$ have image in $P(X) \subset P(\P(V))$.
Then by restriction we obtain a quotient presentation
$$ \M^{G,\lev}(C,X) / GL(l) = \M^G(C,X) .$$
Allowing the associated quotient $P \times_G V^\dual \to P \times_G
L^\dual$ to develop base points gives a compactified moduli stack of
gauged maps with level structure $\ol{\M}^{G,\quot,\lev}(C,X)$.
Schmitt \cite{schmitt:univ, schmitt:git} shows that the stack
$\ol{\M}^{G,\quot,\lev}(C,X)$ has a canonical linearization and the
git quotient $\ol{\M}^{G,\quot,\lev}(C,X) \qu GL(l)$ defines a
compactification $\ol{\M}^{G,\quot}(C,X)$ of $\M^G(C,X)$ independent
of the choice of $l$ as long as $l$ is sufficiently large.  A version
of the quot-scheme compactification with markings is obtained by
adding tuples of points to the data.  That is,
$$ \ol{\M}^{G,\quot}_n(C,X) := \ol{\M}^{G,\quot}(C,X) \times
\ol{\M}_n(C) $$
where we recall that $\ol{\M}_n(C)$ is the moduli stack of stable maps
$p: \hat{C} \to C$ of class $[C]$ with $n$ markings and genus that of
$C$.  The orbit-equivalence relation in can be described more
naturally in terms of {\em $S$-equivalence}: Given a family
$(P_S,u_S)$ of semistable gauged maps over a scheme $S$, such that the
generic fiber is isomorphic to some fixed $(P,u)$, then we declare
$(P,u)$ to be $S$-equivalent to $(P_s,u_s)$ for any $s \in S$.  Any
equivalence class of semistable gauged maps has a unique
representative that is polystable, by the git construction in Schmitt
\cite[Remark 2.3.5.18]{schmitt:univ}.  From the construction
evaluation at the markings defines maps to the quotient stack
$$ \ol{\M}_n^{G,\quot}(C,X,d) \to (V/\C^\times)^n, \quad ((p \circ
z_i)^*L, j \circ p \circ z_i) $$
rather than to the git quotient $X^n \subset \P(V)^n$.\footnote{The
  Ciocan-Fontanine-Kim-Maulik \cite{cf:st} moduli space of {\em stable
    quotients} remedies this defect by imposing a stability condition
  at the marked points $z_1,\ldots, z_n \in C$.  The moduli stack then
  admits a morphism to $\ol{\cI}_{X \qu G}^n$ by evaluation at the
  markings.}

\begin{example}\label{toric example} {\rm (Mundet semistable maps in the toric case)}
If $G$ is a torus and $X = \P(V)$ then we can given an explicit description of Schmitt's quot-scheme compactification $\ol{\M}^{G,\quot}(C,X,d) $ of Mundet semistable maps \cite{schmitt:univ}.

Specifically let $X = \P(V)$ where $V$ is a $k$-dimensional vector space and
\begin{equation}  \label{decompose}  V = \bigoplus_{i=1}^k V_i \end{equation} 
is the decomposition of $V$ into weight spaces $V_i$ with weight
$\mu_i \in \g_\Z^\dual$.  

A point of $\M^{G}(C,X,d)$ is a pair $(P,u)$:
\[
P\to C \ \ \ \ \ \  u\colon C \to P(X),
\]
where $P$ is a $G$-bundle and $u$ is a section. We consider $u$ as a morphism $\widetilde{u} \colon L\to P(V)$ with $L\to C$ a line bundle. Via the decomposition of $V$, we can write $\widetilde{u}$ as a $k$-tuple:
\[
(\widetilde{u}_1, \dotsc, \widetilde{u}_k) \in \bigoplus_{i=1}^k H^0(P(V_i)\otimes L^\dual).
\]
The compactification $\ol{\M}^{G,\quot}(C,X,d)$ is obtained by allowing the $\widetilde{u}_i$ to have simultaneous zeros:
\[
\widetilde{u}_1^{-1}(0) \cap \dotsb \cap \widetilde{u}_k^{-1}(0) \neq \emptyset
\]
We make use of this example later one so we collect a few results about $\ol{\M}^{G,\quot}(C,X,d)$ below.

Recall \eqref{eqiv c_1} there is a projection $H^2_G(X)\to H^2(B G)=\g_\Z^\dual $ and similarly we have $H_2^G(X)\to H_2(B G) = \g_\Z$.  Associated to $v=(P,u)$ is the discrete data:
\begin{itemize}
\item[] $v_*[C]=d \in H_2^G(X,\Z)$ and its image $d(P) \in H_2(BG)$
\item[] $c_1^G(\widetilde{X}) \in H^2_G(X)$ and its image $\theta \in H^2(B G)$
\item[] $d(u):= -c_1(L) \in H^2(C,\Z) \cong \Z$.
\end{itemize}
Note $d(P)$ is the degree of $P$; that is, $d(P) = c_1(P) \in H^2(C,\g_\Z)\cong \g_\Z$. We can now state the following.

\begin{lemma}\label{torus action on  P(V)} Let $G$ be a
  torus acting on a vector space $V$. Let $V = \bigoplus_{i=1}^kV_i$
  be its decomposition into weight spaces with weights $\mu_1, \dotsc,
  \mu_k$.  
\begin{enumerate}
\item \label{one}
The Mundet semistable locus consists of pairs $(P,u)$ such
  that
\begin{equation} \label{oneeq} \on{hull} ( \{
		- d(P)^\dual + \mu_i | \ti{u}_i \neq 0 \}) \ni
		\theta. \end{equation} 
\item \label{two} let $W = \bigoplus_{i=1}^k H^0(P(V_i)\otimes L^\dual)$ and
let $W^{ss}$ consist of $(\widetilde{u}_1, \dotsc, \widetilde{u}_k)$ such that
\eqref{oneeq} holds. Then
 $ \ol{\M}^{G,\quot}(C,X,d) \cong W^{ss}/G. $
  \item \label{three} If $\widetilde{u}_i \neq 0$ then $
    (\mu_i,d(P)^\vee)+d(u) \geq 0$. If moreover \eqref{oneeq} holds
    then
$$ (\theta-d(P)^\dual,d(P)) + d(u) \geq 0. $$
  \item \label{four} $(v_*[C],c_1(P(\widetilde{X})) =
    (\theta,d(P))+d(u)$.  
\end{enumerate}
\end{lemma}

\begin{proof}   Since $G$ is abelian, $\Gr(P) = P$ for any
  pair $(\lambda,\sigma)$.  It follows that for any $\lambda \in
  \g_\Q$, the Mundet weight is
$$ \mu(\sigma,\lambda) := \{ \min_i (d(P)^\vee,\lambda) -
  \mu_i(\lambda) + \theta(\lambda), \ti{u}_i \neq 0\} .$$
Hence the semistable locus is the space of pairs $(P,u)$
where
$$ \on{hull} ( \{ - d(P)^\dual + \mu_i | \ti{u}_i \neq 0 \})
\ni \theta .$$
This proves \eqref{one}.  The description \eqref{two} follows immediately.
For \eqref{three}, if
$\widetilde{u}_i \neq 0$ then $\deg \on{div}(\widetilde{u}_i) \geq
0$. But we also have
\begin{equation}\label{threeeq}
  (\mu_i,d(P)^\dual)+d(u) = c_1(P(V_i)\otimes L^\vee) =\deg \on{div}(\widetilde{u}_i)  \geq 0 .
\end{equation}

In particular $- d(P)^\dual + \theta \in \on{hull} ( \{ \mu_i |
\ti{u}_i \neq 0 \})$.  Together with \eqref{threeeq} this shows
$(\theta+d(P)^\dual,d(P)) \geq 0$.  

To prove \eqref{four} we use that the sections $\widetilde{u}_i$ above
are homotopic to the zero section
$$ \ti{u}_0: C \to P(V)\otimes L^\dual, \quad z \mapsto (z,0) $$
and $\ti{X}$ is induced from an equivariant line bundle on $V$ with
character $\theta$ at the fixed point at zero. Therefore  we have
\begin{equation} \label{dform}
  (v_*[C],c_1(P(\ti{X}))) = (u_{0,*}[C], c_1(P(\C_\theta)\otimes L^\dual)) =
  (\theta,d(P)) + d(u) .\end{equation}
\end{proof}

For an explicit example, if $G = \C^\times $ and $V =
\C^k$ then
$$\deg(  P(V_i)\otimes L^\dual) =  \deg(P(V_i)) - \deg(L) = d(P) + d(u),
\quad i = 1,\ldots, k .$$
It follows that the moduli stack admits an isomorphism
$$ \ol{\M}^{G,\quot}(C,X,d) \cong \C^{k(d(P) + d(u) + 1),\times} /
\C^\times \cong \P^{k(d(P) + d(u) + 1) - 1}.$$
This moduli stack is substantially simpler in topology than the moduli space of stable maps to $C × X/G$, despite the dramatically more complicated stability condition. This ends the example.
\end{example}

A {\em Kontsevich-style compactification} of the stack of
Mundet-semistable gauged maps which admits evaluation maps as well as
a Behrend-Fantechi virtual fundamental class \cite{cross} is defined
as follows.  The objects in this compactification allow {\em stable
  sections}, that is, stable maps $u : \hat{C} \to P(X) $ whose
composition with $P(X) \to C$ has class $[C]$.  Thus objects of
$\ol{\M}^G_n(C,X)$ \label{mss} are triples $(P, \hat{C}, u,\ul{z})$
consisting of a $G$-bundle $P \to C$, a projective nodal curve
$(\hat{C},\ul{z})$, and a stable map $u: \hat{C} \to P \times_G X$
whose class projects to $[C] \in H_2(C,\Z)$.  Morphisms are the
obvious diagrams.  To see that this category forms an Artin stack,
note that the moduli stack of bundles $\Hom(C,BG)$ has a universal
bundle
$$U \to C \times \Hom(C,BG) .$$ 
Consider the associated $X$-bundle
$$U \times_G X \to C \times \Hom(C,BG)  .$$  
The stack $\ol{\M}_n^G(C,X)$ is a substack of the stack of stable maps to $U
\times_G X$, and is an Artin stack by e.g. Lieblich
\cite[2.3.4]{lieblich:rem}, see \cite{qk2} for more details.  Note
that hom-stacks are not in general algebraic \cite{bhatt}.

Properness of the Kontsevich-style compactification follows from a
combination of Schmitt's construction and the Givental map.  A proper
relative Givental map is described in Popa-Roth \cite{po:stable}, and
in this case gives a morphism
\begin{equation} \label{givmor} \ol{\M}^G(C,X,d)
  \to\ol{\M}^{G,\quot}(C,X,d).\end{equation}
For each fixed bundle this map collapses bubbles of the section $u$
and replaces them with base points with multiplicity given by the
degree of the bubble tree.  Since the Givental morphism
\eqref{givmor}, the forgetful morphism
$\ol{\M}^G_n(C,X,d) \to \ol{\M}^G(C,X,d)$ and the quot-scheme
compactification $\ol{\M}_n^{G,\quot}(C,X,d)$ are all proper, so is
$\ol{\M}^G_n(C,X,d)$.


\subsection{Energy positivity} 

A natural notion of {\em energy} of a gauged map is defined as
follows.  For a gauged map $v = (P,u)$ the energy is given by the
pairing with the equivariant first Chern class of the linearization
$$ \cE(v): = (d, c_1(P(\ti{X}))) \in \Z,\quad d = v_* [C] \in
H_2^G(X,\Z) .$$
From Mundet's correspondence \cite{mund:corr} it is immediate that the
energy is non-negative, since in the symplectic definition the energy
is defined as an integral of a non-negative function (the {\em energy
  density}) over the domain curve.  Here we give an algebraic proof:

\begin{lemma}  \label{positivity} For any Mundet-semistable gauged map $v = (P,u)$ from a smooth projective genus zero curve $C$ with class  $d = u_* [C] \in H_2^G(X,\Z)$, the pairing $\cE(v) = (d,
  c_1(P(\ti{X}))) \in \Z$ is non-negative.  The energy $\cE(v)$
  vanishes only if the bundle $P$ is trivializable and $u$ constant in
  some trivialization of $P(X)$ induced by a trivialization of $P$.
\end{lemma} 

\begin{proof}   We give two proofs.  By a special case of the
Drinfeld-Simpson theorem \cite{ds:red}, \cite[Lemma 3.2.7]{cf:st}, $P$
admits a reduction to a Borel subgroup $B \subset G$.  Let $\pi_B: B
\to T$ be the projection \eqref{piq}, and $\Gr(P)$ the associated
graded.  Since the map $\pi_Q$ is $T$-equivariant, the section $u$
induces a section $\Gr(u): \cC \to \Gr(P)(X)$ that is also
$T$-semistable.  Therefore it suffices to consider the case $G = T$.
Let $k = \dim(V)$ and $ V = \bigoplus_{i=1}^k V_i $ the decomposition
of $V$ into weight spaces $V_i$ with weight $\mu_i$.  

We use the notation introduced in example \ref{toric example}. In particular the first Chern class $c_1^G(\ti{X})$ becomes identified, up to positive scalar multiple with a pair 
$$c_1^G(\ti{X}) \mapsto (1,\theta) \in \Z \oplus \g_\Z^\dual .$$
The Mundet semistability criterion for one-parameter subgroups of $\Aut(P) \cong T$
has Mundet weights equal to 
$$ \mu_{BG}(\sigma,\lambda) = (d(P)^\vee,\lambda), \quad
\mu_M(\sigma,\lambda) = (d(P)^\vee,\lambda) + \mu_X(\sigma,\lambda)
.$$

By lemma \ref{torus action on  P(V)} \eqref{three} we have
\[
 (\theta-d(P)^\dual,d(P)) + d(u) \geq 0 .
\]
This implies
%
\begin{eqnarray}\label{dPP} (v_*[C],c_1(P(\ti{X}))) &=&  (\theta,d(P)) + d(u) \\
&=& (\theta - d(P)^\vee, d(P)) + d(u) +  (d(P)^\vee, d(P)) \\ &\ge&
  (\theta - d(P)^\vee, d(P)) + d(u) \ge 0
\end{eqnarray} 
as claimed.  If $(d(P),\d(P)^\vee)$ is zero then we must have $d(P) =
0$, hence $P$ is trivializable.  Hence
$$P(\ti{X}) = C \times \ti{X}, \quad (\pi \circ u)_* [C] = 0 \in
H_2(X)$$ 
where $\pi: P(X) \cong C \times X \to X$ is the projection on the
second factor.  This implies that $u$ is constant.

In the second proof we evaluate the Mundet weight for a carefully
chosen one-parameter subgroup. As before assume $G =
T=(\mathbb{C}^\times)^r$ and $X = \mathbb{P}(V)$. Consider $v \colon C
\to X/G$ as a pair $(P\to C, P\xrightarrow{\alpha} \P(V))$ with
$\alpha$ a $T$-equivariant map.  The energy of $v$ is the pairing
$$ (v_*[C], c_1^G(\tilde{X})), \quad \text{where} \ v_*[C] \in
\Z\oplus \mathfrak{t}_\Z, \quad c_1^G(\tilde{X}) \in \Z \oplus
\mathfrak{t}_\Z^\vee .$$
The latter is up to positive scalar equal to $(1, \theta)$; the former
is equal to $(\lambda, d(P))$ for an appropriate integer $\lambda$.
The energy is equivalently
$$\cE(v)= v^*c_1^G(\tilde{X}) \in H^2(C) = \Z .$$
Hence $\lambda = \deg(\alpha)$.  The following is readily verified. If
$d(P) = (d_1, \dotsc, d_r)$ then $P$ is the frame bundle of the vector
bundle $E$ defined by
$$E:=\oplus_{i  = 1}^r \mO_C(d_i) .$$ 
The map $\alpha$ is given by global sections
$u_0, \dotsc, u_m \in H^0(\alpha^*\mO_{\P(V)}(1))$
which are weight vectors for $G$.  Consider the weight space
decomposition $V = \oplus_i V_i$ where $G$ acts on $V_i$ with weight
$\mu_i$.  Equivariance implies that $-\mu_i$ is the weight of $u_i$.
We claim $\deg(\alpha)\geq (-\mu_i,d(P))$. To see this let
$$ |E| = \underline{\Spec}(\on{\Sym}^*(E^\vee))$$
be the total space of $E$.  Via the clutching construction $T$ is
given by gluing of trivializations in coordinate charts near
$[1,0],[0,1] \in \P^1$,
$$ |E| = \Spec \C[z,x_1, \dotsc, x_r] \cup \Spec
\C[z^{-1},y_1, \dotsc, y_r]$$
with $y_i = z^{-d_i}x_i$.  The space of global sections
$H^0(\alpha^*\mO_{\P(V)}(1))$ has a basis of pairs
$(z^j\prod_ix_i^{n_i},z^{-k}\prod_i y_i^{n_i})$ that transform as
follows
\[
z^j \prod_i x_i^{n_i} = z^j \prod_i z^{n_i d_i} y_i^{n_i} = z^{j+(-\mu_i,d(P)) - \deg(\alpha)) }\prod_i y_i^{n_i}
\]
That is 
$-k = j+(-\mu_i,d(P)) - \deg(\alpha) \leq 0. $
As $j\geq 0$ we conclude 
$(-\mu_i,d(P))\leq \deg(\alpha) .$
Mundet stability for the one-parameter subgroup generated by $\lambda$
is defined by a limiting equivariant map $\alpha^0 \colon P\to \P(V)$
given by sections $u_i^0$ whose image is fixed by $z^\lambda$. The
stability condition is
\[
\min_{i, u_i^0\neq 0} (d(P)^\vee,\lambda)+(\theta-\mu_i, \lambda)\leq 0.
\]
Substitute in $\lambda = - d(P)^\vee$ and multiply by $-1$ to obtain
$$(d(P),d(P)^\vee) + (\theta,d(P)^\vee)+(-\mu_i,d(P)^\vee)\geq 0 .$$
Therefore
$$ \cE(v)= (\theta,d(P)^\vee)+\deg(\alpha)\geq
(\theta,d(P)^\vee)+(-\mu_i,d(P)^\vee)\geq 0 .$$
For equality to hold we need $d(P) = 0$ and $\deg(\alpha)=0$.  The
first condition together with equivariance says $\alpha$ factors
through $C$; the second condition says $\alpha$ is constant.
\end{proof}

\begin{remark} \label{alten}
The following gives an alternative proof of non-negativity of the
energy in the case that any Mundet semistable map $(P,u)$ has
non-empty semistable locus $u^{-1}(P(X^{\ss}))$, see Corollary
\ref{large} below.  In this case an invariant ample divisor $D \subset
X$ is given by choosing an invariant section of the ample bundle
$\ti{X}^k$ for $k$ large, as in \eqref{invsection}.  Let
$$D \qu G : = (X^{\ss} \cap D)/G \subset X \qu G $$  
denote the associated divisor in the git quotient.  We may assume that
the divisor $D \qu G \subset X \qu G \subset X/G$ does not contain
$v(C)$, since $D \qu G$ is ample.  Since the divisor $D$ is
$G$-invariant and ample, $D$ contains the unstable locus, that is, $ D
\supset (X \ssm X^{\ss}) .$ The divisor $D$ then induces a divisor
$$P(D) = P \times_G D \subset P(X) .$$ 
Let $u^{X \qu G}: C \to X \qu G$ denote the induced map to the
symplectic quotient.  Since $u(C)$ is not contained in $P(D)$, the
pairing is the number of intersection points counted with
multiplicity:
$$ (v_*[C], c_1(P \times_G \ti{X})) = \# u^{-1} (P(D)) .$$
If the pairing is zero then the image of $u$ is contained in the
semistable locus, and $u$ induces a constant map to $X \qu G$.  Hence
the bundle and section are trivializable.
\end{remark}

\subsection{Convex targets}
\label{convex}

The definition of Mundet semistability also gives good moduli spaces
in the cases of some affine targets.  A finite dimensional complex
$G$-vector space $V$ is said to be {\em convex} if there exists a
central one-parameter subgroup $\phi_\xi: \C^\times \to G$ such
that $V$ has positive weights for the induced action of
$\phi_\xi$,
$$ V = \bigoplus_\mu V_i, \quad \mu_i(\xi) > 0, \quad i =1,\ldots,k .$$
Given a convex $G$-vector space, the {\em projectivization} of $V$ is
the quotient
$$ \ol{V} = ((V \times \C)^\times - \{ (0 , 0) \})/\C_\xi^\times $$
where $\C^\times$ acts on $\C$ with weight one.  Thus $\ol{V}$ is a
weighted projective space (in a particular a smooth Deligne-Mumford
stack) and contains $V$ as an open substack.  A quasiprojective
$G$-variety $X$ is {\em convex} if there exists a projective embedding
$\pi: X \to V$ to a convex $G$-vector space $V$ whose image intersects
the locus $V - \{ 0 \}$.  The following is a simple application of the
technique called {\em symplectic cutting} in the literature
\cite{le:sy2}:

\begin{lemma} Any convex $G$-variety $X$ embeds as a dense open
  substack of a Deligne-Mumford stack $\ol{X}$ with complement a prime
  $\C^\times_\xi$-fixed divisor isomorphic to $(X - \{ 0
  \})/\C_\xi^\times$.
\end{lemma} 

\begin{proof} 
Let $\ti{X} \to X$ denote the given linearization on $X$ and
$\ti{X}(l)$ the linearization on $X \times \C$ obtained by twisting by
the $\C^\times$-character with weight $l$.  Consider the git quotient
$$ \ol{X} = (X \times \C) \qu \C^\times_\xi .$$
The inverse image of $ (0,0) \in V \times \C$ is unstable, for
sufficiently large $d$.  Thus the proper morphism $X \to V$ induces a
proper morphism $\ol{X}$ to $\ol{V}$.  In particular, this implies
that $\ol{X}$ is also proper.  The $G$ action on $X \times \C$ given
by $g(x,z) = (gx,z)$ descends to a $G$-action on $\ol{X}$, and
restricts to the given action on the open substack $X \subset \ol{X}$.
\end{proof} 

In the following we will refer to $\ol{\M}_n^G(C,\ol{X},d)$ allowing
$\ol{X}$ to be a smooth Deligne-Mumford stack without further comment;
we do not allow stacky structures on the domain curves since we are
only interested in defining $\ol{\M}_n^G(C,X,d)$ in which case the
target $X$ is a variety.

\begin{corollary} \label{disjoint} Let $d \in H^2_G(\ol{X})$ be a class that pairs
  trivially with the divisor class $[\ol{X} - X] \in H_2^G(\ol{X})$.
  Then there exists a constant $l(E)$ such that if the energy bound
  $\mE(d) < E$ holds and $l \ge l(E)$ then the moduli stack
  $\ol{\M}_n^G(C,\ol{X},d)$ consists of maps whose images are disjoint
  from $(\ol{X} - X)/G$.  
\end{corollary}

\begin{proof} The intersection number of any curve $u: \P^1 \to
  \ol{V}$ contained in $\ol{V} - V$ with $\ol{V} - V$ is non-negative.
  Indeed $\ol{V} - V \cong \P[\mu_1,\ldots,\mu_k]$ has ample normal
  bundle $\mO_{\P[\mu_1,\ldots,\mu_k]}(1)$.  On the other hand, there
  are no stable gauged maps $C \to X/G$ with image in $(\ol{V} - V)/G$
  for sufficiently large $l > l(d)$.  The trivial reduction $\sigma$
  together with the generator $\xi$ of the one-parameter subgroup
  $\C^\times$ has weight $\mu(\sigma,\xi) \to \infty $ as $d \to
  \infty$, while \eqref{dPP} implies that $( d(P)^\dual, \lambda )$ is
  bounded in terms of the energy.  Combining these observations let
  $v: \hat{C} \to \ol{V}/G$ be a stable gauged map intersecting
  $(\ol{V} - V)/G$.  The intersection number $\# u^{-1}(P(\ol{V}- V))
  > 0 $ is positive and equal to the pairing $(d, [\ol{V} - V]) \in
  \Q$ of $d \in H_2^G(X,\Q)$ with $[\ol{V} - V] \in H^2_G(\ol{V},\Q)$.
  The latter vanishes by assumption, a contradiction.
\end{proof} 

\begin{lemma} Suppose that $(P,u)$ is a map from $C$ to $X/G$ with $X$
  convex.  Any destabilizing pair $(\sigma,\lambda)$ has associated
  graded $(\Gr(P),\Gr(u))$ disjoint from the divisor at infinity
  $P(\ol{X} - X)$, for $l$ sufficiently large.
\end{lemma}

\begin{proof} 
  As in the proof of Lemma \ref{positivity}, by choosing a Borel
  structure refining the parabolic structure and passing to the
  associated graded we may assume that $G$ is a torus.  Suppose that
  $(\sigma,\lambda)$ has associated graded $\Gr(P),\Gr(u)$
  intersecting $P(\ol{X}- X)$.  By invariance of intersection number,
  the limit $\Gr(u)$ must take values in $P(\ol{X} - X)$, since
  otherwise the intersection number with $P(\ol{X} - X)$ would be
  positive.  We suppose that $\mu_1,\ldots, \mu_k$ are the weights of
  $G$ on $V$, so that $u$ has components $u_1,\ldots, u_k$.  From the
  description of the associated graded, if $\lambda$ satisfies
  $\mu_i(\lambda) \ge 0$ for all $i = 1,\ldots, k$ such that $u_i$ is
  non-zero, then the associated graded takes values in $P(X)$.  Hence
  $\mu_i(\lambda) > 0$ for non-empty subset $I \subset \{ 1,\ldots, k
  \}$ of indices such that $u_i$ is non-zero.  Let
$$m = \min_{i \in I} \frac {\mu_i(\lambda)}{ \mu_i(\xi)}. $$
The minimum $m$ is negative since some $\mu_i(\lambda) < 0 $ and
$\mu_i(\xi) > 0$ for all $i = 1,\ldots, k$.  The associated graded
section is then given by the collection of sections $u_i, i \in I$
with $\mu_i(\lambda)/\mu_i(\xi) = m$.  The corresponding
Hilbert-Mumford weight is the weight of the action of
$\C^\times_\lambda$ on the fibers of $\ti{X} / \C^\times_\xi$, and is
equal to $ml$.  Therefore, the weight $\mu(x,\lambda)$ is positive.
For $l$ sufficiently large the pair is not destabilizing.
\end{proof}

As a result, for convex target it suffices to check semistability for
pairs such that the associated graded exists without compactification.

\section{Variation of polarization}

The moduli space of Mundet-semistable gauged maps depends on the
linearization.  Changing the linearization leads to wall-crossing in
which loci of bundles with the same associated graded are flipped
\cite{cross} as is standard in variation of git as explained in e.g
Thaddeus \cite{th:fl}.  Consider the family of linearizations
$\ti{X}^k$ given by the $k$-th tensor product of the given one $\ti{X}$
for $k$ a positive integer.  While taking tensor products does not
change the definition of semistability for $X$, it does change the
definition of Mundet semistability.

\begin{lemma} \label{finitelem} For any fixed degree $d \in H_2^G(X)$,
  there are at most finitely many changes in the stability condition
  as $k$ varies.  That is, there exist
$$ -\infty = k_0 < k_1 < \ldots < k_l =\infty \in \Q \cup \{
-\infty,\infty \} $$
such that if $k',k'' \in (k_i,k_{i+1})$ then the semistable loci for
$k',k''$ are equal.
\end{lemma} 

\begin{proof}   Denote by $\mu_k(\sigma,\lambda)$ the
  Mundet weight corresponding to $\ti{X}^k$.  Changes in the
  definition of stability correspond to pairs $(P,u)$ such that for
  some pair $(\sigma,\lambda)$ and $k_-,k_+ \in \Q$,
$$ \mu_{k_-}(\sigma,\lambda) < 0 , \quad \mu_{k_+}(\sigma,\lambda) > 0
  $$
while for some $k\in (k_-,k_+)$, $ \mu_{k}(\sigma,\lambda ) = 0 $ so
that the pair $(P,u)$ is semistable but not stable.  As in \ref{torus action on  P(V)}(1),
the wall-crossings arise from pairs $(P,u)$ such that
\begin{equation} \label{dP2}
\dim( \on{hull} ( \{ \mu_i | \ti{u}_i \neq 0 \})) < \rank(G),\quad
\on{hull} ( \{ \mu_i | \ti{u}_i \neq 0 \}) \ni \theta + d(P)/k.
\end{equation}

Suppose there are infinitely many wall-crossings.  Let $(P_k,u_k)$
denote the corresponding reducible gauged maps for some $k$ in an
unbounded set $\cW(d) \subset \Q$. The equation \eqref{dP2} implies
that $ \Vert d(P_k) \Vert > c k$ for some positive constant $c$ and
all $k \in \cW(d)$.  On the other hand, as in \eqref{dPP}
\begin{eqnarray*}\label{dPP2} (u_{k,*}[C],c_1(P_k(\ti{X}))) &=& (\theta,d(P_k))+ d(u_k) \\
&=& (\theta - d(P_k)/k,
d(P_k))+ d(u_k) + (d(P_k)/k,d(P_k)) \\ & \ge& (d(P_k),d(P_k))/k  \ge  c^2 k  \end{eqnarray*}
for $k \in \cW(d)$.  Since $\cW(d)$ is unbounded, this implies that
the homology class $ v_{k,*}[C] \in H_2^G(X)$ is also unbounded, a contradiction.
\end{proof} 

Denote by $\ol{\M}^{G}(C,X,d,k)$ the moduli space of Mundet semistable
maps using the linearization $\ti{X}^k$.  By the finite-ness above in
Lemma \ref{finitelem}, we have the following:

\begin{corollary} \label{large}  For any $d \in H_2^G(X,\Q)$ there exists
$k(d)$ such that for $k \ge k(d)$, the stack $\ol{\M}^G(C,X,d,k)$
  consists of those bundles that are Mundet semistable for all $k \ge
  k(d)$, that is,
$$ ( k_1,k_2 \ge k(d)) \ \implies \ (\ol{\M}^G(C,X,d,k_1) =
\ol{\M}^G(C,X,d,k_2) ) .$$
More precisely, an object $(\hat{C},P,u,\ul{z})$ is destabilized by
$(\sigma,\lambda)$ for some $k \ge k(d)$ iff it is destabilized for
all $k \ge k(d)$.
\end{corollary} 

The following describes the Mundet semistability condition for large
$k$.

\begin{lemma} \label{genlem} For any $d \in H_2^G(X,\Q)$ there exists
  $k(d)$ such that for $k \ge k(d)$, the stack $\ol{\M}^G(C,X,d,k)$
  has objects given by tuples $(P,\hat{C},u, \ul{z})$ for which $(u |
  \hat{C}_0)^{-1}(X^{\ss}/G)$ is non-empty.
\end{lemma} 

\begin{proof}  It suffices to
show that if $v: C \to X/G$ is a Mundet unstable gauged map with class
$d$ for large $k$, then $v(C)$ is contained in some Kirwan-Ness
stratum $X_\lambda/G$ and vice versa.  Let $(\sigma,\lambda)$ be a
pair destabilizing $u$:
$$ \mu(\sigma,\lambda) = \mu_{BG}(\sigma,\lambda) + k
\mu_X(\sigma,\lambda) > 0 .$$
 By Corollary \ref{large}, we may assume that $v = (P,u)$ is Mundet
 destabilized by $(\sigma,\lambda)$ for all $k \ge k(d)$.  Then the
 Hilbert Mumford weight
$\mu_X(\sigma,\lambda) > 0 $
must be positive.  The associated graded $\Gr(u)$ is contained in the
fixed point set $Z_\lambda$, that is,
$ \Gr(u) ( \hat{C}_0) \subset P(Z_\lambda) $,
and $(\Gr(P),\Gr(u))$ has positive Hilbert-Mumford weight with respect
to $(\sigma,\lambda)$.  Thus $u$ is generically unstable. 

Conversely, suppose that $v = (P,u): \hat{C} \to X/G$ takes values in
some stratum $X_\mu/G$ generically on the root component $\hat{C}_0
\subset \hat{C}$.  The stratum fibers

\begin{equation} \label{Qproj}X_\mu = G \times_{Q_\mu} Y_\mu \to G/Q_\mu\end{equation} 
as in \eqref{kn}.  By composition with the map $P(X) \to P(Q_\mu) =
P/Q_\mu$ arising from \eqref{Qproj} we obtain a map
$ \sigma: (u | \hat{C}_0 \cap u^{-1}(X_\mu)) \to P/Q_\mu $.
Locally $P$ is trivial, and so in a neighborhood of any point in
$u^{-1}(X_\mu)$ the map is given by a map to $G/Q_\mu$.  By
completeness of $G/Q_\mu$ this map extends to $\sigma: \hat{C}_0 \to
P/Q_\mu$, by definition a parabolic reduction $\sigma$.  Consider the
one-parameter subgroup generated by a positive coweight $\mu$.  The
associated graded $\Gr(u)$ maps to $Z_\mu$ on the root component.  The
Hilbert-Mumford weight $\mu_X(\sigma,\lambda)$ is positive, by
construction. Hence $\mu_{BG}(\sigma,\lambda) + k
\mu_X(\sigma,\lambda)$ is positive for large $k$, and the pair $(P,u)$
is Mundet unstable for large $k$.
\end{proof}

\section{Scaled gauged maps} 

The Mundet semistable moduli spaces have a large linearization limit
which includes both stable maps to the git quotient as well as what we
called affine gauged maps.  This is an algebraic version of a limit
that was first studied in the symplectic context by Gaio-Salamon
\cite{ga:gw}.

\begin{definition}  {\rm (Scaled gauged maps)}  \label{scaledgauged} A prestable {\em scaled gauged map} is a datum
  $(P,\hat{\cC},u, \delta,\ul{z})$ consisting of a prestable scaled
  curve $(\hat{\cC},\delta,\ul{z})$ and pair $(P \to C,u :\hat{\cC}
  \to P(X) )$ giving a map to the quotient stack $\hat{C} \to X/ G$.
  In the case that $X \qu G$ is an orbifold, the domain $\hat{C}$ is
  allowed to have a twisted stacky structure $\hat{\cC}$ so that the
  points with non-trivial automorphism are nodes with infinite scaling
  and the data above gives a representable morphism $v: \hat{\cC} \to
  X \qu G$ as in \cite{agv:gw}.  Denote by
$$D_\infty = \P(\omega_{\hat{\cC}/C}) \subset \P(\omega_{\hat{\cC}/C}
  \oplus \mO_{\hat{\cC}}), \quad \text{resp.} \quad D_0 =
  \P(\mO_{\hat{\cC}}) \subset \P(\omega_{\hat{\cC}/C} \oplus
  \mO_{\hat{\cC}})$$
the divisor at infinity resp.  the zero section.  The datum
$(P,\hat{\cC},u,\delta)$ is {\em semistable} if either
\begin{enumerate}
\item the scaling $\delta | \hat{C}_0$ is finite, and the datum
  $(P,\hat{\cC},u)$ is Mundet semistable; here we are interested in
  the chamber $k \ge k(d)$ from Lemma \ref{genlem}, or
\item the scaling $\delta | \hat{\cC}_0$ on $\hat{\cC}_0$ is infinite,
  and $\delta^{-1}(D_\infty) \subset \hat{\cC}$ maps to the semistable
  locus in $X/G$,
\end{enumerate} 
A semistable scaled gauged map is {\em stable} if it has finitely many
automorphisms.
\end{definition} 

We introduce the following notation for moduli stacks.  Denote by
$\ol{\M}^G_{n,1}(C,X)$ the moduli of stable marked scaled gauged maps.
The existence of a universal scaled curve implies that again,
$\ol{\M}^G_{n,1}(C,X)$ is a hom stack from a Deligne-Mumford stack to
a quotient stack of a variety by a reductive group, and so Artin by
\cite[Proposition 2.3.4]{lieblich:rem}.

The moduli stack of stable scaled curves defines a cobordism using the
following forgetful morphism.  Forgetting everything besides the map
$\delta$ defines a morphism
$$ \rho: \ol{\M}_{n,1}^G(C,X,d) \to \ol{\M}_{0,1} \cong \P^1 , 
\quad (C,v,\delta,\ul{z} = (z_1,\ldots,z_n)) \mapsto \delta
|_{\hat{C}_0 \cong C} .$$
The fiber of $\rho$ over any non-infinite point $\alpha \in \P^1 - \{
\infty \}$ is
$$ \rho^{-1}(\alpha) \cong \ol{\M}_n^G(C,X,d) $$
the space of Mundet semistable gauged maps in the chamber $k \ge
k(d)$.  On the other hand, the fiber over infinity consists of stable
maps to $C \times X \qu G$ of degree $(1,d)$ together with bubble
trees which call {\em affine gauged maps} because of the affine
structure given by the one-form.  Affine gauged maps were introduced
first in a symplectic context by Ziltener \cite{zilt:qk}; a
Narasimhan-Seshadri correspondence which relates that viewpoint with
the one given here is in Venugopalan-Woodward \cite{venuwood:class}.

\begin{definition} \label{affinegauged} 
An affine gauged map is a datum
$$ (C,\delta,\ul{z} = (z_0,\ldots,z_n),v)$$
where $(C,\delta,\ul{z})$ is an affine scaled marked curve from
Definition \ref{affine}, and $v = (P,u): C \to X/ G$ is a morphism to
the quotient stack such that
\begin{enumerate}
\item $v( \delta^{-1}(D_\infty)) \subset X^{\ss}/G$.  In other words,
  on the locus $u^{-1}(\delta^{-1}(D_\infty))$, the map has image in
  the $X$-semistable locus; and
\item on the locus $v^{-1}(\delta^{-1}(D_0))$, the bundle is trivial.
\end{enumerate} 
In the case that $X \qu G$ is an orbifold, $C$ is equipped with an
twisted stacky structure $\cC$ with non-trivial automorphism groups
only at the nodes and marking with infinite scaling as in
\cite{agv:gw} and data above defines a representable morphism $v: \cC
\to X/G$.  Such a datum is {\em stable} if there exist only finitely
many automorphisms $\varphi \in \Aut(C,v,\delta,\ul{z})$, or in other
words, if there each component $C_i$ on which the map $v_* [C_i] = 0
\in H_2^G(X,\Q)$ the scaling is finite and non-zero (resp. zero or
infinity) has at least two (resp. three) special points.
\end{definition} 

We introduce the following notation for moduli spaces of affine gauged
maps. Let $\ol{\M}^G_{n,1}(\bA,X)$ denote the moduli space of affine
gauged maps with group $G$ and target $X$ with $n$ markings in
addition to the marking at infinity.  For each $d \in H_2^G(X,\Q)$,
let $\ol{\M}^G_{n,1}(\bA,X,d)$ denote the locus of maps with $v_*[C] =
d \in H_2^G(X,\Q)$.  The moduli stack admits natural evaluation maps
$$ \ev \times \ev_\infty : \ol{\M}^G_{n,1}(\bA,X) \to (X/G)^n \times \overline{\cI}_{X \qu
G} $$
given by evaluation at the markings $z_i, i = 1,\ldots, n$ and $z_0$.
Also define, for ease of notation,
$$ \ol{\M}_{n}(C,X \qu G,d) = \ol{\M}_{g,n}(C \times X \qu G,(1,d)) $$
the so called {\em graph space} of stable maps to $X \qu G \times C$
of degree $(1,d)$.  Generalizing the positivity of energy of stable
gauged maps in Lemma \ref{positivity} we have the following:

\begin{proposition}  \label{posen}
 Any object $(P,\hat{\cC},u,\ul{z})$ of $\ol{\M}_{n,1}^G(C,X)$ or
 object $(\cC, \ul{z},\delta,P,u)$ of $\ol{\M}_{n,1}^G(\bA,X)$ has
 non-negative energy, and vanishing energy only if the bundle and
 section are trivializable on each component.
\end{proposition} 

\begin{proof}  Each irreducible component of the domain carries either 
a Mundet-semistable map, a map to $X$, a map to $X \qu G$, or an
affine gauged map which is necessarily generically semistable.  The
statement of the proposition follows from applying Lemma
\ref{positivity} and Remark \ref{alten} to each component.
\end{proof} 

Later we will need a bound on the number of irreducible components of
the domain in terms of the energy.

\begin{corollary} \label{kcor} 
Let $k$ be an integer such that if $x$ is any object of $\Hom(\pt,X
\qu G)$, then the order of the automorphism group $| \Aut(x)|$ of $x$
divides $k$.  Any object $(P,\hat{\cC},u,\ul{z},\delta)$ of
$\ol{\M}_{n,1}^G(C,X,d)$ or object $(P,\cC, u, \ul{z},\delta)$ of
$\ol{\M}_{n,1}^G(\bA,X,d)$ with non-zero energy has energy at least
$1/k$.
\end{corollary} 

\begin{proof}  

By Proposition \ref{posen}, any component with non-zero energy has
positive energy defined as the pairing of $u^* c_1(P(\ti{X}))$ with
$[\cC]$.  By \cite[Theorem 1.187 part (iii)]{behrend:review}, $k u^*
c_1(P(\ti{X}))$ is represented by an integral divisor, and so an
integral cohomology class.  The statement of the Corollary follows.
\end{proof}

\section{Properness for trivial actions}

 In this section we show properness of the moduli stack of gauged
 scaled maps in the case that the group acting is trivial.

 \begin{proposition} \label{trivaction} Let $\XX$ be a smooth proper
   Deligne-Mumford stack with projective coarse moduli space $X$ and
   ample line bundle on the coarse moduli space $\ti{X} \to X$.  For
   any $n > 0, E > 0$, the union of the stacks
   $\ol{\M}_{n,1}(C,\XX,d)$ for $(d, c_1(\ti{X})) < E$ is proper.
\end{proposition} 

\begin{proof}
By properness of moduli stacks of stable maps to Deligne-Mumford
stacks \cite[Section 6]{abramovich:compactifying}, the union of the
components
$$ \bigcup_d \ol{\M}_n(C,\XX,d)
:= \ol{\M}_{g,n}( C \times \XX, (1,d)) , \quad (d, c_1(\ti{X})) <
E$$
is proper.  Therefore it suffices to show that the forgetful morphism
$$ f: \ol{\M}_{n,1}(C,\XX,d) \to 
\ol{\M}_n(C,\XX,d) := \ol{\M}_{g,n}( C
\times \XX, (1,d)) $$
obtained by forgetting $\delta$ and collapsing unstable components is
proper.  Let $[u] \in \ol{\M}_n(C,\XX,d)$ with representative $u:
\hat{C} \to \XX$.  Since $\XX$ is projective, Bertini implies that
there exists a divisor $D \subset X$ transverse $u$ and meeting each
non-constant component of $u$ transversally and disjoint from the
markings and images of unstable components of the domain.  Let $\DD =
D \times_X \XX$ and $U \subset \ol{\M}_{n,1}(C,\XX,d)$ be the open
substack of maps such that each component meets $\DD$ transversally
and in a set of distinct points disjoint from the markings and ghost
components.  By taking a divisor $D$ of sufficiently large degree, we
may assume that for each component $\hat{C}_i \subset \hat{C}$, the
map $u$ restricted to $\hat{C}_i$ meets the divisor in at least three
points:
$$ \# (u |_{\hat{C}_i})^{-1}(\DD) \ge 3, \quad \forall \hat{C}_i \subset
\hat{C} .$$
Choose an ordering of the additional points $u^{-1}(\DD)$ meeting $u$.
Let $U$ denote the substack of $\ol{\M}_{n+k,1}(C,\XX,d)$ so that the
last $k$ points represent transverse intersections with $\DD$.  The
map forgetting the last $k$ points gives an \'etale morphism from $U$
to $\ol{\M}_n(C,\XX,d)$, see for example \cite[Proposition 4]{fu:st}.
The map
$$\ol{\M}_{n,1}(C,\XX,d) \supset U \to \ol{\M}_{n+k,1}(C), \quad (u:
\hat{\cC} \to \XX, \ul{z}) \mapsto (u: \hat{\cC} \to \XX, \ul{z} \cup
u^{-1}(\DD) ) $$
fits into a Cartesian diagram
$$ \begin{diagram} \node{ \ol{\M}_{n,1}(C,\XX,d)} 
\arrow{s} \node{U} \arrow{e} \arrow{w}
  \arrow{s} \node{\ol{\M}_{n+k,1}(C,\XX,d)} \arrow{e}
\arrow{s} \node{
    \ol{\M}_{n+k,1}(C)} \arrow{s} \\ \node{ \ol{\M}_{n}(C,\XX,d)} 
\node{ f(U)}
 \arrow{e} \arrow{w} \node{\ol{\M}_{n+k}(C,\XX,d)} \arrow{e}
  \node{ \ol{\M}_{n+k}(C)}
\end{diagram} $$
where the right-hand vertical arrow is proper.  Since the pull-back of
proper morphisms is proper and properness is \'etale local in the
target, the left-hand-arrow is also proper.  Since
$\ol{\M}_{n}(C,\XX,d)$ is proper, $\ol{\M}_{n,1}(C,\XX,d)$ is proper
as well.\end{proof}

\begin{proposition} \label{trivaction2}
For any $E, n >0 $ the union of moduli stacks $\ol{\M}_{n,1}(\bA,\XX,d)$
with $(d, c_1(\ti{\XX})) < E$ is proper.
\end{proposition} 

\begin{proof}  Consider the forgetful map
$ f: \ol{\M}_{n,1}(\bA,\XX,d) \to \ol{\M}_{0,n+1}(\XX,d) $
defined by composing $\ol{\M}_{n,1}(\bA,\XX,d) \to
\ol{\MM}_{0,n+1}(\XX,d)$  with the stabilization map
$\ol{\MM}_{0,n+1}(\XX,d) \to \ol{\M}_{0,n+1}(\XX,d)$ \cite[Proposition
  3.10]{bm:gw}, \cite[Proposition 9.1.1]{abramovich:compactifying}.
As before, choose $[u] \in \ol{\M}_{0,n+1}(\XX,d)$ with representative
$u: \cC \to \XX$ and a divisor $\DD \subset \XX$ meeting $u$
transversally away from the markings and ghost components.  Properness
of $\M_{n+k,1}(\bA)$ and the Cartesian diagram
$$ \begin{diagram} \node{ \ol{\M}_{n,1}(\bA,\XX,d)}\arrow{s}  \node{  U } \arrow{w} \arrow{e} \arrow{s} \node{ \ol{\M}_{n+k,1}(\bA)} \arrow{s}  \\
\node{ \ol{\M}_{0,n}(\XX,d)} \node{ f(U)} \arrow{w} \arrow{e} \node{ \ol{\M}_{0,n+k}}
\end{diagram} $$
imply that $\ol{\M}_{n,1}(\bA,\XX,d)$ is proper over
$\ol{\M}_{0,n}(\XX,d)$.  Since $\ol{\M}_{0,n}(\XX,d)$ is itself proper
by \cite[Theorem 1.4.1]{abramovich:compactifying},
$\ol{\M}_{n,1}(\bA,\XX,d)$ is itself proper.
\end{proof} 

\section{Boundedness}  

In this section we show that the moduli space of gauged scaled maps
with fixed numerical invariants is finite type.  The results of
Ciocan-Fontanine-Kim-Maulik \cite[Section 3.2]{cf:st} imply such a
result in the case of a vector space target.  We extend the argument
here to the case of projective spaces. 

\begin{theorem} \label{bounded} (c.f. \cite[Theorem 3.2.5]{cf:st}) Let
  $E > 0 $ and $\cC$ a twisted prestable curve. Let $V$ be a
  finite-dimensional complex vector space with an action of $G$ via a
  representation $G \ra GL(V)$ with finite kernel and $X = \P(V)$.
  Suppose that the semistable locus $X^{\ss}$ is non-empty and equal
  to the stable locus.  Then the following family of gauged maps is
  bounded: pairs $v=(P,u)$ consisting of a principal $G$-bundle $P \to
  \cC$ and representable section $u: \cC \to P(X)$ such that the
  energy $\cE(v) < E$ and the section $u$ sends the generic point of
  $\cC$ to $P(X^{\ss})$.
\end{theorem} 

\begin{proof}  
  A similar theorem with $\P(V)$ replaced by $V$ is given by
  Ciocan-Fontanine-Kim-Maulik \cite[Theorem 3.2.5]{cf:st}.  

  First we assume that $\cC$ is an ordinary curve and $G$ is a torus.
By lemma \ref{torus action on  P(V)}(4) we have
\begin{equation} \label{dform2}
 \cE(v) = (u_*[C],c_1(P(\ti{X}))) = (\theta,d(P)) + d(u) \in [0,E]
 .\end{equation}
By the Hilbert-Mumford criterion the semistable locus in $\P(X)$ is
$$ X^{\ss} = \P(V)^{\ss} = \left\{ [x_1,\ldots,x_k] \in \P(V) \ |
  \ \on{hull} ( \{ \mu_i | x_i \neq 0 \}) \ni \theta \right\} .$$
Let $u: C \to P(X)$ be a section that is generically semistable. Recall example \ref{toric example} and that $u$ is given by a $k$-tuple $(\ti{u}_1, \dotsc, \ti{u}_k)$.  The
condition that $u$ is generically semistable means each $\ti{u}_i \neq 0$ hence by lemma \ref{torus action on  P(V)}(3)
$$  (\mu_i, d(P)^\vee) + d(u) \ge 0, \quad \forall
i = 1,\ldots, k.$$
The same holds for any vector near $\theta$, since the condition of
lying in the convex hull is open in in the interior.  Choose a basis
$$ \xi_1,\ldots, \xi_r \in \g^\dual_\Q, \quad \on{hull}(\xi_1,\ldots,
\xi_r) \ni \theta $$
of points near $\theta$ so that
\begin{equation} \label{xiform}
 (\xi_i,d(P)^\vee) - d(u) \ge 0, \quad i = 1,\ldots, r .\end{equation}
Combining \eqref{xiform} and \eqref{dform2} shows that the possible
degrees $d(P)$ lie in a finite set.  

Next we consider the case that $\cC$ is an ordinary curve and $G$ is
an arbitrary compact connected reductive complex group.  By a simple
case of the Drinfeld-Simpson theorem \cite{ds:red}, \cite[Lemma
3.2.7]{cf:st}, the bundle $P$ admits a reduction to a Borel subgroup
$B \subset G$.  Let $\pi_B: B \to T$ be the projection \eqref{piq},
and $\Gr(P)$ the associated graded.  Since the map $\pi_Q$ is
$T$-equivariant, the section $u$ induces a section $\Gr(u): \cC \to
\Gr(P)(X)$ that is also $T$-semistable.  A $G$-bundle corresponds via
a faithful representation $ G \to GL(r)$ to a vector bundle $F \to C$
together with a reduction of the structure group, given by a section
of an affine bundle $GL(r)/G$.  Equation \eqref{xiform} shows that
splitting type of the associated graded $\Gr(F)$ is uniformly bounded
given a fixed $d \in H_2^G(X,\Z)$, and furthermore the first Chern
class $d(P)$ has bounded pairing with $c_1^G(\ti{X})$.

Given this bound on the splitting type of the associated graded, a
standard argument (see for example the boundedness arguments in
\cite[3.3]{hl:mod}) shows that after twisting by a sufficiently
positive bundle depending only on a bound on the splitting type, any
vector bundle $F \to C$ as above is generated by their global sections
and has no higher cohomology.  Indeed the long exact sequence in
cohomology shows that, for any locally free subsheaf $F' \subset F$
appearing as a summand in the associated graded $\Gr(F)$ we have an
exact sequence
$$ H^0(C,F') \to H^0(C,F) \to H^0(C,F/F') \to H^1(C,F') .$$
From this and the corresponding sequence for the twist $F(-z), z \in
C$ one obtains that if $F',F/F'$ are generated by their global
sections and have no higher cohomology then $F$ has the same property.
An induction shows that $F \to C$ is a quotient of a fixed trivial
bundle $ \mO_C^{\oplus l}$ for $l \ge k(E)$ where $k(E)$ is a constant
depending only on the energy bound $E$.  Thus the family of bundles is
bounded.

To show that the families of sections are bounded, note that any two
sections $u_0,u_1: \hat{C} \to P(X)$ have homology classes that differ
by an element of $H_2(X,\Q) \subset H_2^G(X,\Q)$.  The homomorphism
$$( \cdot, c_1^G(\ti{X})) \in \Hom(H_2^G(X,\Q),\Q)$$
restricts to the standard pairing on $H_2(X,\Q)$ corresponding to the
hyperplane class, the energy bound \eqref{dform2} implies that $d(u)$
is bounded from above and below.  Now the difference of homology
classes of any sections of $P(\P(V))$ lies in the kernel of the map
$H_2(P(\P(V)) \to H_2(C)$ and so, since $\P(V)$ is simply-connected,
lie in the image of the inclusion $H_2(\P(V)) \to H_2(P(\P(V))$ of a
fiber.  It follows that the degree also classifies homology classes of
sections:
\begin{equation} \label{iff} (u_1 \cong u_2 ) \iff (d(u_1) = d(u_2)) .\end{equation} 
By \eqref{iff} the possible homology classes of the sections are
bounded as well.  This shows that the family of maps to $X/G$ is
bounded.

Finally suppose that $\cC$ is a twisted curve and $G$ is complex
connected reductive.  Let $\hat{\cC} \to S$ be a family of twisted
curves, $P \to \hat{\cC}$ a family of bundles and $u: \cC \to P(X)$ a
family of sections as above.  By \cite[Theorem 1.14]{ol:logtwist},
after \'etale cover there exists a finite flat morphism $\pi: Z \to
\hat{\cC}$ from a projective scheme $Z \to S$ to $\hat{\cC}$; the
proof in fact shows that $\pi$ is surjective.  By faithfully flat
descent, sheaves on $\cC$ may be described in terms of descent data as
sheaves $E$ on $ Z \lefttwoarrow Z \times_{\cC} Z. $ Such data
consists of the bundle $Z$ together with isomorphisms $\varphi:
\pi_1^* E \to \pi_2^* E,$ where $\pi_1,\pi_2$ are the projections onto
the factors of $Z \times_{\cC} Z$ see
\href{http://stacks.math.columbia.edu/tag/03O6}{Tag 03O6} in
\cite{dejong:stacks}.  A principal $G$-bundle is given via an
embedding $GL(r)$ as descent data for a locally free sheaf of rank $r$
together with a reduction given by a section of the associated
$GL(r)/G$ bundle and an isomorphism $\varphi: \pi_1^* E \to \pi_2^* E$
preserving the $G$-reduction.  Any such substack may be realized via
quot scheme techniques as a quotient of a variety by a reductive group
action as above.
\end{proof} 

\begin{corollary}
  \label{mainfin} For any real $E > 0$, the union of components
$\ol{\M}_{n,1}^G(C,X,d)$, $\ol{\M}_n^G(C,X,d)$, and
  $\ol{\M}_{n,1}^G(\bA,X,d)$ with $(d, c_1^G(\ti{X})) < E$ is finite
  type.
\end{corollary}  

\begin{proof} 
  We consider only $\ol{\M}_{n,1}^G(C,X,d)$; the proof for the other
  moduli spaces is similar.  We first show that only finitely many
  combinatorial types are possible for a given energy bound.  Let $v =
  (P \to C, u: \hat{\cC} \to P(X))$ be a gauged map of class $d$ and
  energy $\cE(v) = ( d, c_1^G(\ti{X}) )$.  The energy of any component
  with non-zero energy $ \lan v_*[\cC_i], c_1^G(\ti{X}) \ran$ is at
  least $1/k$ for some integer $k$, by Corollary \ref{kcor}. It
  follows that the number of irreducible components $\hat{\cC}_i$ of
  the domain $\hat{\cC}$ with positive energy is bounded by $ k \cE(v)
  + n$.
  
  The bound on the number of components with positive energy gives a
  bound on the total number of components as follows.  Any component
  $\hat{\cC}_i$ of $\hat{\cC}$ on which the map $v: \hat{\cC} \to X/G$
  is trivial and has trivial scaling has at least three special
  points.  Removing the component $\hat{\cC}_i$ defines a curve
  $\hat{\cC} - \hat{\cC}_i$ with at least three connected components,
  and so a partition of the markings $\{ z_1,\ldots, z_n \}$ and
  irreducible components of $\hat{\cC}$ with non-trivial energy into
  three non-empty subsets.  Thus the number of components of the
  domain with trivial scaling is also bounded by the number of
  partitions of $ k \mE(d) + n$.  On the other hand, by the
  monotonicity condition the number of components with scaling is at
  most the number of terminal components, since there is at most one
  component with finite, non-zero scaling on the any path from a root
  component to the terminal component.  That is, the number of
  vertices $\Ver(\Gamma)$ of the combinatorial type $\Gamma$ of
  $\hat{\cC}$ is bounded by an integer $v = v(d)$.  There are finitely
  many trees $\Gamma$ satisfying this bound, hence finitely many
  possibilities for $\Gamma$.

  Given an energy bound, the possible homology classes of each
  component form a finite set by Theorem \ref{bounded} and the
  requirement that $d(P) \in H_2(\cC,\Z/k)$.  It follows that there
  are only finitely many possible labellings $\Ver(\Gamma) \to
  H_2^G(X,\Q)$ of the given graphs by degree two homology classes with
  the given energy bound; hence finitely many combinatorial types as
  claimed.  It follows that the image of $\ol{\M}_{n,1}^G(C,X,d) \to
  \ol{\MM}_{n,1}(C)$ is contained in an Artin stack of finite type for
  each $d \in H_2^G(X,\Z)$.

  We now use boundedness of the splitting type to prove that the
  moduli stack is finite type.  As in the proof of Theorem
  \ref{bounded}, we describe bundles on stacky curves in terms of
  descent data.  Let $\hat{\cC} \to S$ be a family of stacky curves,
  $P \to \hat{\cC}$ a family of bundles and $u: \cC \to P(X)$ a family
  of sections as above.  By \cite[Theorem 1.14]{ol:logtwist}, after
  \'etale cover there exists a finite flat morphism $\pi: Z \to
  \hat{\cC}$ from a projective scheme $Z \to S$ to $\hat{\cC}$; the
  proof in fact shows that $\pi$ is surjective.  By faithfully flat
  descent, sheaves on $\cC$ may be described in terms of descent data
  as sheaves $E$ on $ Z \lefttwoarrow Z \times_{\cC} Z ;$ that is, $Z$
  together with isomorphisms $\varphi: \pi_1^* E \to \pi_2^* E,$ where
  $\pi_1,\pi_2$ are the projections onto the factors of $Z
  \times_{\cC} Z$ see
  \href{http://stacks.math.columbia.edu/tag/03O6}{Tag 03O6} in
  \cite{dejong:stacks}.  A principal $G$-bundle is given via an
  embedding $GL(r)$ as descent data for a locally free sheaf of rank
  $r$ together with a reduction given by a section of the associated
  $GL(r)/G$ bundle and an isomorphism $\varphi: \pi_1^* E \to \pi_2^*
  E$ preserving the $G$-reduction.  Because the splitting type of $P
  \to \cC$ is bounded, the splitting type of $E$ is bounded as well,
  and the image of $P$ in $\Hom(Z,BG)$ is contained in a substack of
  finite type.  Any such substack may be realized by standard
  constructions via quot scheme techniques as a quotient of a variety
  by a reductive group action as in \cite[2.3.4]{lieblich:rem}.  Since
  the homology class $d \in H_2^G(X,\Q)$ is bounded, the image of $S$
  in $\Hom(Z,X/G)$ consists of sections with bounded homology class,
  and so also lies in a substack of finite type by \cite[Theorem
    1.4.1]{abramovich:compactifying}, see also Lieblich
  \cite{lieblich:rem}.  It follows that $\ol{\M}_{n,1}^G(C,X)$ is
  covered by finitely many stacks of finite type, and so is itself
  finite type.
\end{proof}

\section{Universal closure} 

In this section we show that the moduli stack of scaled gauged maps is
universally closed using the valuative criterion and Schmitt's git
construction \cite{schmitt:univ}.

\subsection{Removal of singularities for bundles on surfaces}

We begin with the following theorem of J.-L. Colliot-Th\'el\`ene and
J.-J. Sansuc, \cite{ciollot} describes extensions of bundles on
complements of finite subsets of surfaces, see also
Ciocan-Fontanine-Kim-Maulik \cite{cf:st}: For any scheme $X$ and
reductive group $G$ a {\em principal $G$-bundle} is a scheme $P$ with
a free right action of $G$ that is locally trivial in the fpqc
topology on $X$.  If $X$ is smooth, then this is equivalent to local
triviality in the \'etale topology.

\begin{theorem}  \label{extend} Let $X$ be a smooth complex variety of dimension two and 
$G$ a connected reductive group.  If $U \subset X$ is the complement
  of a finite set of non-singular points on $X$, then any principal
  $G$-bundle on $U$ is the restriction of a principal $G$-bundle on
  $X$, unique up to isomorphism.
\end{theorem} 

We briefly recall the main point of the proof.  First, we prove the
corresponding result for vector bundles.  let $F \to U$ be a locally
free sheaf and $i: F \to X$ the inclusion.  The sheaf $i_* F \cong
i_*F^{\dual \dual}$ is reflexive by e.g. Hartshorne \cite[Corollary
  1.7]{har:ref}, and reflexive sheaves on surfaces are locally free by
e.g. Hartshorne \cite[Corollary 1.4]{har:ref}.  This shows that $F$
has an extension. The extension is unique up to isomorphism by
Hartog's theorem.  If $F_1,F_2 \to X$ are two such locally
free extensions then the given isomorphism $\varphi \in H^0(U,
\Hom(F_1,F_2))$ extends over $X$ since $X - U$ is a finite set.

Next we consider the case of arbitrary reductive groups.  Fix an
embedding $G \to GL(r)$ for some $r \ge 1$ and let $P \to U$ be a
principal $G$-bundle.  Using the embedding, we obtain a vector bundle
$E$, which extends uniquely to $X$.  Let $Q$ denote the frame bundle
of $F$.  The bundle $P$ corresponds to a reduction of structure group
$\sigma: U \to Q/P$.  Since $G$ is connected reductive, it follows
from Matsushima's criterion \cite{mats}, \cite{bb} that the
homogeneous space $GL(r)/G$ is a smooth affine variety. Moreover, by
assumption, $Q/G$ admits a section $\sigma$ over $U$. Thus, by
Hartogs' theorem, $\sigma$ extends over $X$.

\subsection{Existence of limits for families with finite scaling} 

First we show properness over the space of finite scalings using
Schmitt's git construction \cite{schmitt:univ} the Keel-Mori theorem
\cite{km:quot}.

\begin{lemma}   \label{forprop} 
The forgetful morphism to the moduli space of finite scalings
$$ \ol{\M}_{n,1}^G(C,X,d) \supset \rho^{-1}(\M_{0,1}) \to {\M}_{0,1}
\subset \ol{\M}_{0,1}, \quad [\hat{\cC},u, P, \ul{z}, \delta] \mapsto
\delta $$
is proper.
\end{lemma}

\begin{proof} 
  This follows an argument given with Gonz\'alez \cite{cross}: there
  is a proper relative Givental map described in Popa-Roth
  \cite{po:stable}
  $$\ol{\M}^G(C,X,d) \to\ol{\M}^{G,\quot}(C,X,d).$$
  For each fixed bundle, this map collapses bubbles of the section $u$
  and replaces them with base points with multiplicity given by the
  degree of the bubble tree.  On the other hand,
  $\ol{\M}^{G,\quot}(C,X,d)$ has a git construction given in
  \eqref{schmittgit} and so has a proper coarse moduli space.  Finally
  $\ol{\M}_n^G(C,X,d) \to \ol{\M}^G(C,X,d)$ is proper, each forgetful
  map being isomorphic to a universal curve.  Under the
  stable=semistable assumption, the Luna slice theorem
  \cite{luna:slice} implies that $\ol{\M}_n^{G,\quot}(C,X,d)$ is
  \'etale-locally the quotient of a smooth variety by a finite group
  and so has finite inertia stack.  By the Keel-Mori theorem
  \cite{km:quot}, explicitly stated in \cite[Theorem 1.1]{conrad:kl},
  the morphism from $\ol{\M}_n^{G,\quot}(C,X,d)$ to its coarse moduli
  space is proper, so $\ol{\M}_n^G(C,X,d)$ is proper as well.  Hence
$$ \rho^{-1}(\C) \cong \ol{\M}_n^G(C,X,d) \times \C \to \C $$
%
is proper.  Schmitt's construction \cite[Section 2.7.2]{schmitt:git}
implies that if stable=semistable then the automorphism groups of
objects in $\ol{\M}^{G,\quot}_n(C,X,d)$ are finite, and so the moduli
stack is Deligne-Mumford.  Since quot schemes are projective, and
moduli spaces of stable maps to projective schemes are projective, the
moduli spaces $\ol{\M}^{G,\quot}_n(C,X,d)$ have projective coarse
moduli spaces and so are proper.
\end{proof} 

Next we show the valuative criterion for universal closure in the case
that the scalings go to infinity.  This is a combination of properness
for stable maps to the targets, its quotient, and removal of
singularities for bundles on surfaces.

\begin{theorem} \label{finite}
Given a family of scaled Mundet-semistable gauged maps over a
punctured curve $S$ with finite scaling $\delta$
$$(P,\hat{C}, u, \ul{z},\delta) \to S = \ol{S} - \{ \infty \} $$
there exists an extension over $\ol{S}$,  after \'etale cover.
\end{theorem} 

\begin{proof}
It suffices, by the Lemma \ref{forprop}, to consider the case that the
scaling $\delta$ becomes infinite.  We first consider the case that
$\hat{\cC} \cong C$.  There are three steps, in which we construct the
central fiber curve and a scaled gauged map by stages.  First we
construct a limit
$$\hat{\cC}_\infty^{X \qu G} \to C, \quad v_\infty^{X \qu G} :
\hat{\cC}_\infty^{X \qu G} \to X \qu G, \quad \ul{z}_\infty \subset
\hat{\cC}_\infty^{X \qu G} $$
by properness of the moduli space of stable maps to $X \qu G$.
However, the limiting domain $\hat{\cC}_\infty^{X \qu G}$ is not the
one we want because there may be bubbling in $X$ that is not captured
by bubbling in $X \qu G$.  Forgetting some of the components of
$\hat{C}_\infty^{X \qu G}$ and using removal of singularities for
bundles on surfaces gives a curve and map
$$\hat{\cC}_{BG} \to C, \quad \phi: \hat{\cC}_{BG} \to BG $$
where the map $\phi$ is a classifying map for an extension of the
bundle $P$ over $\hat{\cC}_{BG}$. Then we apply properness of the
moduli stack of sections of $P(X)$ to obtain the desired limiting
curve
$$\hat{\cC}_\infty \to C, \quad u_\infty: \hat{\cC}_\infty \to P(X),
\quad \delta_\infty: \hat{\cC}_\infty \to
\P(\omega_{\hat{\cC}_\infty/C} \oplus \mO_{\hat{\cC}_\infty}), \quad
\ul{z}_\infty \subset \hat{\cC}_\infty, .$$
Here are the details:

\vskip .1in \noindent {\em Step 1: Construct the part with infinite
  scaling.}  We first introduce the following notation for the maps to
the git quotient.  By definition of $k(d)$ and Lemma \ref{genlem}, the
maps $u: C \to P(X)$ are generically semistable, and so defines a
curve and map
$$C^{X \qu G} := u^{-1} P(X^{\ss}) \neq \emptyset, \quad u^{X \qu G} =
(u^{X \qu G}: C^{X \qu G} \to X \qu G) .$$
By properness of $X \qu G$, $u^{X \qu G}$ extends to a family of
stable maps with domain $C \times S$.  Order the base points so that
they give sections
$$ \zeta_i: S \to C, \quad \zeta_i(S) \subset P(X^{\us}), i =
1,\ldots, l.$$
Denote by $\ul{z} \cup \ul{\zeta}: S \to C^{n+ l}$ the family of
sections obtained by adding the base points and removing duplicates;
that is, after restricting to an open subvariety we may assume that
any two sections that coincide in one fiber, coincide everywhere; then
we just remove one of the duplicate sections.  Because the domain is
irreducible, the datum $ (C, u^{X \qu G}, \ul{z} \cup
\ul{\zeta},\delta )$ is a stable scaled map to the smooth
Deligne-Mumford stack $X \qu G$.

Properness for the moduli space of scaled maps with trivial group
action in Proposition \ref{trivaction} implies that the family extends
over the central fiber: Since $\ol{\M}_{n,1}(C, X \qu G,d)$ is proper,
the map $u^{X \qu G}$ extends over the central fiber to a stable
scaled map
$$ \left( \hat{\cC}_\infty^{X \qu G} \to C, \quad u^{X \qu G}_\infty:
\hat{\cC}_\infty^{X \qu G} \to X \qu G, \quad \ul{z}_\infty \cup
\ul{\zeta}_\infty, \quad \delta_\infty \right) .$$
In particular, the markings $\ul{\zeta}_\infty$ lie on the locus of
$\hat{\cC}_\infty^{X \qu G}$ with finite scaling $\hat{\cC}_\infty^{X
  \qu G} \ssm \delta_\infty^{-1}(D_\infty)$.

\vskip .1in \noindent {\em Step 2: Construct the part with finite
  scaling.}  Let $\Gamma$ be the combinatorial type of the limit
$\hat{\cC}_\infty^{X \qu G}$ in the previous step and $\Gamma'$ the
combinatorial type obtained by forgetting the components of
$\hat{\cC}_\infty^{X \qu G}$ on which the scale $\delta |
\hat{\cC}_\infty^{X \qu G}$ is zero.  More precisely, choose a family
of sections $ S \to (\hat{\cC}_\infty^{X \qu G})^k $ taking values in
the locus with non-zero scaling with the property that the components
with non-zero scaling become stable.  By e.g. Behrend-Manin
\cite[Lemma 3.12]{bm:gw}, there exists a proper family
$\hat{\cC}_\infty^{BG}$ of stable curves with a morphism from
$\hat{\cC}_\infty^{X\qu G}$ collapsing the components with zero
scaling.  The family $\hat{\cC}_\infty^{BG}$ consists of a collection
of components on which the scaling is finite and non-zero, or
infinite, with a morphism
$$ \varphi: \hat{\cC}_\infty^{X \qu G}\to \hat{\cC}_\infty^{BG} .$$
The scaling $\delta_\infty$ is finite at the base points
$\ul{\zeta}_\infty$ and markings $\ul{z}_\infty$.  The image of the
base points $\ul{\zeta}_\infty$ under the morphism $\varphi$ are
denoted
$\ul{\zeta}_\infty^{BG}= \varphi(\ul{\zeta}_\infty) .$
The points $\ul{\zeta}_\infty^{BG}$ are no longer necessarily distinct
from each other and the markings.  Because the scalings $\delta^{X \qu
  G}_\infty$ are finite at $\ul{\zeta}_\infty$, the scalings
$\delta^{BG}_\infty$ are finite at $\ul{\zeta}_\infty^{BG}$, that is,
$ \delta^{BG}_\infty(\ul{\zeta}_\infty^{BG}) < \infty .$ In
particular, all of the points $\ul{\zeta}_\infty^{BG}$ are
non-singular, since the only nodes in $\hat{\cC}_\infty^{BG}$ are
contained in $(\delta^{BG}_\infty)^{-1}(\infty)$.

Removal of singularities for bundles on surfaces implies that the
bundle extends over the central fiber.  The morphism $u^{X \qu G}$
induces an extension of the bundle $P_\infty^{BG}$ over the complement
of the base points $\ul{\zeta}^{BG}_{\infty}$, given by pull-back of
$$ P_\infty^{BG} \to \hat{\cC}_\infty^{BG} \ssm
\ul{\zeta}_\infty^{BG}, \quad P_\infty^{BG} := (u^{ X\qu G}_\infty |
\hat{\cC}_\infty^{BG} \ssm \ul{\zeta}_\infty^{BG})^* (X^{\ss} \to
X^{\ss}/G)$$
under $u^{X \qu G}_\infty$.  By construction, the points
$\zeta_{i,\infty}$ are non-singular points in $\hat{C}^{BG}_\infty$.
By removal of singularities for bundles Theorem \ref{extend}, the
bundle $P^{BG} \to \hat{\cC}^{BG}$ given by $P_\infty^{BG}$ over the
central fiber has a unique extension over the points
$\zeta_{i,\infty}$.  This implies the existence of a limiting bundle
$P_\infty \to \hat{\cC}_\infty^{BG}$ with classifying map
$$ \phi_\infty : \hat{\cC}^{BG}_\infty \to BG, \quad P_\infty :=
\phi_\infty^*(EG \to BG). $$
Denote by $\hat{C}^{BG}$ the resulting family over $\C^\times$, and
$P$ the resulting bundle over $\hat{\cC}^{BG}$.

\vskip .1in \noindent {\em Step 3: Construct the full limit.}  In the
last step we apply properness for the moduli stack of stable sections.
The associated fiber bundle $P(X) \to \hat{\cC}^{BG}$ is projective,
since $X$ is projective and $\hat{\cC}^{BG}$ is projective.  Then it
follows from properness of stable maps to $P(X)$ that there exists a
limit $u_\infty : \hat{\cC}_\infty \to P(X)$ extending $u$.  The
scaling naturally extends to a scaling $\delta_\infty$, possibly after
adding additional components with finite scaling and trivial maps.

We check that the limit constructed above satisfies the axioms of a
stable scaled gauged map.  The monotonicity condition on the scalings
is guaranteed by the description of the one-form in \eqref{oneform}.
Furthermore, on the locus $\delta^{-1}(D_\infty)$, the map $u_\infty $
agrees with the pull-back of $u_\infty^{X \qu G}$ and so takes values
in the semistable locus.  The locus $\delta^{-1}(D_0)$ is a union of
components that map to points in $\hat{\cC}_\infty^{BG}$.  This
implies that bundle $P$ is trivial on $\delta^{-1}(D_0)$.  Finally,
the inequality $\delta(z_{i, \infty}) < \infty$ is automatically
satisfied since the scaling on $\hat{\cC}_\infty^{X \qu G}$ is finite
at the markings, and the forgetful map maps all components with zero
scaling to loci where the scaling is finite.  Each component on which
the scaling and gauged are trivial has at least three special points,
since the limit $u_\infty$ is a stable section.  Each component with
finite, non-zero scaling has at least two special points contains
either the limit of a marking or a base point.  If trivial such a
component is attached to a component with trivial scaling, and so has
at least two special points.  Finally each component with infinite
scaling occurs in the domain $u^{X \qu G}$ and so has at least three
special points.

Finally we consider the general case that domains of the family are
nodal.  That is, we have a family of gauged maps $(P,\hat{\cC}, u,
\ul{z},\delta) \to S = \ol{S} - \{ \infty \} $ such that every
$\hat{\cC}_s$ is a nodal projective curve.  In this case we repeat the
first two steps for the family obtained by restricting to the
root component $\hat{\cC}_0 \subset \hat{\cC}$.  In the last
stage, properness for stable maps to $P(X)$ implies the existence of a
limit of $u: \hat{\cC} \to P(X)$.
\end{proof} 

Almost exactly the same argument shows that the moduli stack of affine
gauged maps is universally closed:

\begin{lemma} \label{infinite} 
Given a family of stable affine gauged maps over a punctured curve
$S$,
$$(P,\hat{\cC}, u, \ul{z},\delta) \to S := \ol{S} - \{ \infty \} $$
there exists an extension over $\ol{S}$, after \'etale cover.
\end{lemma} 

We use this and the statement of the Lemma above, which dealt with 
irreducible domain, to prove universal closure of the moduli space
of scaled gauged maps:

\begin{theorem} The moduli stack $\ol{\M}_n^G(C,X)$ is universally closed. 
\end{theorem} 

\begin{proof} Let $(P,\hat{C}, v, \ul{z},\delta) \to S = \ol{S} - \{ \infty \}$
be a family of scaled gauged maps over a curve $S$. We may assume that
either $\delta$ is finite or infinite on the root component
$\hat{C}_0$ for all $s \in S$, after possibly replacing $S$ with an
open subscheme.  In the finite case, the existence of a central
extension over the central fiber follows from \ref{finite}.  In the
infinite case, the central fiber is a collection of affine gauged maps
and maps to $X \qu G$ by gluing at the nodes.  That is, the curve
$\hat{C}$ is a union of components $\hat{C} = \hat{C}_0 \cup \hat{C}_1
\cup \ldots \hat{C}_r $ where $\hat{C}_0$ is a family of curves with
infinite scaling and  $\hat{C}_1,\ldots, \hat{C}_r$ are families of
affine scaled curves.  By Lemma \ref{infinite}, the restriction of the
families $(P,\hat{C}, v, \ul{z},\delta)$ to $\hat{C}_1,\ldots,
\hat{C}_r$ have extension over the central fiber.  Similarly
properness of the moduli stack of stable maps to $X \qu G$ implies the
existence of a limit of the restriction of the family to $\hat{C}_0$.
By closure of the diagonal, these families glue together to a scaled
gauged map on the central fiber.
 \end{proof}

\section{Separation} 

In this section we check the valuative criterion for separatedness.
This is again a combination of separatedness of the moduli stack of
stable maps to the target, its quotient, and uniqueness of extensions
on bundles on surfaces.  

\begin{proposition} 
For $i = 0,1$ let 
$v^i := (\hat{C}^i \to \ol{S}, P^i \to \hat{C}^i,
u^i: \hat{C}^i \to X / G, \delta^i, \ul{z}^i: S \to
\hat{C}^{i,n}) $ be families of stable scaled gauged maps over a curve
$\ol{S}$ that are isomorphic over the punctured curve $S = \ol{S} -
\{ \infty \}$.  Then $v^0$ is isomorphic to $v^1$ over $\ol{S}$.
\end{proposition}

\begin{proof}  
Again this follows from a three-step process, in which we show that
the maps agree on parts of the limit corresponding to infinite,
finite, and zero scaling.

\vskip .1in \noindent {\em Step 1: The maps agree on the part with
  non-zero scaling.}  First we introduce for notation the induced map
to the git quotient.  Let $ \hat{\cC}_\infty^{i,X \qu G}$ denote the
union of components $\hat{\cC}^i_{\infty} $ with infinite or finite,
non-zero scaling $\delta_\infty^i | \hat{\cC}^i_{\infty}$.  The
inverse image of the semistable locus $X \qu G$ is dense in
$\hat{\cC}_\infty^{i,X \qu G}$ by Lemma \ref{genlem}.  By properness
of the stack $X \qu G$ we obtain maps
$$ u_\infty^{i,X \qu G} 
: \hat{\cC}_\infty^{i,X \qu G} \to X \qu G .$$

We would like to apply separatedness for the moduli stack of stable
scaled maps to the git quotient to show that these maps are
isomorphic.  However, the maps $u_\infty^{i, X \qu G}$ may not be
stable since they may contain unstable components.  To remedy this,
choose an ample invariant divisor $D \subset X$ as in
\eqref{invsection}.  Let $\ul{\zeta}^i = u^{-1}(P(D)) $ denote the
points mapping to $P(D)$.  Consider the scaled curve to $X \qu G$
given by $ (u^{i,X \qu G} , \delta^i, \ul{z}^i \cup \ul{\zeta}^i) $
where $\ul{z}^i \cup \ul{\zeta}^i$ denotes the union obtained by
adding the intersection points with the divisor and deleting
duplicates and nodes.  After restricting to an open subscheme of $S$
containing the central fiber, we may assume that $\ul{z}^i \cup
\ul{\zeta}^i$ are distinct, non-singular points.

We claim that the tuples constructed in the previous paragraph are
stable scaled maps.  Any component of $\hat{\cC}^i$ with
finite, non-zero scaling that becomes a unstable component after
passing to $X \qu G$ must either be stable, or correspond to a map to
$X/ G$ that lies generically in a fiber of $X^{\ss} \to X^{\ss}/G$.
Such maps always intersect $D$ since $D$ is ample.  Hence any such
component has at least two special points and a non-trivial scaling,
and so is stable.

Separation for the moduli stack of stable scaled maps to the git
quotient implies that the central fibers are isomorphic: By
Proposition \ref{trivaction} there exists an isomorphism
\begin{multline} 
(u_\infty^{0,X \qu G} : \hat{\cC}_\infty^{0,X \qu G} \to X \qu G , \quad
  \delta^0, \ul{z}^0 \cup \ul{\zeta}^0 ) \\ \cong (u_\infty^{1,X \qu
    G} : \hat{\cC}_\infty^{1,X \qu G} \to X \qu G , \quad \delta^1,
  \ul{z}^1 \cup \ul{\zeta}^1 ) .\end{multline}

\vskip .1in \noindent {\em Step 2: The limits agree on the part with
  non-zero scaling.}  This step is an application of uniqueness of
removal of singularities for bundles on surfaces in Theorem
\ref{extend}.  Let $\hat{\cC}_{\infty}^{i,X \qu G} $ denote the curves
from Step 1.  The classifying maps
$ \phi_i: \hat{\cC}_{\infty}^{i,X \qu G} 
\to BG $
are isomorphic (that is, the corresponding bundles are isomorphic)
except at finitely many non-singular points, the base points, since
the maps $u_i^{X \qu G}$ agree.  Furthermore, the base points
$\ul{\zeta}_i$ are contained in components of $\hat{\cC}_\infty^{i,X
  \qu G}$ with infinite scaling, since $D$ contains the unstable locus
$X^{\us}$.  By uniqueness of the completion of bundles on surfaces in
Theorem \ref{extend}, the bundles $P_0,P_1$ are isomorphic over
$\hat{\cC}^{i,X \qu G} $.

\vskip .1in \noindent {\em Step 3: The limits agree entirely.}
Finally we apply separatedness for families of stable sections to show
that the limiting sections are isomorphic.  Separation of stable maps
to $P_\infty(X)$ where
$ P_\infty : = P_i |_{\hat{\cC}_{\infty}^{i,X \qu G} } $
implies that there exists an isomorphism,
$$ (u^0 : \hat{\cC}^{0} \to P_\infty(X) , \ul{z}^0 ) \cong (u^1 :
\hat{\cC}^{1} \to P_\infty(X), \ul{z}^1 ) .$$
Since the scaled curves appearing in the limit already agree, this
implies that the stable scaled maps 
$$(\hat{C}^i \to \ol{S}, P_i \to
\hat{C}^i, u^i: \hat{C}^i \to P_i(X), \delta^i, \ul{z}^i: \ol{S} \to
\hat{C}^{i,n}) , \quad i = 0,1$$ 
are isomorphic.

The existence of a unique limit in the case of a family with infinite
scaling is similar and left to the reader.
\end{proof} 

This proves the valuative criterion for separatedness.  Universal
closure and of finite-type was shown in previous sections.  This
completes the proof of properness of $\ol{\M}_{n,1}^G(C,X,d)$ in
Theorem \ref{main}.  Finally we complete the proof of the properness
of moduli stacks of affine gauged maps.

\begin{corollary} For any $E > 0$,   the union of  moduli stacks of affine gauged maps $\ol{\M}^G_{n,1}(\bA,X,d)$ 
with $(d, c_1^G(\ti{X})) \leq E$ is proper.
\end{corollary} 

\begin{proof} The proof is by an embedding argument.  The stack
  $\ol{\M}^G_{n,1}(\bA,X,d)$ embeds in $\ol{\M}^G_n(C,X,y)$ as
  follows: Given an affine gauged map $(C_0,\delta,\ul{z},u)$ and a
  point $z \in C$ define $\hat{\cC} := (C_0 \sqcup C)/(z_0 \sim z)$
  and extend the map $u$ so that it is constant on the root component
  $\hat{\cC}_0 \cong C$.  Since any closed substack of a proper stack
  is proper, $\ol{\M}^G_{n,1}(\bA,X,d)$ is proper.
\end{proof}

\begin{remark} {\rm (Convex targets)} We conclude by describing 
results for convex targets: For the moduli stack of gauged maps to a
convex variety $X$ defined in Section \ref{convex}, the conclusion of
Theorem \ref{main} also holds.  As explained in Corollary
\ref{disjoint}, Mundet stable maps to $X$ are equivalent to maps to
$\ol{X}$ as long as the linearization $\ti{X}$ is chosen so that the
linearization is obtained from $\ti{X}(l)$ for $l$ sufficiently large.  

Conversely, for any class $d$ which pairs trivially with the class of
the divisor at infinity, affine gauged maps to $\ol{X} \qu G$ and maps
to $X \qu G$ of class $d$ are equivalent: the intersection number
between the map and divisor is zero and any such map cannot have a
component mapping entirely to the divisor.  In the case of maps to the
quotient, any point in $\ol{X}$ is unstable for $\ti{X}(l)$ for $l$
sufficiently large, and so $\ol{X} \qu G$ is isomorphic to $X \qu G$.
It follows that the inclusion $\ol{\M}_{n,1}^G(C,X) \to
\ol{\M}_{n,1}^G(C,\ol{X})$ is an isomorphism, as claimed, and in
particular the union of components $\ol{\M}_{n,1}^G(C,X,d)$ with
$\mE(d) < E$ is proper for any energy bound $E > 0$.
\end{remark}


\section{Table of notation} 
\label{table} 

This section contains a table of the notations for the different moduli stacks 
of stable maps to quotient stacks used in the paper. 

\begin{center} 
\begin{tabular}{l|l|l} 
Notation & Moduli stack & Page Number \\
\hline 
 $\ol{\M}_{g,n}(X)$ & Stable
maps of genus $g$ with $n$ markings & \pageref{mgn}
\\ 
$\ol{\M}^G_n(C,X)$ & Mundet-semistable gauged maps with $n$
markings & \pageref{mss} \\ $ \ol{\M}_n(C, X \qu G)$ & Stable sections of
$C \times X \qu G \to C$  & \pageref{graphs}
\\ $\ol{\M}^{G,\quot,\lev}_n(C,X)$ & Gauged maps with level structure
& \pageref{level} \\ $\ol{\M}^{G,\quot}_n(C,X)$ & Quot-scheme
compactification of gauged maps & \pageref{quots}
\\ $\ol{\M}^G_{n,1}(\bA,X)$ & Scaled affine gauged maps &
\pageref{affinegauged} \\ $\ol{\M}^G_{n,1}(C,X)$ & Scaled gauged maps
with domain $C$, $n$ markings & \pageref{scaledgauged} \\
\end{tabular}
\end{center}

\def\cprime{$'$} \def\cprime{$'$} \def\cprime{$'$} \def\cprime{$'$}
\def\cprime{$'$} \def\cprime{$'$}
\def\polhk#1{\setbox0=\hbox{#1}{\ooalign{\hidewidth
      \lower1.5ex\hbox{`}\hidewidth\crcr\unhbox0}}} \def\cprime{$'$}
\def\cprime{$'$}


\begin{thebibliography}{10}

\bibitem{abramovich:compactifying}
D.~Abramovich and A.~Vistoli.
\newblock Compactifying the space of stable maps.
\newblock {\em J. Amer. Math. Soc.}, 15(1):27--75 (electronic), 2002.


\bibitem{agv:gw}
D.~Abramovich, T.~Graber, and A.~Vistoli.
\newblock Gromov-{W}itten theory of {D}eligne-{M}umford stacks.
\newblock {\em Amer. J. Math.}, 130(5):1337--1398, 2008.

\bibitem{aov:twisted}
D.~Abramovich, M.~Olsson, and A.~Vistoli.
\newblock Twisted stable maps to tame {A}rtin stacks.
\newblock {\em J. Algebraic Geom.}, 20(3):399--477, 2011.

\bibitem{ar:alg2}
E.~Arbarello, M.~Cornalba, and P.~A. Griffiths.
\newblock {\em Geometry of algebraic curves. {V}olume {II}}, volume 268 of {\em
  Grundlehren der Mathematischen Wissenschaften [Fundamental Principles of
  Mathematical Sciences]}.
\newblock Springer, Heidelberg, 2011.
\newblock With a contribution by Joseph Daniel Harris.

\bibitem{be:gw}
K.~Behrend.
\newblock Gromov-{W}itten invariants in algebraic geometry.
\newblock {\em Invent. Math.}, 127(3):601--617, 1997.

\bibitem{bf:in}
K.~Behrend and B.~Fantechi.
\newblock The intrinsic normal cone.
\newblock {\em Invent. Math.}, 128(1):45--88, 1997.

\bibitem{bm:gw}
K.~Behrend and Yu. Manin.
\newblock Stacks of stable maps and {G}romov-{W}itten invariants.
\newblock {\em Duke Math. J.}, 85(1):1--60, 1996.


\bibitem{bhatt} B. Bhatt.  A non-algebraic hom-stack. Note.
  \href{http://www-personal.umich.edu/~bhattb/math/hom-stack-example.pdf}{http://www-personal.umich.edu/$\sim$bhattb/math/hom-stack-example.pdf}





\bibitem{bb} A. Bialynicki-Birula.  \newblock On Homogeneous Affine
  Spaces of Linear Algebraic Groups.  \newblock {\em American Journal
    of Mathematics}, 85: 577--582

\bibitem{bo:lag}
A.~Borel.
\newblock {\em Linear algebraic groups}.
\newblock Springer-Verlag, New York, second edition, 1991.

\bibitem{cr:orb}
W.~Chen and Y.~Ruan.
\newblock Orbifold {G}romov-{W}itten theory.
\newblock In {\em Orbifolds in mathematics and physics ({M}adison, {WI},
  2001)}, volume 310 of {\em Contemp. Math.}, pages 25--85. Amer. Math. Soc.,
  Providence, RI, 2002.

\bibitem{behrend:review} K.~Behrend.  Introduction to Algebraic
  stacks.  \newblock In {\em Moduli Spaces}, London Mathematical
  Society Lecture Notes 441.  \newblock Edited by Leticia Brambila Paz
  and Peter Newstead.  2014.


\bibitem{cf:st}
I.~Ciocan-Fontanine, B.~Kim, and D.~Maulik.
\newblock Stable quasimaps to {GIT} quotients.
\newblock {\em J. Geom. Phys.}, 75:17--47, 2014.


\bibitem{ciollot} 
J.-L. Colliot-Th\'el\`ene and J.-J. Sansuc. 
\newblock Fibr ́es quadratiques et composantes connexes r ́eelles.
\newblock {\em Math. Ann.}, 244(2):105--134, 1979.

\bibitem{conrad:kl}
B.~Conrad. 
\newblock  The {K}eel-{M}ori theorem via stacks. 
\newblock \href{http://math.stanford.edu/~conrad/papers/coarsespace.pdf}{http://math.stanford.edu/$\sim$conrad/papers/coarsespace.pdf}.


\bibitem{dejong:stacks}
J.~de~Jong et al.
\newblock \href{http://math.columbia.edu/algebraic\_geometry/stacks-git/}{The {S}tacks {P}roject.}


\bibitem{dm:irr}
P.~Deligne and D.~Mumford.
\newblock The irreducibility of the space of curves of given genus.
\newblock {\em Inst. Hautes \'Etudes Sci. Publ. Math.}, (36):75--109, 1969.

\bibitem{ds:red}
V. G. Drinfeld and C. Simpson. 
\newblock B-structures on G-bundles and local triviality.
\newblock {\em Math. Res. Lett.}, 2(6):823--829, 1995.


	
\bibitem{glsm}
The Moduli Space in the Gauged Linear Sigma Model.
\newblock H.~Fan, T.~Jarvis and Y.~Ruan. 
\newblock \href{http://www.arxiv.org/abs/1603.02666}{arXiv:1603.02666}.


\bibitem{toll:gw1} E.~Frenkel, C.~Teleman, A.~J.~Tolland.  \newblock
  Gromov-{W}itten gauge theory {I}.  \newblock arxiv:0904.4834.


\bibitem{fu:st}
W.~Fulton and R.~Pandharipande.
\newblock Notes on stable maps and quantum cohomology.
\newblock In {\em Algebraic geometry---Santa Cruz 1995}, pages 45--96. Amer.
  Math. Soc., Providence, RI, 1997.

\bibitem{ga:gw}
A.~R.-P.~Gaio and D.~A. Salamon.
\newblock Gromov-{W}itten invariants of symplectic quotients and adiabatic
  limits.
\newblock {\em J. Symplectic Geom.}, 3(1):55--159, 2005.


\bibitem{gi:eq}
A.~B. Givental.
\newblock Equivariant {G}romov-{W}itten invariants.
\newblock {\em Internat. Math. Res. Notices}, (13):613--663, 1996.

\bibitem{cross}
E.~Gonz\'alez and C.~Woodward.
\newblock {Q}uantum {W}itten localization and abelianization for qde solutions.
\newblock 
\href{http://www.arxiv.org/abs/0811.3358}
{arxiv:0811.3358}.

\bibitem{har:ref}
R.~Hartshorne.
\newblock Stable reflexive sheaves.
\newblock {\em Math. Ann.}, 254(2):121--176, 1980.

\bibitem{hl:mod}
D.~Huybrechts and M.~Lehn.
\newblock {\em The geometry of moduli spaces of sheaves}.
\newblock Friedr. Vieweg \& Sohn, Braunschweig, 1997.

\bibitem{hess:strat}
Wim~H. Hesselink.
\newblock Desingularizations of varieties of nullforms.
\newblock {\em Invent. Math.}, 55(2):141--163, 1979.

\bibitem{km:quot}
S.~Keel and S.~Mori.
\newblock Quotients by groupoids.
\newblock {\em Ann. of Math. (2)}, 145(1):193--213, 1997.

\bibitem{ki:coh}
F.~C. Kirwan.
\newblock {\em Cohomology of Quotients in Symplectic and Algebraic Geometry},
  volume~31 of {\em Mathematical Notes}.
\newblock Princeton Univ. Press, Princeton, 1984.

\bibitem{ko:lo}
M.~Kontsevich.
\newblock Enumeration of rational curves via torus actions.
\newblock In {\em The moduli space of curves (Texel Island, 1994)}, pages
  335--368. Birkh\"auser Boston, Boston, MA, 1995.

\bibitem{le:sy2}
E.~Lerman.
\newblock Symplectic cuts.
\newblock {\em Math. Res. Letters}, 2:247--258, 1995.


\bibitem{lly:mp1}
B.~H. Lian, K.~Liu, and S.-T.~Yau.
\newblock Mirror principle. {I}.
\newblock {\em Asian J. Math.}, 1(4):729--763, 1997.


\bibitem{lieblich:rem}
M.~Lieblich.
\newblock Remarks on the stack of coherent algebras.
\newblock {\em Int. Math. Res. Not.}, pages Art. ID 75273, 12, 2006.

\bibitem{luna:slice} D.~Luna.  \newblock Slices \'etales.  \newblock
  In {\em Sur les groupes algébriques}, Bull. Soc. Math. France,
  Paris, Mémoire 33, Société Mathématique de France, pp. 81--105.


\bibitem{mau:mult}
S.~Ma'u and C.~Woodward.
\newblock Geometric realizations of the multiplihedra.
\newblock {\em Compos. Math.}, 146(4):1002--1028, 2010.

\bibitem{mats}
Y.~ Matsushima.
\newblock Espaces homog´enes de Stein des groupes de Lie complexes.
\newblock {\em Nagoya Math J.} 16:205--218, 1960.

\bibitem{la:ch}
G.~Laumon and L.~Moret-Bailly.
\newblock {\em Champs alg\'ebriques}, volume~39 of {\em Ergebnisse der
  Mathematik und ihrer Grenzgebiete. 3. Folge. A Series of Modern Surveys in
  Mathematics [Results in Mathematics and Related Areas. 3rd Series. A Series
  of Modern Surveys in Mathematics]}.
\newblock Springer-Verlag, Berlin, 2000.

\bibitem{mp:si}
D.~R. Morrison and M.~R.~Plesser.
\newblock Summing the instantons: quantum cohomology and mirror symmetry in
  toric varieties.
\newblock {\em Nuclear Phys. B}, 440(1-2):279--354, 1995.

\bibitem{mu:ge}
D.~Mumford, J.~Fogarty, and F.~Kirwan.
\newblock {\em Geometric Invariant Theory}, volume~34 of {\em Ergebnisse der
  Mathematik und ihrer Grenzgebiete, 2. Folge}.
\newblock Springer-Verlag, Berlin-Heidelberg-New York, third edition, 1994.

\bibitem{mund:corr}
I.~Mundet~i Riera.
\newblock A {H}itchin-{K}obayashi correspondence for {K}\"ahler fibrations.
\newblock {\em J. Reine Angew. Math.}, 528:41--80, 2000.

\bibitem{ne:st}
L.~Ness.
\newblock A stratification of the null cone via the moment map.
\newblock {\em Amer. J. Math.}, 106(6):1281--1329, 1984.
\newblock with an appendix by D. Mumford.

\bibitem{ol:logtwist}
M.~C. Olsson.
\newblock ({L}og) twisted curves.
\newblock {\em Compos. Math.}, 143(2):476--494, 2007.

\bibitem{ott:remov}
A.~Ott.
\newblock Removal of singularities and {G}romov compactness for symplectic
  vortices.
\newblock {\em J. Symplectic Geom.} 12:257--311, 2014.

\bibitem{po:stable}
M.~Popa and M.~Roth.
\newblock Stable maps and {Q}uot schemes.
\newblock {\em Invent. Math.}, 152(3):625--663, 2003.

\bibitem{ra:th}
A.~Ramanathan.
\newblock Moduli for principal bundles over algebraic curves. {I}.
\newblock {\em Proc. Indian Acad. Sci. Math. Sci.}, 106(3):301--328, 1996.


\bibitem{schmitt:univ}
A.~Schmitt.
\newblock A universal construction for moduli spaces of decorated vector
  bundles over curves.
\newblock {\em Transform. Groups}, 9(2):167--209, 2004.

\bibitem{schmitt:git}
A.~Schmitt.
\newblock {\em Geometric invariant theory and decorated principal bundles}.
\newblock Zurich Lectures in Advanced Mathematics. European Mathematical
  Society (EMS), Z\"urich, 2008.

\bibitem{st:hs}
J.~Stasheff.
\newblock {\em {$H$}-spaces from a homotopy point of view}.
\newblock Lecture Notes in Mathematics, Vol. 161. Springer-Verlag, Berlin,
  1970.


\bibitem{te:qu}
C.~Teleman.
\newblock The quantization conjecture revisited.
\newblock {\em Ann. of Math. (2)}, 152(1):1--43, 2000.

\bibitem{th:fl}
M.~Thaddeus.
\newblock Geometric invariant theory and flips.
\newblock {\em J. Amer. Math. Soc.}, 9(3):691--723, 1996.



\bibitem{tianxu}
G.~Tian and G.~ Xu.
\newblock Virtual cycles of gauged Witten equation.
  \href{http://arxiv.org/abs/1602.07638}{arXiv:1602.07638}



\bibitem{venuwood:class}
S.~{Venugopalan} and C.~Woodward.
\newblock Classification of affine vortices.
\newblock \href{http://www.arxiv.org/abs/1301.7052}{arXiv:1301.7052}.
\newblock To appear in {\em Duke Math. Journal}.


\bibitem{qk1} Chris T.  Woodward.  Quantum Kirwan morphism and Gromov-Witten
  invariants of quotients III.  Transformation Groups 1--39.  First
  online: 21 September 2015.
  \href{http://arxiv.org/abs/1408.5869}{arXiv:1408.5869}.

\bibitem{qk2} Chris T. Woodward.  Quantum Kirwan morphism and Gromov-Witten
  invariants of quotients II.  Transformation Groups 20 (2015)
  881--920.  \href{http://arxiv.org/abs/1408.5864} {arXiv:1408.5864}.


\bibitem{qk3} Chris T. Woodward.  Quantum Kirwan morphism and
  Gromov-Witten invariants of quotients I.  Transformation Groups 20
  (2015) 507--556.
  \href{http://arxiv.org/abs/1204.1765}{arXiv:1204.1765}.




\bibitem{xu:compact}
Guanbo Xu. 
	Moduli space of twisted holomorphic maps with Lagrangian boundary condition: compactness
  \href{http://arxiv.org/abs/1202.4096}{arXiv:1202.4096}.


\bibitem{zilt:qk}
F.~Ziltener.
\newblock {\em A Quantum Kirwan Map: Bubbling and Fredholm Theory for
  Symplectic Vortices over the Plane}, volume 230 of {\em
  Mem.~Amer.~Math.~Soc.}  no. 1082, 2014.

\end{thebibliography}
\end{document}